\documentclass[11pt]{amsart}

\usepackage{latexsym}
\usepackage{amstext,amsmath}
\usepackage{array}
\usepackage{amssymb}
\usepackage{enumitem}
\usepackage[pdftex, bookmarks=true, linkbordercolor={0 1 0}]{hyperref}
\usepackage{a4wide}

\newcommand{\R}{\mathbb R}
\newcommand{\Q}{\mathbb Q}
\newcommand{\C}{\mathbb C}
\newcommand{\Z}{\mathbb Z}
\newcommand{\N}{\mathbb N}
\newcommand{\A}{\mathbb A}

\newcommand{\eps}{\varepsilon}
\newcommand{\minus}{\backslash}

\newcommand{\Hom}{\mathop{\rm Hom}\nolimits}

\newcommand{\tr}{\mathop{\rm tr}\nolimits}
\newcommand{\vol}{\mathop{\rm vol}\nolimits}
\newcommand{\GL}{\mathop{\rm GL}\nolimits}

\newcommand{\Sp}{\mathop{\rm Sp}\nolimits}

\newcommand{\Orth}{{\rm O}}
\newcommand{\Unit}{{\rm U}}
\newcommand{\SL}{\mathop{\rm SL}\nolimits}

\newcommand{\discr}{\mathop{\rm discr}\nolimits}
\newcommand{\res}{\mathop{\rm res}}

\newcommand{\diag}{\mathop{\rm diag}}

\newcommand{\supp}{\mathop{\rm supp}}
\newcommand{\Mat}{\mathop{\rm Mat}\nolimits}
\newcommand{\val}{\mathop{\rm val}\nolimits}

\newcommand{\cpt}{{\bf K}}
\newcommand{\Ad}{\mathop{{\rm Ad}}\nolimits}

\newcommand{\Wt}{\mathop{{\rm Wt}}\nolimits}

\newcommand{\aaa}{\mathfrak{a}}
\newcommand{\bbb}{\mathfrak{b}}
\newcommand{\ooo}{\mathfrak{o}}
\newcommand{\ppp}{\mathfrak{p}}
\newcommand{\uuu}{\mathfrak{u}}
\newcommand{\mmm}{\mathfrak{m}}
\newcommand{\ggG}{\mathfrak{g}}

\newcommand{\OOO}{\mathcal{O}}
\newcommand{\UUU}{\mathcal{U}}
\newcommand{\LLL}{\mathcal{L}}
\newcommand{\FFF}{\mathcal{F}}
\newcommand{\III}{\mathcal{I}}
\newcommand{\PPP}{\mathcal{P}}
\newcommand{\MMM}{\mathcal{M}}
\newcommand{\NNN}{\mathcal{N}}
\newcommand{\BBB}{\mathcal{B}}
\newcommand{\CCC}{\mathcal{C}}
\newcommand{\DDD}{\mathcal{D}}

\newcommand{\AAA}{\mathcal{A}}
\newcommand{\SSS}{\mathcal{S}}

\newcommand{\VVV}{\mathcal{V}}
\newcommand{\XXX}{\mathfrak{X}}
\newcommand{\Ufrak}{\mathfrak{U}}

\newcommand{\One}{\mathbf{1}}
\newcommand{\Mbf}{\mathbf{M}}
\newcommand{\Ubf}{\mathbf{U}}
\newcommand{\signature}{\mathbf{\underline{r}}}

\theoremstyle{plain}
\newtheorem{theorem}{Theorem}[section]
\newtheorem{lemma}[theorem]{Lemma}
\newtheorem{cor}[theorem]{Corollary}
\newtheorem{proposition}[theorem]{Proposition}
\newtheorem{conjecture}[theorem]{Conjecture}
\theoremstyle{definition}

\newtheorem{definition}[theorem]{Definition}
\newtheorem{rem}[theorem]{Remark}

\begin{document}
\title[Bounds for global coefficients]{Bounds for global coefficients in the fine geometric expansion of Arthur's trace formula for ${\rm GL}(n)$}
\author{Jasmin Matz}
\address{Mathematisches Institut, Rheinische Friedrich-Wilhelms-Universit\"at Bonn, Endenicher Allee ~60, 53115 Bonn, Germany}
\email{matz@math.uni-bonn.de}
\thanks{Research partially supported by grant \#964-107.6/2007 from the German-Israeli Foundation for Scientific Research and Development and by a Golda Meir Postdoctoral Fellowship}
\begin{abstract}
 We give upper bounds for the absolute value of the global coefficients $a^M(\gamma, S)$ appearing in the fine  geometric expansion of Arthur's trace formula for $\GL(n)$. 
\end{abstract}
\maketitle
\tableofcontents

\section{Introduction}
Let $F$ be a number field with ring of adeles $\A_F$ and let $G$ be a reductive  group defined over~$F$. 
 Arthur's trace formula for $G$ is an identity of distributions
\[
J_{\text{geom}}(f)
=\sum_{\ooo\in\OOO} J_{\ooo}(f)
=\sum_{\chi\in\XXX} J_{\chi}(f)
=J_{\text{spec}}(f)
\]
between the so-called geometric and spectral side on some space of test functions $f$ (for example smooth and compactly supported functions on $G(\A_F)^1$). Here $\OOO$ denotes the set of certain equivalence classes which are parametrised by conjugacy classes of semisimple elements in $G(F)$, and $\XXX$ is the set of spectral data for $G(F)$.
In \cite[Theorem 8.1]{Ar86} Arthur obtains the following fine expansion for $J_{\ooo}(f)$: There  exist  coefficients $a^M(\gamma, S)\in\C$ such that
\begin{equation}\label{expansion_geometric_distribution}
J_{\ooo}(f)
=\sum_{M}\frac{\left|W^M\right|}{\left|W^G\right|}\sum_{\gamma} a^M(\gamma, S) J_M^G(\gamma, f)
\end{equation}
 for all $f\in C_c^{\infty}(G(\A_F)^1)$ and all finite sets of places $S$ of $F$ such that $S$ is (in a certain sense) sufficiently large with respect to $\ooo$ and the support of $f$.
 Here $M$ runs over the finite set of Levi subgroups of $G$ containing a fixed minimal Levi subgroup, and $\gamma\in M(F)\cap \ooo$ runs over a system of representatives of a certain equivalence relation. In the case of $G=\GL_n$ this equivalence relation is given by $M(F)$-conjugation and does not depend on $S$ \cite[\S 19]{Ar05}. Further,  $W^G$  denotes the Weyl group of $G$, and the distributions $J_M^G(\gamma, f)$ can be defined as $S$-adic weighted orbital integrals ~\cite{Ar88a}. 

The coefficients $a^M(\gamma, S)$ depend on the normalisation of measures on $G(\A_F)$ and its subgroups. 
Exact formulas for them are known for semisimple $\gamma$ in arbitrary $M\subseteq G$ by \cite[Theorem 8.2]{Ar86}, and in the case of $\GL_2$, $\GL_3$, $\SL_2$, $\SL_3$, and $\Sp_2\subseteq \GL_4$ for arbitrary $M$ and $\gamma$ \cite{JaLa70,Fl82,HoWa13}. In general, however, no such formulas are known.

The purpose of this paper is to give an upper  bound for the absolute value of these coefficients (with respect to some fixed choice of measures) for Levi subgroups of $\GL_n$ and arbitrary $\gamma$.
Such upper bounds are  needed -  among other things - to establish asymptotics for traces of Hecke operators on $\GL_n$ with uniform error term along the lines of \cite{LaMu09}.
We first need to find bounds for unipotent  $\gamma$, since by definition \cite[(8.1)]{Ar86} of the coefficients the general case is reduced to the unipotent one. 

\subsection*{Setup}
To describe our results in more detail, fix an integer $n\geq 1$ and let $G=\GL_n$. 
Let $\LLL$ be the set of all Levi subgroups $M\subseteq G$ over $F$ which contain the minimal Levi subgroup consisting of diagonal matrices. We denote by $\UUU_M$ the variety of unipotent elements in $M$, and by $\Ufrak^M$ the finite set of $M$-conjugacy classes in $\UUU_M$.
Fix a finite set of places  $S$ of $F$ containing all the archimedean places. 
The unipotent elements $\UUU_G\subseteq G$ constitute exactly one equivalence class $\ooo_{\text{unip}}=\UUU_G(F)\in~\OOO$. The distribution associated with $\ooo_{\text{unip}}$ is the unipotent distribution
\[
J_{\text{unip}}=J_{\ooo_{\text{unip}}}:
C_c^{\infty}(G(\A_F)^1)\longrightarrow~\C
\]
 studied in \cite{Ar85}.
By \cite[Theorem 8.1]{Ar85} specialised to $\GL_n$ there are uniquely determined numbers $a^M(\VVV, S)\in\C$ for $\VVV\in\Ufrak^M$ such that
\begin{equation}\label{expansion_unipotent_dist}
J_{\text{unip}}(f)
=\sum_{M\in\LLL}\frac{\left|W^M \right|}{\left|W^G\right|} \sum_{\VVV\in \Ufrak^M} a^M(\VVV, S) J_M^G(\VVV, f) 
\end{equation}
holds for all functions $f\in C_c^{\infty}(G(\A_F)^1)$ of the form $f_S\otimes \One_{\cpt^S}$ with $f_S\in C_c^{\infty}(G(F_S)^1)$ and $\One_{\cpt^S}$ the characteristic function of the maximal compact subgroup in $\prod_{v\not\in S} G(F_v)$. Here $J_M^G(\VVV, f):=J_M^G(u,f)$ for some (and hence any) $u\in \VVV$.
The expansion of the general distribution in \eqref{expansion_geometric_distribution} is a generalisation of this equality.

As mentioned before, the absolute value of the constants depend on a choice of measures for $G(\A_F)$ and its subgroups which will be fixed in \S \ref{subsection_measures}. This can already be seen for $\GL_1$ where there is only one coefficient. This coefficient equals
$a^{\GL_1}(1,S)=\vol(F^{\times}\backslash \A_F^1)$. 

\subsection*{Results}
We prove the following bound on the coefficients associated with the unipotent elements.
\begin{theorem}\label{estimate_for_coeff}
Let $ n, d\in \Z_{\geq1}$. There exist constants $\kappa=\kappa(n,d)\geq0$ and $C=C(n,d)\geq0$ such that for every number field $F$ of degree $[F:\Q]=d$ and absolute discriminant $D_F$ the following holds: 
For any  finite set of places $S$ of $F$ containing  all the archimedean places, all $M\in \LLL$ and all unipotent orbits $\VVV\in\Ufrak^M$, we have
\begin{equation}\label{coeffic_estimate}
 \left|a^M(\VVV, S)\right|
\leq C D_F^{\kappa} 
\sum_{\substack{s_v\in\Z_{\geq0}, v\in S_{\text{fin}}: \\ \sum s_v\leq\eta}}\;
\prod_{v\in S_{\text{fin}}}\bigg|\frac{\zeta_{F, v}^{(s_v)}(1)}{\zeta_{F, v}(1)}\bigg|
\end{equation}
 with respect to the measures described in \S \ref{subsection_measures}. The sum here runs over tuples of integers $s_v\geq0$ for $v\in S_{\text{fin}}$ such that the sum $\sum_{v\in S_{\text{fin}}}s_v$ satisfying $\eta\leq\dim\aaa_0^M$, the semisimple rank of $M$. 
Moreover, $S_{\text{fin}}$ is the set of non-archimedean places contained in $S$ and if $v\in S_{\text{fin}}$, $\zeta_{F,v}$ denotes the local factor of the Dedekind zeta function associated with $F_v$.
\end{theorem}

\begin{rem}
\begin{enumerate}[label=(\roman{*})]
\item
The term $\left|\frac{\zeta_{F,v}^{(s_v)}(1)}{\zeta_{F,v}(1)}\right|$  in \eqref{coeffic_estimate} is of the same order as $\frac{(\log q_v)^{s_v}}{q_v-1}$
 for $q_v$ the cardinality of the residue field of the local field $F_v$. In particular, the sum over the logarithmic derivatives of the zeta functions in \eqref{coeffic_estimate} could be replaced by
\[
\sum_{\substack{s_v\in\Z_{\geq0}, v\in S_{\text{fin}}: \\ \sum s_v\leq\eta}}\;
\prod_{\substack{v\in S_{\text{fin}}: \\ s_v>0}}\frac{(\log q_v)^{s_v}}{q_v-1}.
\]
However, the examples discussed below suggest that it is more canonical to use the logarithmic derivatives of the zeta function for the formulation of the theorem.
 
\item For the examples $G=\GL_2$ and $G=\GL_3$ the logarithmic factor is sharp, cf. \S \ref{section_example_gl2}.
\end{enumerate}
\end{rem}

If one keeps track of all constants occurring in the proof of the theorem, one can extract a polynomial upper bound for $\kappa$ in $n$ and $d$. We did not do this though to make the proofs not more technical than necessary.
It is natural to ask for the minimal possible  $\kappa$  such that \eqref{coeffic_estimate} holds, and the examples in \S \ref{section_example_gl2} suggest that any $\kappa>0$ will do. 
More precisely, we conjecture the following about the actual size of the coefficients.

\begin{conjecture}\label{conj}
 For any $\kappa>0$ and all $n,d\in\Z_{\geq1}$ there exists a constant $C=C(n, d, \kappa)\geq0$ such that
\begin{align}
\label{conj1}
\left|a^M(\VVV, S)\right|
&\leq C D_F^{\kappa} \sum_{\substack{s_v\in\Z_{\geq0}, v\in S_{\text{fin}}: \\ \sum s_v\leq\eta}}\;
\prod_{v\in S_{\text{fin}}}\bigg|\frac{\zeta_{F, v}^{(s_v)}(1)}{\zeta_{F, v}(1)}\bigg|,\;\;\;\text{ and }\\
\label{conj2}
\bigg|\frac{a^M(\VVV, S)}{a^{\Mbf_{M,\VVV}}(\One^{\Mbf_{M,\VVV}}, S)}\bigg|
&\leq C D_F^{\kappa} \sum_{\substack{s_v\in\Z_{\geq0}, v\in S_{\text{fin}}: \\ \sum s_v\leq \eta}}\;
\prod_{v\in S_{\text{fin}}}\bigg|\frac{\zeta_{F, v}^{(s_v)}(1)}{\zeta_{F, v}(1)}\bigg|
\end{align}
for any $M\in \LLL$, $\VVV\in\Ufrak^M$, all  number fields $F$ of degree $[F:\Q]=d$, and any finite set of places $S$ of $F$ provided that $S$ contains all archimedean places of $F$. 
Here for a unipotent class $\VVV\in\Ufrak^M$ the Levi subgroup $\Mbf_{M,\VVV}\subseteq M$ is chosen such that $\VVV\subseteq M$ is the conjugacy class induced from the trivial conjugacy class $\One^{\Mbf_{M,\VVV}}$ in $\Mbf_{M,\VVV}$.
\end{conjecture}

Note that unipotent classes in $\GL_n$ are always Richardson classes so that it is always possible to find such a $\Mbf_{M, \VVV}$, see also \S \ref{section_unip_orb_int} for further details.
The denominator $a^{\Mbf_{M,\VVV}}(\One^{\Mbf_{M,\VVV}}, S)$ on the left hand side of \eqref{conj2} equals by \cite[Corollary 8.5]{Ar85} the volume of the quotient $\Mbf_{M,\VVV}(F)\backslash \Mbf_{M,\VVV}(\A_F)^1$ and is in particular independent of the set ~$S$. It is conceivable that the quotient on the left hand side of \eqref{conj2} is independent of the choice of global measure on the various groups involved but only depends on the local measures.
We will give some further comments regarding Conjecture \ref{conj} below.

Suppose now that $\gamma\in M(F)$ is arbitrary. 
The coefficients $a^M(\gamma, S)$ are defined in terms of coefficients $a^{H}(u, S)$ for $H\subseteq M$ certain reductive subgroups and $u\in \UUU_H(F)$ unipotent (see \cite[(8.1)]{Ar86} and also \S \ref{section_arb_coeff}). 
From our main result we will deduce the following bound for general coefficients in \S \ref{section_arb_coeff}.

\begin{cor}\label{cor_bounds_gen_coeff}
For every $n, d\in\Z_{\geq1}$ there exist $\kappa=\kappa(n, d)\geq 0$ and $C=C(n, d)\geq0$ such that the following holds. Let $F$ and $S$ be as in Theorem \ref{estimate_for_coeff}.
 Let $M\in\LLL$ and $\gamma\in M(F)$, and write $\gamma_s\in M(F)$ for the semisimple part of $\gamma$ in its Jordan decomposition. Suppose that the eigenvalues of $\gamma_s$ (in some algebraic closure of $\Q$) are algebraic integers. Further, let $M_1(F)\subseteq M(F)$ be the unique Levi subgroup such that $\gamma_s\in M_1(F)$ is regular elliptic in $M_1(F)$. Then, if $\gamma_s$ is elliptic in $M(F)$,
\begin{equation}\label{bound_for_gen_coeff}
|a^M(\gamma, S)|
	      \leq C |\discr^{M_1}(\gamma_s)|_{\infty}^{\kappa}
	      \sum_{\substack{s_v\in\Z_{\geq0}, v\in S_{\text{fin}}: \\ \sum s_v\leq \eta}}\;
	      \prod_{v\in S_{\text{fin}}}\bigg|\frac{\zeta_{F, v}^{(s_v)}(1)}{\zeta_{F, v}(1)}\bigg|	
\end{equation}
with respect to the measures defined in \S \ref{subsection_measures}, and $a^M(\gamma, S)=0$ if $\gamma_s$ is not elliptic in $M(F)$.
Here $|\discr^{M_1}(\gamma_s)|_{\infty}$ is the norm of the discriminant of $\gamma_s$ in $M_1(F)$ as an element of $F$ over $\Q$, and $\eta=\dim\aaa_{M_{1,\gamma_s}}^{M_{\gamma_s}}$, where $M_{\gamma_s}$ (resp.\ $M_{1, \gamma_s}$) denotes centraliser of $\gamma_s$ in $M$ (resp.\ $M_1$).
\end{cor}

\begin{rem}
\begin{enumerate}[label=(\roman{*})]
\item 
The discriminant $\left|\discr^{M_1}(\gamma_s)\right|_{\infty}$ depends only on the equivalence class $\ooo\in\OOO$ in which $\gamma$ is contained, but not on the specific representative $\gamma\in M(F)\cap\ooo$.

\item 
The equality \eqref{expansion_geometric_distribution} only holds if the set $S$ is sufficiently large with respect to $\ooo$ in the sense of \cite[p. 203]{Ar86}, but the coefficients $a^M(\gamma, S)$ are well-defined for any finite set $S$ containing the archimedean places. 

\item
In view of our anticipated application to the Weyl law for Hecke operators, the most important property of the bound \eqref{bound_for_gen_coeff} is the explicit dependence on the set $S$ and the discriminant of $\gamma$.

\item
For an analogue of Conjecture \ref{conj} for arbitrary coefficients, see Conjecture \ref{conj_arb_coeff}.
\end{enumerate}
\end{rem}

\subsection*{Further remarks}
\begin{itemize}
\item
In all computed examples (cf. \S \ref{section_example_gl2}), the term $D_F^{\kappa}$ on the right hand side of \eqref{coeffic_estimate} comes from bounds on certain logarithmic derivatives of  Dedekind zeta functions and residues of such zeta functions.
In particular, the Euler-Kronecker constant associated to $F$ needs to be estimated (cf. \cite{Ih06}). 
In view of this, it is conceivable to expect that the term $D_F^{\kappa}$ can in fact be replaced by $(\log D_F)^{k}$ for some suitable number $k>0$.

\item
If we consider the trivial conjugacy class $\One^M$ in some Levi subgroup $M\in \LLL$, then the associated Levi subgroup $\Mbf_{M,\One^M}$ is contained in the Weyl group orbit of $M$. Therefore, \eqref{conj2} of Conjecture \ref{conj} is trivially true for the trivial conjugacy class in any Levi subgroup because of
$a^{\Mbf_{M,\VVV}}(\One^{\Mbf_{M,\VVV}},S)=a^M(\One^M, S)$.
Since the coefficient for the trivial class is always given by a volume \cite[Corollary 8.5]{Ar85}, \eqref{conj1} of Conjecture \ref{conj} can in this case be deduced from upper bounds on the residue of the Dedekind zeta function of $F$ as given in Proposition \ref{brauer_siegel}.

\item
\eqref{conj1} of Conjecture \ref{conj} also holds  for all Levi subgroups and unipotent conjugacy classes in $\GL_2$ and $\GL_3$ as shown in \S \ref{section_example_gl2}. We will recall in \S \ref{section_example_gl2} the exact formulas for the coefficients in both cases. 
However, all mentioned examples suggest that the second inequality \eqref{conj2} is of the more natural form. 

\item
As the denominator $a^{\Mbf_{M,\VVV}}(\One^{\Mbf_{M,\VVV}}, S)$ on the left hand side of \eqref{conj2} is just a certain volume, by the choice of our measure both parts of the conjecture are equivalent if the lower bound of the Brauer-Siegel Theorem holds for $F$ (for example, if $F$ is a normal extension of $\Q$, or if we assume GRH).
However, there is also a more structural reason, why $a^{\Mbf_{M,\VVV}}(\One^{\Mbf_{M,\VVV}}, S)$ should appear as the ``main'' part of $a^M(\VVV, S)$:
In the cases where an exact formula for the coefficients is known, these coefficients are  given in terms of derivatives of  certain zeta functions associated with the unipotent orbits.
This should be possible in general, suggesting that there are indeed terms of the form 
$\sum_{s_v}\prod_{v}\Big|\frac{\zeta_{F, v}^{(s_v)}(1)}{\zeta_{F, v}(1)}\Big|$,
 but also that $a^{\Mbf_{M,\VVV}}(\One^{\Mbf_{M,\VVV}}, S)$ should  occur naturally in an exact formula for $a^M(\VVV, S)$, cf. also \S \ref{section_example_gl2}.

\item
There are various points in our proof which make use of particular properties of $\GL_n$ and do not easily carry over to other reductive groups. For example, we use that $\GL_n$ is split over $\Q$ which allows us to bound certain weight functions in a uniform way. 
A more serious complication might arise from the fact that the equivalence relation on $\UUU_M(F)$ (or, more generally $M(F)\cap \ooo$) in \eqref{expansion_geometric_distribution} is in general not only given by conjugacy, but may also depend on $S$. In particular, the number of equivalence classes might grow with $S$ requiring to choose special test functions at all places in $S$ instead of only at the archimedean places as suffices in our situation. This would complicate the analysis of the local distributions $J_M^G(\VVV,f)$ in \S \ref{section_weighted_orb_int} considerably.
\end{itemize}

\subsection*{Outline}
Our proof of Theorem \ref{estimate_for_coeff} will be by induction on the semisimple rank $\dim\aaa_0^M$ of $M$. 
For the initial case we have $M=T_0$, $\Ufrak^{T_0}=\{\One^{T_0}\}$, and hence an exact formula for $a^{T_0}(\One^{T_0}, S)$ by \cite[Corollary 8.5]{Ar85}) giving our desired bound.
For the induction step we will use \eqref{expansion_unipotent_dist}  with respect to $M$ instead of $G$ and write it as
\begin{equation}\label{reduction_to_ind_step}
 \sum_{\VVV\in\Ufrak^M}a^M(\VVV, S) J_M^M(\VVV, f)
=J_{\text{unip}}^M(f)
-\sum_{ L\subsetneq M}\frac{|W^L|}{|W^M|}\sum_{\VVV\in\Ufrak^L} a^L(\VVV, S)J_L^M(\VVV, f).
\end{equation}
To get estimates on the coefficients on the left hand side we first need to control the different terms on the right hand side and make their dependence on $F$ explicit.
Finally we plug in special test functions $f=f_{\VVV}$ which separate the distributions $J_M^M(\VVV, \cdot)$, and therefore the coefficients $a^M(\VVV, S)$ for the different classes $\VVV\in \Ufrak^M$.

A more careful choice and analysis of these special  test functions might lead to a better upper bound on the contribution of $J_L^M(\III_L^M\VVV, f)$ to the exponent $\kappa$ in \eqref{coeffic_estimate}. However, it is  doubtful that with our methods one can get Conjecture \ref{conj}. The main reason is that in order to estimate the global distributions $J_{\text{unip}}^M(f)$, we bound integrals over $M(F)\backslash M(\A_F)^1$ by integrals over large compact sets which inevitably leads to the addition of non-trivial powers of $D_F$ on the right hand side of \eqref{reduction_to_ind_step}. 

The organisation of the paper is as follows: After fixing some notation in \S \ref{section_notation}, we shall recall some results from number theory in \S \ref{section_number_theory} which will be needed later. In \S \ref{section_unip_orb_int} we recall a few facts about unipotent conjugacy classes, define the weighted orbital integrals $J_M^G(\VVV, f)$, and fix suitable test functions for the separation of the coefficients. After that, we give an upper bound for the weighted orbital integrals in \S \ref{section_weighted_orb_int}. In sections  \S \ref{section_reduction_theory} - \S \ref{section_unip_contr} we prove an upper bound for the global unipotent contribution $|J_{\text{unip}}(f)|$ and finish the proof of our main result. For this we first need to make reduction theory for $\GL_n$ over a number field $F$ explicit in the sense of Proposition \ref{prop_red_theory_over_number_field}, and then approximate the unipotent contribution by related distributions studied in \cite{Ar85}.
Finally, we study the examples $G=\GL_2$ and $G=\GL_3$ in \S \ref{section_example_gl2}, and prove Corollary \ref{cor_bounds_gen_coeff} in \S \ref{section_arb_coeff}.

\subsection*{Acknowledgments}
I would like to thank Erez Lapid and Tobias Finis for reading a first draft of the paper and many helpful discussions.

\section{Notation}\label{section_notation}
\subsection{Basic notation}
Let $\signature=(r_1,r_2)$ with $r_1+2r_2=d$ denote a signature of degree $d$, and let $\R^{\signature}=\R^{r_1}\oplus\C^{r_2}$. 
If $F$ is a number field of degree $d$ there is a signature $\signature$ of degree $d$ such that $F$ has signature $\signature$, i.e., $F$ has $r_1$ real and $r_2$ pairs of complex embeddings. We will usually keep $d$ and a signature $\signature$ of degree $d$ fixed, and let $F$ vary over number fields of signature $\signature$.

For a number field $F$ let  $\A_F$ be the ring of adeles of $F$ and $D_F$ the absolute discriminant of $F$ over $\Q$. We write $\Delta_F=\sqrt{D_F}$ for the root of the discriminant.
Let $G=\GL_n$ for some fixed integer $n\geq 2$.
Let $S$ be a finite set of places of $F$ containing all the archimedean places. We write  $S_{\text{fin}}$ for the set of non-archimedean places contained in $S$, and $S_{\infty}=S\minus S_{\text{fin}}$ for the set of archimedean places of $F$.

If $v$ is a place of $F$, let $F_v$ denote the completion of $F$ at $v$, and write
$F_{\infty}=\prod_{v\in S_{\infty}} F_v$
and
$F_{S}=\prod_{v\in S} F_v$.
We may identify the spaces $F_{\infty}=F\otimes_{\Q}\R=\R^{r_1}\oplus\C^{r_2}$ with each other for all number fields $F$ of signature $\signature$.
If $v$ is non-archimedean, then  $\OOO_v\subseteq F_v$ denotes the ring of integers, $\varpi_v\in \OOO_v$ a uniformising element, and $q_v$ the number of elements in  the respective residue field. The norm $|\cdot|_v$ on $F_v$ is normalised by $|\varpi_v|_v=q_v^{-1}$.
For $v$ an archimedean place, $|\cdot|_v$ denotes the usual norm obtained from the identification of $F_v$ with $\R$ if $v$ is real, and the square of the usual norm obtained from the identification of $F_v$ with $\C$ if $v$ is complex. 
We write $|\cdot|_{\A_F}=|\cdot|:\A_F^{\times}\longrightarrow\R_{\geq0}$ for the adelic norm, which is given by
$|a|_{\A_F}
=|a|
=\prod_v |a_v|_v$
if $a=(a_v)_v\in\A_F^{\times}$.
If $\aaa\subseteq \OOO_F$ is an ideal, we let 
$\N(\aaa)
=\left|\OOO_F/\aaa\right|$
be the norm of the ideal.
Similarly, if $\aaa\subseteq \OOO_v$ is a local ideal, we also write
$\N(\aaa)
=\big|\OOO_v/\aaa\big|$.

Finally, for two quantities $A,B$ we write $A\ll B$ if there exists $c>0$ such that $A\leq cB$, and $\ll_{\alpha, \ldots}$ if $c$ depends on certain parameters $\alpha, \ldots$. When writing $A\ll_{\alpha, \ldots} B$ we understand that the implied constant $c$ does not depend on any parameter other than $\alpha, \ldots$.

\subsection{Subgroups of ${\rm GL}_n$}
Let $T_0\subseteq G$ be the maximal split torus of diagonal matrices and $P_0=T_0U_0$ the minimal parabolic subgroup of upper triangular matrices with unipotent radical $U_0$ consisting of all unipotent upper triangular matrices. We denote by $\FFF$ the finite set of parabolic subgroups in $G$ over $\Q$ containing $T_0$.
Any $P\in \FFF$  can be uniquely decomposed into $P=M_PU_P$ with $M_P$ its Levi component containing $T_0$ and $U_P$ its unipotent radical. Let $\LLL=\{M_P\mid P\in \FFF\}$ and if $L\in\LLL$, write $\LLL(L)=\{M\in\LLL\mid M\supseteq L\}$ and $\FFF(L)=\{P\in \FFF\mid L\subseteq P\}$. 
Note that any $M\in\LLL$ is isomorphic to a group $\GL_{n_1}\times\ldots\times\GL_{n_{r+1}}\hookrightarrow\GL_n$ for a suitable partition $n_1, \ldots, n_{r+1}$ of $n$, $n_1+\ldots+n_{r+1}=n$.
We say that $P\in\FFF$ is standard if $P_0\subseteq P$ and write $\FFF_{\text{std}}\subseteq \FFF$ for the subset of all standard parabolic subgroups.
For $M\in\LLL$ let $\PPP(M)=\{P\in\FFF\mid M_P=M\}$, and we call $M$ standard if $M$ is the Levi component of a standard parabolic subgroup.

We choose maximal compact subgroups as follows: If $v$ is a place of $F$, 
\[
\cpt_v
=\begin{cases}
  G(\OOO_v)					&\text{if } v \text{ is non-archimedean,}\\
\Orth(n)					&\text{if } v \text{ is real,}\\
\Unit(n)					&\text{if } v \text{ is complex.}
 \end{cases}
\]
are the usual maximal compact subgroups of $G(F_v)$ which are hyperspecial if $v$ is non-archimedean.
Then $\cpt=\prod_v\cpt_v$ is the standard maximal compact subgroup of $G(\A_F)$. 
We shall also write 
$
\cpt_{S}=\prod_{v\in S}\cpt_v
$,
$
\cpt^S=\prod_{v\not\in S} \cpt_v
$,
$
\cpt_{\text{fin}}=\prod_{v\not\in S_{\infty}}\cpt_v
$,
and
$
\cpt_{\infty}=\cpt_{\signature}
=\prod_{v\in S_{\infty}} \cpt_v
=\Orth(n)^{r_1}\times \Unit(n)^{r_2}
$.
This last group is independent of $F$ and only depends on the signature $\signature$.

\subsection{Root systems}
For $P\in \FFF$ let $A_P^{\infty}=A_{M_P}^{\infty}$ be the identity component of the centre of $M_P(\R)$ and write $A_0^{\infty}=A_{P_0}^{\infty}$.
Then $A_0^{\infty}\simeq \R_{>0}^{n}$, and we embed $A_0^{\infty}\hookrightarrow T(\A_F)$ by sending $(t_1, \ldots, t_n)\in \R_{>0}^n$ to $\diag(t_1,  \ldots, t_n)\in T(\A_F) $, where $t_i\in\A_F^{\times}$ also denotes the idele given by
\[
(\underbrace{t_i^{1/d}, \ldots, t_i^{1/d}}_{\text{arch. places}},\underbrace{ 1,1,\ldots}_{\text{finite places}})\in\A_F^{\times}
\]
(so that the adelic norm of $t_i$ as an element of $\A_F^{\times}$ equals $t_i\in\R_{>0}$).
Let $X_F(M_P)$ be the lattice of $F$-defined algebraic characters of $M_P$ and $\aaa_{M_P}=\Hom_F(X_F(M_P), \R)$ (note that $X_F(M_P)=X_{\Q}(M_P)$ and $\Hom_F(X_F(M_P), \R)=\Hom_{\Q}(X_{\Q}(M_P), \R)$ here, since $\GL_n$ is $\Q$-split).
Then $\aaa_0=\aaa_{M_{P_0}}\simeq \R^n$ and we have a canonical group isomorphism $\log: A_0^{\infty}\longrightarrow \aaa_0$ with inverse $\exp\aaa_0\longrightarrow A_0^{\infty}$. 

Let $\Delta_0=\{\alpha_1, \ldots, \alpha_{n-1}\}$ be the set of simple roots of $(P_0, A_0^{\infty})$ enumerated in such a way that if $X=(X_1, \ldots, X_n)\in \aaa_0$, then $\alpha_i(X) =X_i-X_{i+1}$, 
and let $\widehat{\Delta}_0$ be the weights of $(P_0, A_0^{\infty})$. We denote by $\Sigma^+(P_0, A_0^{\infty})$ the set of positive roots of $(P_0, A_0^{\infty})$.
If $\alpha\in\Delta_0$ (resp.\ $\varpi\in\widehat{\Delta}_0$), let $\alpha^{\vee}\in \aaa_0$ (resp. $\varpi^{\vee}\in\aaa_0$) denote the corresponding coroot (resp.\ coweight).
If $P_1\subseteq P_2$ are standard parabolic subgroups, we let $\Delta_{P_1}$ be the set of simple roots of $(P_1, A_{P_1}^{\infty})$ and by $\Sigma^+(P_1,A_{P_1}^{\infty})$ the set of positive roots of $(P_1,A_{P_1}^{\infty})$. Let $\Delta_{P_1}^{P_2}\subseteq \Delta_{P_1}$ denote the subset characterised by the property that its elements vanish when restricted to $\aaa_{P_2}$. Further let $\aaa_{P_1}^{P_2}=\aaa_{M_{P_1}}^{M_{P_2}}$ be the kernel of the map $X_F(M_{P_2})\longrightarrow X_F(M_{P_1})$ sending a character to its restriction to $M_{P_1}$.
As usual, write $\rho_P$ for the half sum over all positive roots of $(P, A_P^{\infty})$ and $\rho=\rho_{P_0}$.
 For $P\in\FFF$ we define $P(\A_F)^1$ to be the intersection of all the kernels of the absolute values $|\chi|$ of the characters $\chi\in X_F(M_P)$ pulled back to $P(\A_F)$ via $\chi(mu)=\chi(m)$, $m\in M_P(\A_F)$, $u\in U_P(\A_F)$, and set $M_P(\A_F)^1=M_P(\A_F)\cap P(\A_F)^1$. In particular, $G(\A_F)^1=\{g\in G(\A_F)\mid |\det g|=1\}\simeq A_G^{\infty}\backslash G(\A_F) $.
 If $P_1\subseteq P_2$ are two standard parabolic subgroups, set
 $
 A_{P_1}^{P_2, \infty}=A_{P_1}^{\infty}\cap P_2(\A_F)^1
 $.
Recall the definition of the map
$
H_P:G(\A_F)^1\longrightarrow \aaa_{M_P}^G
$
which is characterised by 
\[
H_P(um\exp(H)k)=H
\]
for $u\in U(\A_F)$, $m\in M_P(\A_F)^1$, $H\in a_{M_P}^G$, $k\in \cpt$.

\subsection{Measures}\label{subsection_measures}
We need to fix some measures. On the groups $\cpt_v$ and  $\cpt$ we normalise the (left and right) Haar measures by $\vol(\cpt_v)=1=\vol(\cpt)$ for all $v$. 
Let $\psi_v: F_v\longrightarrow\C$ be the additive  character given by 
$
\psi_v(a)
=e^{2\pi i \tr_{F_v/\Q_p}(a)}
$,
where $\tr_{F_v/\Q_p}:F_v\longrightarrow \C$ denotes the trace map for the extension $F_v$ over $\Q_p$ with $p$ the rational prime below $v$  ($p=\infty$, $\Q_{\infty}=\R$, if $v$ is non-archimedean) composed with the canonical embedding $\Q_p/\Z_p\xrightarrow{\sim}\Q/\Z\hookrightarrow\R/\Z$ if $v$ is non-archimedean, cf. \cite[Chapter XIV, \S 1]{La86}.
Further, let
$\psi:\A_F\longrightarrow\C$
be the product $\psi(x)=\prod_v\psi_v(x_v)$, $x=(x_v)_v\in\A_F$. Then $\psi$ is trivial on $F$ so that we in fact get a non-trivial character 
$\psi:F\backslash \A_F\longrightarrow\C$.
We then take the Haar measure on $F_v$  that is self-dual with respect to $\psi_v$. It is the usual Lebesgue measure if $F_v\simeq\R$, twice the usual Lebesgue measure if $F_v\simeq \C$, and gives the normalisation $\vol(\OOO_v)=\N(\DDD_v)^{-\frac{1}{2}}$ if $v$ is non-archimedean where $\DDD_v\subseteq \OOO_v$ denotes the local different.
We fix multiplicative measures on $F_v^{\times}$ by
\[
d^{\times} x_v=
\begin{cases}
\frac{dx_v}{|x_v|_v}												&\text{if } v \text{ is archimedean},\\
\frac{q_v}{q_v-1}\frac{dx_v}{|x_v|_v}										&\text{if } v \text{ is non-archimedean},
\end{cases}
\]
so that $\vol(\OOO_v^{\times})= \N(\DDD_v)^{-\frac{1}{2}}$ in the non-archimedean case.
Globally, we take the product measures $dx=\prod_v dx_v$ and $d^{\times}x=\prod_v d^{\times}x_v$ on $\A_F$ and $\A_F^{\times}$, respectively.
Using the identification $\R_{>0}\ni t\mapsto (t^{1/d}, \ldots, t^{1/d},1,\ldots)\in\A_F^{\times}$ as before,  we get an isomorphism 
$\A_F^{\times}\simeq \R_{>0}\times\A_F^1$
that also fixes a measure $d^{\times}b$ on $\A_F^1$ via
$d^{\times}x= d^{\times} b ~ \frac{d t}{|t|}$
for $d^{\times} x$ the previously defined measure on $\A_F^{\times}$ and $\frac{dt}{|t|}=d^{\times} t$ the usual multiplicative Haar measure on $\R^{\times}$.
With this choice of measures we get
\[
\vol(F\backslash \A_F)=1
~~~~~~~~~~~~~~~~~
\text{ and }
~~~~~~~~~~~~~~~~
\vol(F^{\times}\backslash \A_F^{1})=\res_{s=1}\zeta_F(s)=:\lambda_{-1}^F
\]
where $\zeta_F(s)$ is the Dedekind zeta function of $F$
(cf. \cite[Chapter XIV, \S 7, Proposition 9]{La86} and the next section).
This also fixes measures on $T_0(\A_F)$, $T_0(\A_F)^1$, $T_0(F_v)$, $U(\A_F)$, and  $U(F_v)$ for  $U$ the unipotent radical of any standard parabolic subgroup 
by using the bases given by the coordinate entries of the matrices.

The measure on $G(\A_F)$ (and any of its Levi subgroups) is then defined via the Iwasawa decomposition $G(\A_F)=U_0(\A_F) T_0(\A_F)\cpt$: For any integrable function $f:G(\A_F)\longrightarrow\C$ we require that
\begin{multline*}
\int_{G(\A_F)}f(g)dg
=\int_{\cpt}\int_{T_0(\A_F)}\int_{U_0(\A_F)}\delta_{0}(m)^{-1}f(umk)du~dm~dk\\
=\int_{\cpt}\int_{T_0(\A_F)}\int_{U_0(\A_F)}f(muk)du~dm ~dk
\end{multline*}
(similarly for $G(F_v)=U_0(F_v)T(F_v)\cpt_v$), where  $\delta_0=\delta_{P_0}$ is the modulus function for the adjoint action of $T$ on $U_0$ so that $\delta_0(m)=e^{\langle 2\rho, H_0(m)\rangle}$.
On $G(\A_F)^1$  we define a measure via the exact sequence  
\[
1\longrightarrow G(\A_F)^1\longrightarrow G(\A_F)\xrightarrow{g\mapsto|\det g|}\R_{>0}\longrightarrow 1.
\]
With this choice of measures, we get (cf.\ \cite{La66})
\[
\vol(G(F)\backslash G(\A_F)^1)
=\lambda_{-1}^F\zeta_F(2)\cdot\ldots\cdot\zeta_F(n).
\]

Finally we fix a measure on $\aaa_0$ and all its subspaces: On $\aaa_0$ we take the usual Lebesgue measure induced from the isomorphism $A_0^{\infty}\xrightarrow{\log}\aaa_0$. We fix a euclidean inner product on $\aaa_0$ by sending the pair $X, Y\in\aaa_0$, $X=(X_1, \ldots, X_n)$, $Y=(Y_1, \ldots, Y_n)$ to $\sum_{i=1}^n X_iY_i$. This euclidean structure together with our initial choice of measure on $\aaa_0$ then determines measures on any subspace of $\aaa_0$.

\section{Some necessary number theoretic facts}\label{section_number_theory}
In this section we recall and collect some number theoretic facts that will be used later.

\subsection{Dedekind zeta functions}
We denote by $\zeta_F$ be the Dedekind zeta function of $F$, that is,
$\zeta_F(s)
=\prod_{v\not\in S_{\infty}} (1-q_v^{-s})^{-1}$
if $s\in\C$ with $\Re s>1$,
and write 
\[
\zeta_F(s)
=\lambda_{-1}^F(s-1)^{-1}+\lambda_0^F+ \lambda_1^F(s-1)+\ldots
\]
for its Laurent expansion around $1$.

The residue of $\zeta_F(s)$ at $s=1$ is connected to the class number and regulator of $F$ via the class number formula,
\[
\lambda_{-1}^F=\res_{s=1}\zeta_F(s)
=\frac{2^{r_1}(2\pi)^{r_2} h_F R_F}{w_F \Delta_F}>0,                                                                                                   
\]
where $h_F$ denotes the class number of $F$, $R_F$ its regulator, and $w_F$ the number of roots of unity in $F$ (see \cite[XIV, \S 8, Proposition 13]{La86}). Recall  that we write $\Delta_F=\sqrt{D_F}$.
\begin{proposition}\label{brauer_siegel}
 For all number fields $F$ of degree $d$, all $\eps>0$, and all integers $k\geq -1$  we have
\[
\left|\lambda_{k}^{F}\right|
\ll_{d,\eps,k} D_F^{\eps}.
\]
\end{proposition}

\begin{rem}
 \begin{enumerate}[label=(\roman{*})]
\item The bounds on $\lambda_k^F$ for $k\geq0$ are mainly used to verify parts of Conjecture \eqref{conj} for some examples in \S \ref{section_example_gl2}.
\item If $k=-1$, we have the sharper bound 
\[
\left|\lambda_{-1}^{F}\right|\ll_d (\log D_F)^{d-1}
\]
for all number fields $F$ of degree $d$, cf.\ \cite{Si69}.
See also \cite{Ih06} for better upper  bounds on $\lambda_0^F$ (or rather on the Euler-Kronecker constant $\lambda_0^F/\lambda_{-1}^F$ associated with $F$). 

\item (Brauer-Siegel Theorem) If $F/\Q$ ranges over a sequence of Galois number fields of bounded degree, then we have the even stronger result that
\[
\log (h_FR_F)\sim \log\Delta_F,
\]
see \cite[XVI, \S 4 Theorem 4]{La86}.
\end{enumerate}
\end{rem}

\begin{proof}[Proof of Proposition \ref{brauer_siegel}]
By \cite[Theorem 5.30]{IwKo04}, we have
\begin{equation}\label{est_for_zeta_fct}
|(s-1)\zeta_F(s)|\ll_{d,\eps} |D_Fs^d|^{\frac{1-\sigma}{2}+\eps}
\end{equation}
for all $s\in\C$ with $0\leq \sigma=\Re s\leq 1$.
The case $k=-1$ follows at once by taking the limit $s\nearrow1$, $s\in\R$, in this inequality.
For $k\geq 0$ we use some basic complex analysis.
Define
\[
 {\bf \Gamma}_F(s)
 =\pi^{-\frac{r_1 s}{2}}(2\pi)^{-r_2s}\Gamma(\frac{s}{2})^{r_1}\Gamma(s)^{r_2},
\]
let $\Lambda_F(s):={\bf \Gamma}_F(s)\zeta_F(s)$ be the completed Dedekind zeta function, and put $\tilde{\Lambda}_F(s):=\Gamma(s) \Lambda_F(s)$.
 Then it follows from \eqref{est_for_zeta_fct} that we also have
\begin{equation}\label{est_for_zeta_fct2}
 |(s-1)\Lambda_F(s)|\ll_{d, \eps} |D_F s^d|^{\frac{1-\sigma}{2}+\eps},
\text{ and }
 |(s-1)\tilde{\Lambda}_F(s)|\ll_{d, \eps} |D_F s^d|^{\frac{1-\sigma}{2}+\eps},
\end{equation}
for all $\eps$ and $s$ as before.
Moreover, $\Lambda_F(s)$ is away from its poles bounded in any vertical strip of finite width, and $\tilde{\Lambda}_F(s)$ is of rapid decay at $\infty$ in any vertical strip of finite width because of Stirling's formula. Moreover, $\Lambda_F$ and $\tilde{\Lambda}_F$ have both exactly one pole in $\Re s>0$, namely at $s=1$.

For every $k\geq 0$, Cauchy's integration formula gives for all $t\geq1$ and $0<\sigma_0<1<\sigma_1$ that
\begin{multline*}
 \frac{1}{(k+1)!}\frac{d^{k+1}}{ds^{k+1}}\left[(s-1)\tilde{\Lambda}_F(s)\right]_{s=1}\\
=\frac{1}{2\pi i}
\bigg( \int_{\sigma_0-it}^{\sigma_1-it} + \int_{\sigma_1-it}^{\sigma_1+it} + \int_{\sigma_1+it}^{\sigma_0+it} + \int_{\sigma_0+it}^{\sigma_0-it}\bigg) 
\frac{(s-1)\tilde{\Lambda}_F(s)}{(s-1)^{k+2}} ds.
\end{multline*}
As $\tilde{\Lambda}_F$ is bounded at $\infty$ in every vertical strip of finite width (it is even of rapid decay),  the horizontal boundary terms (i.e., the first and third integral) do not contribute for $t\rightarrow\infty$.
Therefore, 
\begin{multline*}
 \frac{1}{(k+1)!}\frac{d^{k+1}}{ds^{k+1}}\left[(s-1)\tilde{\Lambda}_F(s)\right]_{s=1}\\
 =\frac{1}{2\pi i}\int_{\sigma_1+i\R}\frac{(s-1)\tilde{\Lambda}_F(s)}{(s-1)^{k+2}} ds
 - \frac{1}{2\pi i}\int_{\sigma_0+i\R}\frac{(s-1)\tilde{\Lambda}_F(s)}{(s-1)^{k+2}} ds,
\end{multline*}
since both integrals  on the right hand side converge absolutely.

Note that as $\Re \sigma_1>1$ we have $|\zeta_F(s)|\leq \zeta_F(\sigma_1)\leq \zeta(\sigma_1)^d$, $|\Gamma(s)|\leq \Gamma(\sigma_1)$, and $|\Gamma(s/2)|\leq \Gamma(\sigma_1)$ for every $s\in \sigma_1+i\R$ so that
\[
|\tilde\Lambda_F(s)|\leq |\Gamma(s)| \Gamma(\sigma_1)^d\zeta(\sigma_1)^d.
\]
Hence  if $\sigma_1\geq2$ is fixed, we get
\[
\bigg| \frac{1}{2\pi i}\int_{\sigma_1+i\R}\frac{(s-1)\tilde{\Lambda}_F(s)}{(s-1)^{k+2}} ds\bigg|
\leq \frac{\Gamma(\sigma_1)^d \zeta(\sigma_1)^d }{2\pi}\int_{\R}|\Gamma(\sigma_1+it)|(1+t^2)^{-(k+1)/2}dt.
\]
As $\Gamma(\sigma_1+it)$ decays rapidly for $ |t|\rightarrow\infty$, this last integral converges and is bounded by a constant $c$ which is independent of $F$ , but only depends on $k$ and our (fixed) choice of $\sigma_1$.
Hence
\[
 \bigg| \frac{1}{2\pi i}\int_{\sigma_1+i\R}\frac{(s-1)\tilde{\Lambda}_F(s)}{(s-1)^{k+2}} ds\bigg|
 \leq c \frac{\Gamma(\sigma_1)^d \zeta(\sigma_1)^d }{2\pi}
 \ll_{d, k} 1.
\]
Therefore,
\begin{multline*}
 \bigg|  \frac{1}{(k+1)!}\frac{d^{k+1}}{ds^{k+1}}\left[(s-1)\tilde{\Lambda}_F(s)\right]_{s=1}\bigg|
\ll_{d,  k} 1 +  \bigg| \frac{1}{2\pi i}\int_{\sigma_0+i\R}\frac{(s-1)\tilde{\Lambda}_F(s)}{(s-1)^{k+2}} ds\bigg|\\
\ll_{d, \eps,k, \sigma_0}  1 + \frac{D_F^{\frac{1-\sigma_0}{2}+\eps}}{2\pi}
\int_{\R}\frac{ (\sigma_0^2+t^2)^{d\frac{1-\sigma_0+2\eps}{4}}}{((\sigma_0-1)^2+t^2)^{(k+2)/2 }}dt
\end{multline*}
where we used \eqref{est_for_zeta_fct2} for the last inequality.
Now if $\eps$ is small enough and $\sigma_0$ sufficiently close to $1$, the integral 
$
\int_{\R}\frac{(\sigma_0^2+t^2)^{d\frac{1-\sigma_0+2\eps}{4}}}{((\sigma_0-1)^2+t^2)^{(k+2)/2}}d t
$
is finite and its value only depends on $\eps$, $\sigma_0$, and $k$. Setting $\eps'=\frac{1-\sigma_0}{2}+\eps$, we see that we can make $\eps'$ as small as we wish by choosing $\eps>0$ small and $\sigma_0<1$ close to $1$.
Hence 
\begin{equation}\label{est_modified_zeta_fct}
 \bigg|  \frac{1}{(k+1)!}\frac{d^{k+1}}{ds^{k+1}}\left[(s-1)\tilde{\Lambda}_F(s)\right]_{s=1}\bigg|
\ll_{d, \eps', k} D_F^{\eps'}
\end{equation}

To obtain the desired estimate on $|\lambda_k^F|$, $k\geq0$, we now argue inductively. Note that 
$
\lambda_k^F=  \frac{1}{(k+1)!}\frac{d^{k+1}}{ds^{k+1}}\left[(s-1)\zeta_F(s)\right]_{s=1}
$.
Hence using the product rule, 
\[
 \frac{1}{(k+1)!}\frac{d^{k+1}}{ds^{k+1}}\left[(s-1)\tilde{\Lambda}_F(s)\right]_{s=1}
 =\sum_{j=0 }^{k+1}\lambda_{j-1}^F   \frac{1}{(k-j+1)!}\frac{d^{k-j+1}}{ds^{k-j+1}}\left[\Gamma(s){\bf \Gamma}_F(s)\right]_{s=1}.
\]
As $\Gamma$ and ${\bf \Gamma}_F$ are holomorphic and non-vanishing at $s=1$, the assertion of the lemma follows inductively from \eqref{est_modified_zeta_fct}.
\end{proof}

We will later also need an upper bound on the class number $h_F$. 
\begin{proposition}\label{bound_on_class_number}
For any number field $F$ of degree $d\geq 2$ we have
\[
h_F
\ll_d 
 \Delta_F (\log\Delta_F)^{d-1}.
\]
\end{proposition}
This follows directly from \cite[Theorem 6.5]{Le92} which gives
\[
h_F
\leq \left(\frac{2}{\pi}\right)^{r_2}\Delta_F\frac{(d-1+\log D_F)^{d-1}}{(d-1)!}
\]
and hence implies the above estimate.

\subsection{Local zeta functions}
For $\Re s>1$  the Dedekind zeta function can be written as an Euler product 
\[
 \zeta_F(s)=\prod_{v\not\in S_{\infty}}\zeta_{F,v}(s)
\]
with local zeta functions
\[
\zeta_{F, v}(s)
=(1-q_v^{-s})^{-1}
=\sum_{k\geq0} q_v^{-ks}.
\]
We will later need estimates for these local zeta functions and their derivatives. 
From the series expansion of $\zeta_{F,v}(s)$, we get for any $m\in\Z_{\geq0}$,
\[
 \zeta_{F,v}^{(m)}(s)
=\sum_{k\geq 0} (\log q_v^{-k})^m q_v^{-ks}
= (-\log q_v)^m \sum_{k\geq0} k^m q_v^{-ks}
\]
if $\Re s>0$. 
\begin{lemma}\label{estimate_for_local_zeta}
For any $m_1, m_2\in\Z_{\geq0}$  we have
\begin{equation}\label{est_der_loc_zetafcts}
 \Bigg|\frac{\zeta_{F,v}^{(m_1)}(1)\zeta_{F,v}^{(m_2)}(1)}{\zeta_{F,v}^{(m_1+m_2)}(1)}\Bigg|
\ll_{m_1,m_2} 1.
\end{equation}
\end{lemma}
\begin{proof}
Suppose that $m_1, m_2>0$. Then multiplication of the series expansion gives
\[
 \big|\zeta_{F,v}^{(m_1)}(1)\zeta_{F,v}^{(m_2)}(1)\big|
\leq (\log q_v)^{m_1+m_2} \sum_{l>0} q_v^{-l} \sum_{i+j=l}  i^{m_1} j^{m_2}
\leq (\log q_v)^{m_1+m_2} \sum_{l>0} q_v^{-l} l^{m_1+m_2+1}.
\]
Now for any $m> 0$, 
\begin{multline*}
 \sum_{l>0} q_v^{-l} l^{m+1}
=\sum_{l>0} (l q_v^{-\frac{l}{2}}) q_v^{-\frac{l}{2}} l^m
\leq 2\cdot 2^m \sum_{l>0} q_v^{-\frac{l}{2}} \left(\frac{l}{2}\right)^{m}
\leq 2^{m+1}\sum_{k\geq 0}\left(q_v^{-k}k^m+q_{v}^{-k-\frac{1}{2}}\left(k+\frac{1}{2}\right)^m\right)\\
\leq 2^{m+1}\sum_{k\geq 0}q_v^{-k}k^m\left(1+q_{v}^{-\frac{1}{2}}\left(1+\frac{1}{2k}\right)^m\right)
\leq 2^{2m+2}\sum_{k\geq 0}q_v^{-k}k^m
\end{multline*}
so that
\[
 \big|\zeta_{F,v}^{(m_1)}(1)\zeta_{F,v}^{(m_2)}(1)\big|
\leq 2^{2(m_1+m_2)+2} \big|\zeta_{F,v}^{(m_1+m_2)}(1)\big|.
\]
If $m_1=0$ or $m_2=0$, then the left hand side of \eqref{est_der_loc_zetafcts} equals $\left|\zeta_{F,v}(1)\right|$. As $ \left|\zeta_{F,v}(1)\right|=(1-q^{-1})^{-1}\leq 2$, \eqref{est_der_loc_zetafcts} also holds if one of $m_1, m_2$ is $0$.
\end{proof}

\subsection{Inner products}\label{section_inner_products}
We define an Hermitian inner product 
$
\langle \cdot, \cdot\rangle:\R^{\signature}\times \R^{\signature}\longrightarrow\C
$
by 
\[
\langle x, y\rangle 
=\sum_{i=1}^{r_1}x_iy_i+\sum_{j=1}^{r_2}x_{r_1+j}\overline{y_{r_1+j}}+\sum_{j=1}^{r_2}\overline{x_{r_1+j} }y_{r_1+j}
\]
where 
$x=(x_1,\ldots, x_{r_1}, x_{r_1+1}, \ldots, x_{r_1+r_2}), y=(y_{1}, \ldots, y_{r_1}, y_{r_1+1}, \ldots, y_{r_1+r_2})\in \R^{\signature}=\R^{r_1}\times \C^{r_2}$.
This is the same inner product we would obtain from identifying $\R^{\signature}\simeq\R^{d}$ via 
\[(x_1, \ldots, x_{r_1}, x_{r_1+1}, \ldots,x_{r_1+r_2})\mapsto (x_1, \ldots, x_{r_1}, \Re x_{r_1+1}, \Im x_{r_1+1}, \ldots,  \Re x_{r_1+r_2}, \Im x_{r_1+r_2})\]
and taking the standard inner product on $\R^d$.
The map 
$\R^{\signature}\ni x\mapsto \|x\|:=\sqrt{\langle x, x\rangle} $
then defines a norm on $\R^{\signature}$.

If $K\in\Z_{\geq1}$ and $X=(x^1, \ldots, x^K), Y=(y^1, \ldots, y^K)\in \left(\R^{\signature}\right)^K:=\R^{\signature}\oplus \ldots\oplus\R^{\signature}$ (the $K$-fold direct sum), we set 
\[
\langle X , Y \rangle
=\sum_{k=1}^K \langle x^k, y^k\rangle.
\]
We also write $\|X\|=\sqrt{\langle X, X\rangle}$ if $X\in \left(\R^{\signature} \right)^K$.

\subsection{Lattices and successive minima}
Suppose $\Lambda\subseteq \R^{\signature} \simeq\R^d$  is a lattice (for us a lattice is always an additive subgroup of $\R^{\signature}\simeq \R^d$ of full rank $d$). We denote by $\Lambda^*$ the dual lattice,
\[
\Lambda^*=\{x\in\R^{\signature}\mid \forall y\in\Lambda:~\langle x, y\rangle \in \Z\}.
\]
Let $\lambda_1(\Lambda)\leq \ldots\leq\lambda_d(\Lambda)$ denote the successive minima of $\Lambda$ with respect to the quadratic form $\|\cdot\|^2$, and similarly write 
$
\lambda_1(\Lambda^*)\leq \ldots\leq\lambda_d(\Lambda^*)
$
for the successive minima of $\Lambda^*$. Let $\det\Lambda $ denote the determinant of the lattice $\Lambda$, i.e., the volume of a fundamental mesh of $\Lambda$ in $\R^{\signature}$.
We recall some well-known properties about the successive minima.
\begin{proposition}\label{properties_of_succ_min}
 \begin{enumerate}[label=(\roman{*})]
\item If $F$ is a number field with signature $\signature$ and $\Lambda\subseteq \OOO_F$, then $\lambda_1(\Lambda)\geq1$, where we view $\OOO_F$ as a lattice in $\R^{\signature}$ under a fixed  embedding $F\hookrightarrow \R^{\signature}$.

\item (Minkowski's Second Theorem) 
Let $v_{\signature}$ denote the volume of the unit ball $\{x\in\R^{\signature}\mid \|x\|\leq1\}$  with respect to the usual Lebesgue measure in $\R^d$. Then
\[
2^d (d!v_{\signature})^{-1} \det\Lambda\leq \lambda_1(\Lambda)\cdot\ldots\cdot\lambda_d(\Lambda)\leq 2^dv^{-1}_{\signature}\det\Lambda.
\]

\item\label{properties_of_succ_min2} For $i=1, \ldots, d$,
\[
\lambda_i(\Lambda)\lambda_{d-i+1}(\Lambda^*)\geq 1.
\]

\item Let $\xi_1, \ldots, \xi_d\in \Lambda$ be such that $\|\xi_i\|=\lambda_i(\Lambda)$, $i=1, \ldots, d$, and such that the additive subgroup $\Lambda'\subseteq \Lambda$ spanned by $\xi_1, \ldots, \xi_d$ has rank $d$.
 Then
\[
[\Lambda:\Lambda']\leq 2^{d}v_{\signature}^{-1}.
\]
\end{enumerate}
\end{proposition}
\begin{proof}
 \begin{enumerate}[label=(\roman{*})]
\item Let $x=(x_1, \ldots, x_{r_1}, x_{r_1+1}, \ldots, x_{r_1+r_2})\in \OOO_F\subseteq \R^{\signature}$. Applying the geometric-arithmetic mean inequality to the definition of the norm, we get
\[
\|x\|^2\geq d\bigg(\bigg(\prod_{i=1}^{r_1}x_i^2\bigg)\bigg(\prod_{i=1}^{r_2} (x_{r_1+i}\overline{x_{r_1+i}})^2\bigg)\bigg)^{\frac{1}{d}}
=d\NNN_{F/\Q}(x)^{\frac{2}{d}}
\] 
where $\NNN_{F/\Q}$ denotes the norm of $x$ as an element of $F$ over $\Q$. Since $x\in \OOO_F$, $\NNN_{F/\Q}(x)\geq 1$ so that the assertion follows.
\item See \cite[VIII.4.3, Theorem V]{Ca97}.
\item See \cite[VIII.5, Theorem VI]{Ca97}.
\item A slightly stronger version of this assertion is proven in \cite[Corollary 2.6.10]{Ma03}, we recall here the arguments necessary for our situation. First note that the discriminant of a lattice equals the square of its determinant, and therefore $[\Lambda:\Lambda']=\det\Lambda'/\det\Lambda$. By the second part of the proposition, $1/\det\Lambda\leq 2^{d}v_{\signature}^{-1} \lambda_1(\Lambda)^{-1}\cdot\ldots\cdot\lambda_d(\Lambda)^{-1}$. The volume of $\Lambda'$ equals the absolute value of the determinant of the vectors $\xi_1, \ldots, \xi_d$, and therefore by the Hadamard inequality,
$
\det\Lambda'\leq \|\xi_1\|\cdot\ldots\cdot\|\xi_d\|=\lambda_1(\Lambda)\cdot\ldots\cdot\lambda_d(\Lambda)
$.
Hence,
$
\det\Lambda'/\det\Lambda
\leq 2^{d}v_{\signature}^{-1}
$
and the assertion follows.\qedhere
 \end{enumerate}
\end{proof}
 
We will later need to bound sums over points in dual lattices.
\begin{lemma}\label{lemma_bound_on_lattice_sum}
Let $\Lambda\subseteq \R^{\signature}$ be a lattice with dual lattice $\Lambda^*$, let $K\in\Z_{\geq1}$ and denote by $\left(\Lambda^*\right)^K=\Lambda^*\oplus\ldots\oplus\Lambda^*$ the $K$-fold direct sum of $\Lambda^*$ in $(\R^{\signature})^K$. Then:
\begin{enumerate}[label=(\roman{*})]
\item For all $r>0$,
\[
\left|\left\{X\in \left(\Lambda^*\right)^K\mid \|X\|\leq r\right\}\right|\;\;
\begin{cases}
\;\;=1															&\text{if }r<\lambda_d(\Lambda)^{-1},\\
\;\;\leq \left(3r \lambda_d(\Lambda)\right)^{dK}										&\text{if }r\geq \lambda_d(\Lambda)^{-1}.
\end{cases}
\]

\item For all $t>1$,
\begin{equation}\label{eq:bound:lattice:pts}
\sum_{\substack{X\in \left(\Lambda^*\right)^K \\ X\neq0}} \|X\|^{-dK-t}
\leq 6^{dK}\zeta(t) \lambda_d(\Lambda)^{dK+t} .
\end{equation}
\end{enumerate}
\end{lemma}
Note that the left hand side of \eqref{eq:bound:lattice:pts} grows indeed exponentially fast in $t$ if $\lambda_1(\Lambda)>1$.
\begin{proof}
Note that $\left(\Lambda^*\right)^K\subseteq \left(\R^{\signature}\right)^K$ is a lattice of full rank so that we may speak of successive minima of $\left(\Lambda^*\right)^K$. But then the first successive minimum $\lambda_1(\left(\Lambda^*\right)^K)$ of $\left(\Lambda^*\right)^K$ equals $\lambda_1(\Lambda^*)$, and therefore by Proposition \ref{properties_of_succ_min}\ref{properties_of_succ_min2} we have $\lambda_1(\left(\Lambda^*\right)^K)\geq \lambda_d(\Lambda)^{-1}$. In other words, the norm of any non-zero element in $\left(\Lambda^*\right)^K$ is bounded from below by $\lambda_d(\Lambda)^{-1}>0$. 

Hence for $r>0$ the number of points $X\in \left(\Lambda^*\right)^K$ with $\|X\|\leq r$ either equals $1$ if $r<\lambda_d(\Lambda)^{-1}$ or, if $r\geq\lambda_{d}(\Lambda)^{-1}$, is bounded by \cite[Theorem 2.1]{BeHeWi93} by
\[
\left(\left\lfloor\frac{2r}{\lambda_1((\Lambda^*)^s)}+1\right\rfloor\right)^{dK}
\leq \left(2r\lambda_d(\Lambda)+1\right)^{dK}
\]
where $\lfloor x\rfloor$ denotes the largest integer $\leq x$. This in turn is clearly bounded by
\[
\left(3r\lambda_d(\Lambda)\right)^{dK}
\]
giving the first assertion.

For the second part we then get for all $r_2>r_1>0$,  and all $m>0$ for which the sum converges that
\[
\sum_{\substack{X\in \left(\Lambda^*\right)^K \\ r_1\leq \|X\|\leq r_2}}\|X\|^{-m}
\leq (3r_2\lambda_d(\Lambda))^{dK} r_1^{-m}
\]
and in particular,
\begin{align*}
\sum_{\substack{X\in \left(\Lambda^*\right)^K \\ k\lambda_{d}(\Lambda)^{-1}\leq \|X\|\leq (k+1)\lambda_d(\Lambda)^{-1}}}\|X\|^{-m}
&\leq (3 (k+1)\lambda_d(\Lambda)^{-1}\lambda_d(\Lambda))^{dK} (k\lambda_{d}(\Lambda)^{-1})^{-m}\\
&\leq 6^{dK}\lambda_d(\Lambda)^m k^{dK-m}
\end{align*}
for all integers $k\geq1$. 
Writing $m=dK+t$ with $t\geq1$, the sum 
$
\sum_{\substack{X\in \left(\Lambda^*\right)^K \\ \|X\|>0}}\|X\|^{-m}
$
therefore equals
\[
 \sum_{k\in\Z_{\geq1}}\sum_{\substack{X\in \left(\Lambda^*\right)^K \\ k\lambda_{d}(\Lambda)^{-1}\leq \|X\|\leq (k+1)\lambda_d(\Lambda)^{-1}}}\|X\|^{-(dK+t)}
\leq 6^{dK}\lambda_d(\Lambda)^{dK+t}\sum_{k\geq1} k^{-t}
=6^{dK}\lambda_d(\Lambda)^{dK+t}\zeta(t)
\]
as asserted.
\end{proof}

\subsection{Minkowski constant and ideal classes}\label{section_mink_const}
Let 
\[
M_F =\Delta_F\left(\frac{4}{\pi}\right)^{r_2}\frac{d!}{d^d}\leq \Delta_F
\]
denote the Minkowski constant of the number field $F$. 
\begin{proposition}[\cite{La86}, Chapter V, \S 4 Theorem 4]\label{minkowski_bound}
Let $h_F$ be the class number of $F$.
We can choose representatives $\aaa_1, \ldots, \aaa_{h_F}\subseteq \OOO_F$ for the ideal classes in $\OOO_F$ such that for every $i=1, \ldots, h_F$ we have
\[
\N(\aaa_i)\leq M_F.
\]
\end{proposition}
We fix a set of representatives $\AAA_F=\{\aaa_1, \ldots, \aaa_{h_F}\}$  for the ideal classes with $\N(\aaa_i)\leq M_F$ as in the lemma. Let ${\rm CL}_F$ denote the ideal class group of $F$.
By \cite[Chapter VII, \S 3]{La86} there is a natural isomorphism
\[
 \A_F^{\times}/(F^{\times} F_{\infty}^{\times})\longrightarrow {\rm Cl}_F
\]
defined as follows:
If $a=(a_v)_{v}\in \A_F^{\times}$ with $a_v=1$ for $v\in S_{\infty}$, then the equivalence class of $a$ is mapped to the ideal class of $\bbb(a):=\prod_{v\not\in S_{\infty}} \ppp_v^{\val_v(a_v)}$ for $\ppp_v\subseteq \OOO_v$ the prime ideal of $\OOO_v$, where $\val_v(a_v)$ denotes the $v$-adic valuation of $a_v$. In particular, $|a|_{\A_F} = \N(\bbb(a))^{-1}$. For every $i\in\{1, \ldots, h_F\}$ let $\delta_i'\in \A_F$ be such that $\aaa(\delta_i')=\aaa_i$ and $\delta_{i,v}'=1$ for all archimedean places $v\in S_{\infty}$. Let 
\[
 \delta_i:=(|\delta_i'|_{\A_F}^{-1/d},\ldots, |\delta_i'|_{\A_F}^{-1/d},1,\ldots)\cdot \delta_i'\in\A_F^1.
\]
Then by definition
\[
 \bigcup_{i=1}^{h_F} (\delta_i F^{\times} F_{\infty}^{\times}) =\A_F^{\times},
\]
and for $F_{\infty}^1:=\A_F^1\cap F_{\infty}$,
\[
  \bigcup_{i=1}^{h_F}( \delta_i F^{\times} F_{\infty}^{1})=\A_F^1.
\]
Thus we can choose a fundamental domain $\FFF_{\times}\subseteq \bigcup_{i=1}^{h_F} (\delta_i F_{\infty}^{1})$ for $F^{\times}\hookrightarrow \A_F^1$.
We fix such a fundamental domain from now on so that in particular, $\vol(\FFF_{\times})=\vol(F^{\times}\backslash \A_F^1)=\lambda_{-1}^{F}$.

We later will need to count lattice points in the inverse of an ideal of $\OOO_F$. From Lemma \ref{lemma_bound_on_lattice_sum}, we obtain the following.
\begin{cor}\label{bound_for_sum_over_ideal_points}
Let $K\in\Z_{\geq1}$, $q\in\Z_{\geq0}$.
Let $\bbb\in\AAA_F$, and denote by $\aaa=\bbb^q$ the $q$-th power of the ideal $\bbb$. Then:
\begin{enumerate}[label=(\roman{*})]
\item For all $r>0$, 
\[
\left|\left\{X\in \left(\aaa^{-1}\right)^K\mid \|X\|\leq r\right\}\right|
\ll_{n, d, K} 1+ (rD_F)^{dKq}.
\]

\item For all $t\geq 2$,
\[
\sum_{\substack{X\in \left(\aaa^{-1}\right)^K \\ X\neq0}} \|X\|^{-dK-t}
\ll_{n,d,K}  D_F^{(dK+t)q}.
\]
\end{enumerate}
\end{cor}
\begin{proof}
 Let $\aaa\in\AAA_F$, and consider $\aaa$ as a lattice in $\R^{\signature}$ via the inclusion $\OOO_F\subseteq F\hookrightarrow\R^{\signature}$. 
The norm of $\aaa$ is related to the volume of the fundamental mesh of $\aaa$ in $\R^{\signature}$ by
$
\det\aaa=\Delta_F \N(\aaa)
$
(see \cite[Proposition I.5.2]{Ne99}).
The upper bound of Minkowski's Second Theorem therefore gives
\[
\lambda_1(\aaa)\cdot\ldots\cdot \lambda_d(\aaa)
\leq 2^dv_{\signature}^{-1} \Delta_F\N(\aaa)
=2^d v_{\signature}^{-1} \Delta_F \N(\bbb)^q
\leq 2^dv_{\signature}^{-1}\Delta_F^{1+q}
\]
where we used that $\bbb\in\AAA_F$ for the second inequality. 
Hence $\lambda_d(\aaa)\leq 2^dv_{\signature}^{-1} \Delta_F^{1+q}$. 
Since $\aaa^*=\aaa^{-1} \DDD$ for $\DDD=\OOO_F^*\subseteq F$ the different ideal of $F$, and $\OOO_F\subseteq \OOO_F^*$, we get
$ \aaa^{-1}\subseteq \aaa^{-1} \OOO_F^*=\aaa^*$.
Hence the assertion of the corollary follows from Lemma \ref{lemma_bound_on_lattice_sum} by noting  that $\zeta(t)$ is monotonically decreasing for $t\rightarrow\infty$ so that in particular $\zeta(t)\leq \zeta(2)$ for all $t\geq2$.
\end{proof}

\subsection{Fundamental domains}\label{section_fundamental_domains}
For later purposes, we need to choose a compact set in $\A_F$ containing a fundamental domain for the lattice $F\hookrightarrow\A_F$. Recall that we already fixed a fundamental domain $\FFF_{\times}\subseteq \A_F^1$ for $F^{\times}\hookrightarrow\A_F^1$ in \S \ref{section_mink_const}.
\begin{lemma}[Additive fundamental domain]
Let $F$ be a number field with signature $\signature$.
\begin{enumerate}[label=(\roman{*})]
\item
 The compact set
\[
\FFF^0=\{x\in F_{\infty}=\R^{\signature} \mid \|x\|\leq 2^{2d}v_{\signature}^{-2}\Delta_F\}\subseteq F_{\infty}
\]
contains a fundamental domain for the lattice $\OOO_F\hookrightarrow F_{\infty}$, i.e., $\OOO_F+\FFF^0=F_{\infty}$.

\item
The compact set
\[
\FFF_+=\FFF^0\times \widehat{\OOO}_F\subseteq \A_F,
~~~~~
\widehat{\OOO}_F:=\prod_{v\not\in S_{\infty}} \OOO_v,
\]
contains a fundamental domain for the lattice $F\hookrightarrow \A_F$, i.e., $F+\FFF_+=\A_F$. Its volume is bounded by
\[
\vol(\FFF_+)\ll_{d} \Delta_F^{d} .
\]
\end{enumerate}
\end{lemma}
\begin{proof}
Let $\xi_1, \ldots, \xi_d\in\OOO_F$ be linearly independent such that $\|\xi_i\|=\lambda_i(\OOO_F)$, and let $\Lambda\subseteq F_{\infty}$ be the lattice spanned by $\xi_1, \ldots, \xi_d$. By Proposition \ref{properties_of_succ_min} we have $\lambda_d(\OOO_F)\leq 2^dv_{\signature}^{-1} \det\OOO_F=2^dv_{\signature}^{-1}\Delta_F$ and $[\OOO_F:\Lambda]\leq 2^{d}v_{\signature}^{-1}$. 
In other words, a fundamental mesh for the lattice $\Lambda$ in $\R^{\signature}$ is contained in the compact set $\{x\in\R^{\signature}\mid \|x\|\leq 2^dv_{\signature}^{-1}\Delta_F\}$. Therefore, using the bound on $[\OOO_F:\Lambda]$, there exists a fundamental mesh for the lattice $\OOO_F$ in $\R^{\signature}$ contained in the compact set $\{x\in\R^{\signature}\mid \|x\|\leq 2^{2d} v_{\signature}^{-2}\Delta_F\}=\FFF^0$. This proves the first part of the lemma.

The second part is then deduced from the first one by using  \cite[VII, \S 3, Theorem 3]{La86}  showing that the set $\FFF_+$  contains a fundamental domain for $F\hookrightarrow\A_F$. The bound for the volume of $\FFF_+$ follows from the definition of the measures and the explicit definition of $\FFF^0$.
\end{proof}

We also need compact sets containing fundamental domains for the torus $T_0(F)\hookrightarrow T_0(\A_F)^1$ and the unipotent subgroup $U_0(F)\hookrightarrow U_0(\A_F)$.
\begin{lemma}\label{fundamental_domains_in_TU}
Let $F$ be a number field of signature $\signature$.
The compact sets
\[
\MMM=\{\diag(t_1, \ldots, t_n)\in T_0(\A_F)^1\mid\forall~i:~ t_i\in \FFF_{\times}\}
\subseteq T_0(\A_F)^1
\]
and
\[
\NNN=\{u=(u_{i,j})_{i, j=1, \ldots, n}\in U_0(\A_F)\mid\forall~i<j:~ u_{i,j}\in \FFF_{+}\}
\subseteq U_0(\A_F)
\]
contain fundamental domains for $T_0(F)\hookrightarrow T_0(\A_F)^1$ and  $U_0(F)\hookrightarrow U_0(\A_F)$, respectively, i.e.,
\[
T_0(F)\MMM=T_0(\A_F)^1
\;
\text{ and }
\;\;
U_0(F)\NNN=U_0(\A_F).
\]
Moreover,
\[
\vol(\MMM)=\vol(F^{\times}\backslash \A_F^1)^n=\left(\lambda_{-1}^F\right)^n
\;
\text{ and }
\;\;
\vol(\NNN)\ll_{n, d} \Delta_F^{\frac{n(n-1)d}{2}}.
\]
\end{lemma}
\begin{proof}
The set $\MMM$ is in fact a fundamental domain for $T_0(F)\backslash T_0(\A_F)^1$ which follows immediately from its definition.  The estimate of the volumes $\vol(\NNN)$ and $\vol(\MMM)$ is also a direct consequence of the definitions so that we are left to show that $U_0(F)\NNN=U_0(\A_F)$.
For this,  we essentially follow the first part of the proof of \cite[Lemma 4.4]{PlRa94} where the analogue assertion for $F=\Q$ is shown. Let $u=(u_{i,j})_{i,j=1, \ldots, n}\in U_0(\A_F)$ and let $y^i\in F$ be such that $y^i+u_{i, i+1}\in\FFF_{+}$ for $i=1, \ldots, n-1$. Such $y^i$ exist, since $F+\FFF_{+}=\A_F$. Let $y=(y_{i,j})_{i,j=1, \ldots, n}\in U_0(F)$ with $y_{i, i+1}=y^{i}$, $i=1, \ldots, n-1$, and $y_{i,j}=0$ for all $j>i+1$. Then the product $x=(x_{i,j})=yu\in U_0(\A_F)$ is of the form
\[
x=\begin{pmatrix}
   1		&	y^1+u_{1,2}		&	\ldots	&		*				\\
 0		&	\ddots			&	\ddots	&		\vdots				\\
\vdots		&	\ddots			&	\ddots	&		y^{n-1}+u_{n-1, n}		\\
  0		&	\ldots			&	0	&		1	
  \end{pmatrix}	
\]
and satisfies $x_{i, i+1}=y^i + u_{i, i+1}\in \FFF_{+}$ by construction.
Now assume that for some $1\leq k\leq n-1$ we have found $y\in U_0(F)$ such that for $x=yu$ we have $x_{i,j}\in\FFF_{+}$ for all $1\leq i<j\leq i+k\leq n$.
Let $y^i\in F$ be such that $y^i+x_{i, i+k+1}\in\FFF_{+}$ for $i=1, \ldots, n-k-1$ and define $y'\in U_0(F)$ so that $y_{i,j}'= \delta_{i,j}$ for all $i, j$ unless $(i, j)=(i, i+k+1)$ in which case $y_{i, i+k+1}'=y^i$. (Here $\delta_{i,j}$ denotes the Kronecker delta.) Then the entries of the matrix $y'yu$ on the $(k+1)$-th upper diagonal are contained in $\FFF_{+}$, and the entries on all lower diagonals coincide with the respective entries in $x$ which are already contained in $\FFF_+$ by induction hypothesis. The assertion therefore follows by induction on $k$. 
\end{proof}

\section{Unipotent conjugacy classes and orbital integrals}\label{section_unip_orb_int}
\subsection{Unipotent conjugacy classes in ${\rm GL}_n$}
In this section we recall some properties of unipotent classes in $G=\GL_n$. 
We will in particular make use of the fact that every unipotent conjugacy class in $\GL_n$ is a Richardson class (cf. Proposition \ref{every_unip_orb_induced} below), which will later simplify the definition of certain measures.

The set of unipotent elements $\UUU_G$ in $G$ is a $\Q$-variety and any unipotent conjugacy class is defined over $\Q$ so that we can work over $\Q$ in this section and extend the results to arbitrary number fields and their local completions by extension of scalars.

 Let $\Ufrak^G$ denote the set of unipotent conjugacy classes in $G$ under $G$-conjugation. Since any  Levi subgroup $M\in\LLL$ is $\Q$-isomorphic to a product of general linear groups, everything in this section also applies equally well to $M$ instead of $G$. We attach a superscript $M$ in the notation to indicate that we work with respect to $M$ instead of $G$. In particular, we write $\Ufrak^M$ for the set of unipotent conjugacy classes in $M$. 

Recall the notion of an induced unipotent conjugacy class in  $G$ (see, e.g., \cite[\S 5.10]{Hu95}): Suppose that $M\in\LLL$ is the Levi component of a parabolic subgroup $P=MU\in\FFF$ and $\VVV\in \Ufrak^M$ is a unipotent conjugacy class in $M$.  Then the induced conjugacy class $\III_M^G\VVV\in\Ufrak^G$ is the unique unipotent conjugacy class in $G$ intersecting $\VVV\cdot U$ in a dense-open set. As the notation suggests, this definition is independent of $P$. More precisely, $\III_M^G\VVV $ only depends on the $G$-orbit of the pair $(M, \VVV)$ by \cite[Theorem 2.2]{LuSp79}.
Recall that we denote by $\One^M\in \Ufrak^M$ the trivial conjugacy class in $M$. 
Note that $\III_{M_1}^G\One^{M_1}=\III_{M_2}^G\One^{M_2}$ if $M_1, M_2\in\LLL$ are conjugate. (Note that two elements in $\LLL$ are conjugate if and only if they are conjugate via some Weyl group element in $W^G$.)
\begin{proposition}[Richardson, \cite{Ri74}]\label{every_unip_orb_induced}
The set of unipotent conjugacy classes in $G=\GL_n$ is in bijection with the set of Weyl group orbits of Levi subgroups in $\LLL$. More precisely, the bijection is given by 
\[
W^G\backslash \LLL\ni W^G\cdot M\mapsto \III_M^G\One^M\in \Ufrak^G.
\]
In particular, any unipotent conjugacy class in $G$ is induced from a trivial conjugacy class in the Levi component of some standard parabolic subgroup.
\end{proposition}
Note that if $\VVV_1, \VVV_2\in\Ufrak^M$ are different, their induced classes $\III_M^G\VVV_1$ and $\III_M^G\VVV_2$ might coincide.

We denote the conjugacy class in $M$ induced from the trivial conjugacy class in  $L\subseteq M$  by $\VVV_{L}^M=\III_L^M\One^L$.
If $M\in\LLL$, we choose a section $\Ufrak^M\ni \VVV\mapsto L\in \LLL^M$ for the map $\LLL^M\ni L\mapsto \VVV_L^M$ by choosing a suitable $L\in\LLL^M$, which is the Levi subgroup of a standard parabolic subgroup. 
We denote this Levi subgroup $L$ by $\Mbf_{M,\VVV}$ and the unipotent radical of the unique standard parabolic subgroup in $G$ having $\Mbf_{M,\VVV}$ as its Levi subgroup by $\Ubf_{M, \VVV}$. 
It is clear from the definition that $\Mbf_{M, \VVV}$ and $\Mbf_{G, \III_M^G\VVV}$ are in the same Weyl group orbit in $\LLL$, i.e.\ represent the same equivalence class of $W^G\backslash \LLL$.

\subsection{Reduction to the local case}\label{subsection_orb_int}
Let  $G(\R^{\signature})^1:=G(\R^{\signature})\cap G(\A_F)^1$ and let $C_c^{\infty}(G(\R^{\signature})^1)$ denote the space of complex-valued, smooth,  compactly supported functions on $G(\R^{\signature})^1$.
If $\Xi\subseteq G(\R^{\signature})^1$ is a compact subset, let $C_{\Xi}^{\infty}(G(\R^{\signature})^1)\subseteq C_c^{\infty}(G(\R^{\signature})^1)$ be the subspace of functions supported in $\Xi$.
Let $\ggG_{\signature}^1$ denote the Lie algebra of $G(\R^{\signature})^1$ and $\ggG_{\signature}^1(\C)$ its complexification. 
Let $\UUU(\ggG_{\signature}^1(\C))$ be the universal enveloping algebra of $\ggG_{\signature}^1(\C)$ with usual filtration $\UUU(\ggG_{\signature}^1(\C))_{\leq k}$, $k\in\Z_{\geq0}$. 
For each $k$ fix a basis $\BBB_k=\BBB_{\signature,k}$ of the finite dimensional $\C$-vector space $\UUU(\ggG_{\signature}^1(\C))_{\leq k}$.
Then $ \UUU(\ggG_{\signature}^1(\C))$ acts  from the left on functions $f_{\infty}\in C_c^{\infty}(G(\R^{\signature})^1)$, and
\[
\left\|f_{\infty}\right\|_k=\sum_{X\in\BBB_k} \sup_{x\in G(\R^{\signature})^1} \left|X*f_{\infty}(x)\right|
\]
defines a seminorm on $C_c^{\infty}(G(\R^{\signature})^1)$.

If $F$ is  number field of signature $\signature$, let $C_c^{\infty}(G(F_S)^1)$ be the space of complex-valued, smooth, compactly supported functions on $G(F_S)^1$.
We now  describe the unipotent orbital integrals $J_L^G(\VVV,f)$ for $L\in\LLL$, $\VVV\in \Ufrak^L$, and $f\in C_c^{\infty}(G(F_S)^1)$. 
These distributions are defined in terms of $(G,L)$-families and therefore satisfy some nice properties (cf. \cite[\S 18]{Ar05}). In particular, using Arthur's splitting formula for $(G, L)$-families, we only need to study local distributions on $C_c^{\infty}(G(F_v))$, $v\in S$. 
However, we group the distributions at the archimedean places together as this is more convenient for us.
We first need to define the constant term of a function along the unipotent radical of a  parabolic subgroup: If $Q\in \FFF$ and $f_v\in C_c^{\infty}(G(F_v))$, $v\in S_{\text{fin}}$, set
\[
f_v^{(Q)}(m)=\delta_Q(m)^{\frac{1}{2}}\int_{U_Q(F_v)}\int_{\cpt_v}f_v(k^{-1}muk) dk~du,
\;\text{ for }
m\in M_Q(F_v),
\]
and if $f_{\signature}\in C_c^{\infty}(G(\R^{\signature})^1)$, set
\[
f_{\signature}^{(Q)}(m)
=\delta_Q(m)^{\frac{1}{2}}\int_{U_Q(F_{\signature})}\int_{\cpt_{\signature}}f_{\signature}(k^{-1}muk) dk~ du,
\;\text{ for }
m\in M_Q(\R^{\signature})^1.
\]
The following two properties of $f_v^{(Q)}$ and $f_{\signature}^{(Q)}$ are immediate from the definitions.
\begin{lemma}
 \begin{enumerate}[label=(\roman{*})]
\item We have $f_v^{(Q)}\in C_c^{\infty}(M_Q(F_v))$ and $f_{\signature}^{(Q)}\in C_c^{\infty}(M_Q(\R^{\signature})^1)$.
\item If $v$ is non-archimedean, then 
\[
\One_{\cpt_v}^{(Q)}
=\vol(U_Q(\OOO_{F_v})) \One_{\cpt_v^{M_Q}}
=  \N(\DDD_v)^{-\dim U_Q/2} \One_{\cpt_v^{M_Q}}
\]
is the characteristic function of the maximal compact subgroup $\cpt_{v}^{M_Q}=\cpt_v\cap M_Q(F_v)$ of $M_Q(F_v)$ multiplied by the volume of $\vol(U_Q(\OOO_{F_v}))=\vol(U_Q(F_v)\cap\cpt_v)$ (with respect to the measure on $U_Q(F_v)$).
\end{enumerate}
\end{lemma}

To describe Arthur's splitting formula we proceed inductively on the number of valuations in $S_{\text{fin}}$. 
Suppose we partition $S$ into two disjoint sets $S_1$ and $S_2$. Let $f=f_{S_1}f_{S_2}\in C_c^{\infty}(G(F_{S_1})) \cdot C_c^{\infty}(G(F_{S_2}))$ be a test function, and $\gamma=\gamma_1\gamma_2\in L(F_{S_1})L(F_{S_2})$ be a unipotent element. Then by \cite[(18.7)]{Ar05} (cf.\ also \cite{Ar88b})
\[
 J_L^G(\gamma, f)
=\sum_{L_1, L_2\in \LLL(L)} d_L^G(L_1, L_2) J_L^{L_1}(\gamma_1, f_{S_1}^{(Q_1)}) J_L^{L_2}(\gamma_2, f_{S_2}^{(Q_2)}),
\]
where $d_L^G(L_1, L_2)\in\R_{\geq0}$ are constants depending only on $L$, $L_1$, and $L_2$ (but not on the sets $S_1, S_2, S$) with the property that $d_L^G(L_1, L_2)=0$ unless the natural map 
\begin{equation}\label{eq:isom:of:spaces}
 \aaa_L^{L_1}\oplus \aaa_L^{L_2}\longrightarrow \aaa_L^G
\end{equation}
is an isomorphism. Further, the parabolic subgroups $Q_i\in\PPP(L_i)$ are chosen as follows: If $d_L^G(L_1, L_2)\neq0$, i.e., if \eqref{eq:isom:of:spaces} is an isomorphism, we have $\aaa_L^G\simeq \aaa_L^{L_1}\oplus\aaa_L^{L_2}\simeq \aaa_{L_1}^G\oplus\aaa_{L_2}^G$. Fix a small $\xi\in \{(H, -H)\in\aaa_L\oplus\aaa_L\}\subseteq \aaa_L\oplus\aaa_L$ in general position so that we can write $\xi=\xi_1-\xi_2$ with $\xi_i\in \aaa_{L_i}^G$ in general position. Then $Q_i\in\PPP(L_i)$ is defined to be the unique Levi subgroup such that $\xi_i\in\aaa_{Q_i}^+$ (cf.\ \cite[pp. 357-358]{Ar88b}). 
This defines a partial section $(L_1, L_2)\mapsto (Q_1, Q_2)$ of the map
\[
\FFF(L)\times\FFF(L)\longrightarrow \LLL(L)\times\LLL(L)
\]
for every pair $(L_1, L_2)$ with $d_L^G(L_1, L_2)\neq0$. 

We can repeat the process by partitioning $S_i$ again, and splitting $J_M^{L_i}(\gamma_i, f_{ S_i}^{(Q_i)})$ into a sum as before but with respect to $L_i$ instead of $G$. 
Note that this process stops if $S_i$ is a singleton.  Note also that if $L\subseteq L'\subseteq L''\subseteq G$ and $Q'\in\PPP^{L''}(L')$, $Q''\in \PPP(L'')=\PPP^G(L'')$, then $Q'U''\in\PPP^G(L')$ for $U''$ the unipotent radical of $Q''$, and $(f^{(Q'')})^{(Q')}=f^{(Q'U'')}$. 

Let $\LLL_{S}(L)$ denote the set of all tuples $\underline{L}=(L_{\signature})\cup (L_v)_{v\in S_{\text{fin}}}$ with $L_{\signature},L_v\in\LLL(L)$, and let 
 $\LLL_S^0(L)\subseteq \LLL_S(L)$ denote the subset of tuples for which 
\begin{equation}\label{isomorphism_sum_tuples}
\aaa_{L}^{L_{\signature}}\oplus\bigoplus_{v\in S_{\text{fin}}} \aaa_{L}^{L_v}\longrightarrow \aaa_L^G 
\end{equation}
is an isomorphism.
The procedure described above then yields a map
\[
 \LLL_S^0(L)\longrightarrow \R_{\geq0},\;\;\;
\underline{L}\mapsto d_L^G(\underline{L})
\]
such that $d_L^G(\underline{L})=0$ unless
\[
 \aaa_{L}^{L_{\signature}}\oplus\bigoplus_{v\in S_{\text{fin}}}\aaa_L^{L_v}\longrightarrow\aaa_L^G
\]
is an isomorphism. It is clear from the construction that this can be an isomorphism only if at most $\dim\aaa_L^G$-many $L_v$ are different from $L$.  Moreover, the coefficients $d_L^G(\underline{L})$ can only take finitely many different values, and this set of possible values is independent of the set $S$. 
The procedure also yields a partial section $\underline{L}\mapsto (Q_{\signature}, Q_v, v\in S_{\text{fin}})$ of the map
\[
 \underbrace{\FFF(L)\times\ldots\times\FFF(L)}_{(|S_{\text{fin}}|+1)\text{-times}}
\longrightarrow
\underbrace{\LLL(L)\times\ldots\times\LLL(L)}_{(|S_{\text{fin}}|+1)\text{-times}}=\LLL_S(L)
\]
defined for every $\underline{L}\in\LLL_S^0(L)$.

In particular, Arthur's splitting formula for our case takes the following form:
\begin{lemma}\label{reduction_to_local_case}
 There are constants $d_L^G({\underline{L}})\in \R_{\geq0}$ for $\underline{L}\in \LLL_S^0(L)$  such that for all $f=f_{\infty}\otimes\bigotimes_{v\in S_{\text{fin}}}f_v\in C_c^{\infty}(G(\R^{\signature})^1)\otimes C_c^{\infty}(G(F_{S_{\text{fin}}}))\subseteq C_c^{\infty}(G(F_S)^1)$ and all unipotent conjugacy classes $\VVV\in\Ufrak^L$ we have
\begin{equation}\label{red_to_loc_case}
J_L^G(\VVV,f)
=\sum_{\underline{L}\in\LLL_S^0(L)} d_L^G(\underline{L}) 
\Big(J_L^{L_{\signature}}(\VVV, f^{(Q_{\signature})}_{\infty})\cdot \prod_{v\in S_{\text{fin}}} J_L^{L_v}(\VVV,f_v^{(Q_v)})\Big)
\end{equation}
where $Q_{\signature}=Q_{\signature}(\underline{L})\in\PPP(L_{\signature}) $ and $Q_v=Q_v(\underline{L})\in\PPP(L_v)$, $v\in S_{\text{fin}}$, denote certain parabolic subgroups associated with $\underline{L}$ according to the above procedure. 
Moreover, the coefficients $\underline{L}\in \LLL_S^0(L)$ satisfy the following properties:
\begin{enumerate}[label=(\roman{*})]
\item $|d_L^G(\underline{L})|\ll_n 1$,
\item 
$\left|\left\{v\in S_{\text{fin}}\mid L_v\neq L\right\}\right|\leq \dim \aaa_L^G$,

\item
$\sum_{v\in S_{\text{fin}}} \dim\aaa_L^{L_v} \leq \dim \aaa_L^G$. \qed
\end{enumerate}
\end{lemma}

\subsection{Unipotent orbital integrals}\label{subsection_local_orb_int}
In this section we define the local distributions $J_L^G(\VVV,f_v)$  for  $\VVV\in\Ufrak^L$, $v$ a non-archimedean place of $F$, and $f_v\in C_c^{\infty}(G(F_v))$. 
The definition of $J_L^G(\VVV, f_{\signature})$ for $f_{\signature}\in C_c^{\infty}(G(\R^{\signature})^1)$ is analogous and omitted here as we do not need to analyse the distributions at the archimedean places (cf. the next sections). 
By the splitting formula above it suffices to study these $v$-adic distributions separately.
Hence let $v$ be an arbitrary non-archimedean place of $F$, $\VVV\in\Ufrak^L$, and $f\in C_c^{\infty}(G(F_v))$. 
Consider first the invariant distribution $J_G^G(\VVV, f)$ which by definition equals the integral of $f$ against an invariant measure on $\VVV$. By \cite[Lemma 5.3]{LaMu09} there exists $c>0$ such that for all $f\in C_c^{\infty}(G(F_v))$ we have
\[
 J_G^G(\VVV, f)
= c\int_{\cpt_v}\int_{\Ubf_{G, \VVV}(F_v)} f(k^{-1} u k)~du~dk.
\]
We need to make sure that the normalisation constant $c$ is compatible with \cite{Ar88a}. However, by \cite[p. 255]{Ar88a} we have
\[
\sum_{\VVV'} J_G^G(\VVV', f)
=\lim_{\substack{a\in A_{\Mbf_{G,\VVV}, \text{reg}}^{\infty}: \\ a\rightarrow 1 }} J_{G}^G(a u_0, f)
=\int_{\cpt_v}\int_{\Ubf_{G, \VVV}(F_v)} f(k^{-1} u k)~du~dk
\]
for $u_0\in \Ubf_{G, \VVV}(F_v)$ a representative of $\VVV$, and $\VVV'$ running over all unipotent classes for which $\VVV'\cap U_P(F_v)$ is open in $U_P(F_v)$ for every $P\in\PPP(\Mbf_{G, \VVV})$. Since $U_P$ is an irreducible variety for every $P$, there is only one such orbit, namely $\VVV$ itself. Hence
\begin{multline*}
 c\int_{\cpt_v}\int_{\Ubf_{G, \VVV}(F_v)} f(k^{-1} u k)~du~dk
=J_G^G(\VVV, f)
=\sum_{\VVV'} J_G^G(\VVV', f)\\
=\int_{\cpt_v}\int_{\Ubf_{G, \VVV}(F_v)} f(k^{-1} u k)~du~dk
\end{multline*}
so that $c=1$. Note that this integral does not depend on our previous choice of $\Mbf_{G,\VVV}$ (which determines $\Ubf_{G, \VVV}$) because any other possible choice for $\Mbf_{G, \VVV}$ would be Weyl-group conjugate (i.e.\ in particular $\cpt_v$-conjugate) to $\Mbf_{G, \VVV}$.

If $L\in\LLL$ now is arbitrary, $J_L^G(\VVV, f)$ is an integral over some non-invariant measure on $\VVV$: This measure can be described as the product of the invariant measure on the induced orbit $\III_L^G\VVV$ with a certain weight function $w_{L,\VVV,v}$, i.e.\
\[
J_L^G(\VVV, f)
=  \int_{\cpt_{v}} \int_{\Ubf_{L, \VVV}(F_{v})} f(k^{-1} uk)w_{L,\VVV,v}(u) dk~du,
\]
where $w_{L,\VVV,v}:\III_L^G\VVV(F_v)\longrightarrow\C $ is a certain $\cpt_{v}$-conjugation invariant weight function depending on $L$ and $\VVV$. Note that the intersection of $\III_L^G\VVV$ with $\Ubf_{L,\VVV}$ is open-dense in $\Ubf_{L,\VVV}$ so that $w_{L,\VVV,v}$ is defined almost everywhere on $\Ubf_{L, \VVV}(F_v)\times\cpt_v$.

To describe the weight functions in more detail, we follow their construction in \cite{Ar88a}. As $G=\GL_n$ is $\Q$-split, the construction simplifies. 
Let $r=\dim\aaa_L^G$ and denote by $\Wt(\aaa_L^G)\subseteq(\aaa_L^G)^* $ the weight lattice.
For every $Q\in\PPP(L)$  fix a basis $\varpi_1^Q, \ldots, \varpi_r^Q\in \Wt(\aaa_L^G)$ of $(\aaa_L^G)^*$ consisting of $Q$-dominant elements (if $Q$ is standard, this is just $\widehat{\Delta}_Q$), and let 
$
\Omega_L=\{\varpi_{i}^Q\mid Q\in\PPP(L), ~i=1, \ldots, r\}
$.
For $\varpi\in \Omega_L$ one defines as in \cite[(3.8)]{Ar88a} a polynomial over $\Q$ 
\[
W_{\varpi}^{\VVV}:\III_L^G\VVV \longrightarrow V
\]
for $V$ a finite dimensional affine space over $\Q$ with fixed basis. The target space $V$ can be chosen to be the same for all $\varpi$ and $\VVV$.
 This finite collection of polynomials $\{W_{\varpi}^{\VVV}\}_{\varpi,\VVV}$ is in particular independent of $F$ and $v$, but only depends on $n$ and $L$.
The weight function $w_{L,\VVV,v}$ is by definition the special value of a function associated with this set of polynomials as in \cite{Ar88a}. 
More explicitly, for any $x\in\III_L^G\VVV(F_v) $, 
\[
w_{L,\VVV,v}(x)
=\sum_{\omega}c_{\omega}\prod_{\varpi\in \omega} \log\left\|W_{\varpi}^{\VVV}(x)\right\|_{v}
\] 
where $\omega$ runs over multi-sets consisting of elements from $\Omega_L$ of cardinality $|\omega|\leq r$ (counted with multiplicities),  and the coefficients $c_{\omega}$ are depending on $n$, $L$, and $M$ only, but not on the local field $F_v$.
Here $\|\cdot\|_v$ denotes the usual norm on $V(F_v)$ with respect to the fixed basis (cf. also \eqref{definition_norm} below).
\begin{lemma}\label{lemma_first_est_weight_orb_int}
 For any $L$, $\VVV$,  and $v$  we have
\[
\left|w_{L,\VVV,v}(x)\right|
\ll_n \sum_{\varpi\in\Omega} \Big(1+\left|\log\left\|W_{\varpi}^{\VVV}(x)\right\|_v\right|^r\Big)
\]
for all $x\in \III_L^G\VVV(F_v)$.
If $L=G$, then  $w_{G, \VVV, v}\equiv1$.
\end{lemma}

Let $\Pi$ be the smallest (finite) set of places of $\Q$ (including $\infty$) such that for any $p\not\in\Pi $, the $p$-adic norm of all non-zero coefficients of $W_{\varpi}^{\VVV}$ is $1$ for all $\varpi$ and $\VVV$.
Let  $\Pi_F$ be the set of places of $F$ above the primes in $\Pi$.
Then, if $v\not\in \Pi_F$ and $x\in \Ubf_{L, \VVV}(\OOO_v)$, we have for all $\varpi\in\Omega_L$
\begin{equation}\label{inequality_for_weight_polynomial}
\left\| W_{\varpi}^{\VVV}(x)\right\|_v\leq 1.
\end{equation}

Finally, we record a useful property about the coefficients $a^L(\VVV, S)$ defined by the fine expansion of the unipotent distribution \eqref{expansion_unipotent_dist}:
\begin{lemma}
Suppose $L\in\LLL$ is conjugate to the standard Levi subgroup $L'=\GL_{n_1}\times\ldots\times\GL_{n_{r+1}}\subseteq \GL_n$ via the Weyl group element $w$, i.e. $L'=w^{-1} Lw$. If $M\in\LLL^L$, and 
$ w^{-1} Mw=M_1\times\ldots\times M_{r+1}$ 
with 
$M_i\in\LLL^{\GL_{n_i}}$,
then the coefficient $ a^{L}(\VVV_M^L, S)$ factorises as
\begin{equation}\label{factoring_coeff}
 a^{L}(\VVV_M^L, S)
= a^{\GL_{n_1}}(\VVV_1^{\GL_{n_1}}, S)\cdot\ldots\cdot a^{\GL_{n_{r+1}}}(\VVV_{r+1}^{\GL_{n_{r+1}}}, S)
\end{equation}
where $\VVV_i^{\GL_{n_i}}=\VVV_{M_i}^{\GL_{n_i}}\subseteq \GL_{n_i}(F)$.

\end{lemma}
In particular, it will suffice to show Theorem \ref{estimate_for_coeff} for $L=G$. We will therefore state and prove all auxiliary results  only for this case.
\begin{proof}
The assertion follows immediately from the definition of the weighted orbital integrals and of $J_{\text{unip}}^L(f)$ (cf. \S \ref{section_unip_contr} below), 
because the spaces of test functions
$
C_c^{\infty}(\GL_{n_1}(F_{v}))\times\ldots\times C_c^{\infty}(\GL_{n_{r+1}}(F_{v}))
$
is dense in $C_c^{\infty}(L(F_{v}))$, and 
$
C_c^{\infty}(\GL_{n_1}(\A_F)^1)\times\ldots\times C_c^{\infty}(\GL_{n_{r+1}}(\A_F)^1)
$
is dense in $C_c^{\infty}(L(\A_F)^1)$ (under the canonical inclusions).
\end{proof}

\subsection{Special test functions}\label{subsection_test_fcts}
In this section we define a family of special test functions at the archimedean places separating the contributions from the different unipotent orbits on the left hand side of  \eqref{reduction_to_ind_step}. They will be used in the proof of our main result in \S \ref{section_unip_contr}.
Recall that we defined a norm $\|\cdot\|$ on $\R^{\signature}$ in \S \ref{section_inner_products}. If $X=(X_{ij})\in \Mat_{n\times n}(\R^{\signature})$ is an $n\times n$-matrix, we denote by $\|X\|= \big(\sum_{i,j}\|X_{ij}\|^2\big)^{1/2}$ the usual matrix norm obtained from $\|\cdot\|$.
\begin{lemma}\label{prop_spec_test_fcts}
Let $\signature$ be a signature of degree $d$.
There exists a real number $s>0$ such that for 
\[
 \Xi_{\signature}^+:=\{X\in \Mat_{n\times n}(\R^{\signature}) \mid \|X\|\leq s\}\subseteq \Mat_{n\times n}(\R^{\signature})
\]
and
\[
 \Xi_{\signature}:=\Xi_{\signature}^+\cap G(\R^{\underline{r}})^1
\]
the following holds.
There exist smooth functions $f_{\signature}^{\VVV}\in C_{\Xi_{\signature}}^{\infty}(G(\R^{\signature})^1)$ for $\VVV\in \Ufrak^G$ such that 
\begin{enumerate}[label=(\roman{*})]
\item\label{prop_spec_test_fctsi}
$f^{\VVV}_{\signature}$
is $\cpt_{\signature}$-conjugation invariant for any $\VVV\in\Ufrak^G$.

\item\label{prop_spec_test_fctsiii}
If $F$ is  a number field of signature $\signature$, and $\One_{\cpt_{S_{\text{fin}}}}$ is the characteristic function of $\cpt_{S_{\text{fin}}}\subseteq G(F_{S_{\text{fin}}})$, then
\[
J_{G}^G(\VVV_1, f^{\VVV_2}_{\signature}\cdot \One_{\cpt_{S_{\text{fin}}}})
=
    \begin{cases}
    \prod_{v\in S_{\text{fin}}} \N(\DDD_v)^{-\frac{1}{4}\dim \VVV_1}			&\text{if } \VVV_1=\VVV_2,\\
    0											&\text{if }\VVV_1\neq\VVV_2
    \end{cases}
\]
for all $\VVV_1, \VVV_2\in\Ufrak^G$.
\end{enumerate}
\end{lemma}
\begin{proof}
The distributions $J_{G, \signature}^G(\VVV, \cdot)=\prod_{v\in S_{\infty}}J_{G, v}^G(\VVV, \cdot)$, $\VVV\in\Ufrak^G$, are linearly independent over 
$C_c^{\infty}(G(\R^{\signature})^1)$ (cf. the proof of \cite[Theorem 8.1]{Ar85}) so that we can find functions $f^{\VVV}_{\signature}\in C_{c}^{\infty}(G(\R^{\signature})^1)$ which are $\cpt_{\signature}$-conjugation invariant, and such that
$
J_{G, \signature}^G(\VVV_1, f^{\VVV_2})
$
equals $1$ if $\VVV_1=\VVV_2$, and equals $0$ if $\VVV_1\neq\VVV_2$.
Put 
\[
\tilde{\Xi}_{\signature}=\bigcup_{\VVV\in\Ufrak^G} \supp f^{\VVV}_{\signature}
\]
and let $s>0$ be the smallest number such that 
$
\tilde{\Xi}_{\signature}\subseteq \{X\in \Mat_{n\times n}(\R^{\signature}) \mid \|X\|\leq s\}=\Xi_{\signature}^+
$.
Then $\tilde{\Xi}_{\signature}\subseteq \Xi_{\signature}\cap G(\R^{\underline{r}})^1\subseteq G(\R^{\signature})^1$, and the functions $f^{\VVV}_{\signature}$, $\VVV\in\Ufrak^G$, are by construction $\cpt_{\signature}$-conjugation invariant elements of $C_{\Xi_{\signature}}^{\infty}(G(\R^{\signature})^1)$  satisfying $J_{G,\signature}^G(\VVV_1, f_{\signature}^{\VVV_2})=0$ if $\VVV_1\neq \VVV_2$, and $J_{G,\signature}^G(\VVV_1, f_{\signature}^{\VVV_1})=1$. 
If $v\in S_{\text{fin}}$ is a non-archimedean place, we have
\[
J_{G, v}^G(\VVV, \One_{\cpt_{v}})=\vol(\Ubf_{G, \VVV}(\OOO_v))=\vol(\OOO_v)^{\dim \Ubf_{G, \VVV}}.
\]
 Since $\dim \VVV=2\dim \Ubf_{G, \VVV}$ and $\vol(\OOO_v)=\N(\DDD_v)^{-1/2}$, the second assertion of the lemma follows.
\end{proof}

The second property of the functions given in the lemma is responsible for the separation of the coefficients belonging to different unipotent classes on the left hand side of \eqref{reduction_to_ind_step}.
We fix once and for all  a family of functions $\CCC_{\signature}=\{f^{\VVV}_{\signature}\}_{\VVV\in\Ufrak^G}$ as in the lemma.
Note that the sets $\Xi_{\signature}$ and $\Xi_{\signature}^+$ are $\cpt_{\signature}$-conjugation invariant in the sense that 
$k^{-1}\Xi_{\signature} k=\Xi_{\signature}$
and
$k^{-1}\Xi^+_{\signature}k=\Xi^+_{\signature}$
for all $k\in \cpt_{\signature}$.

\section{Estimates for orbital integrals}\label{section_weighted_orb_int}
As a  first step towards controlling the right hand side of \eqref{reduction_to_ind_step}, we prove an upper bound for the absolute value of the local weighted orbital integrals at the non-archimedean places.
\begin{lemma}\label{estimate_for_local_weighted_orb_int}
 Let $n, d\in\Z_{\geq1}$. Let $F$ be a number field of degree $d$ and let $v$ be a non-archimedean place of $F$. 
 For $L\in\LLL$ and $\VVV\in\Ufrak^L$, set
\[
\nu_v(\VVV,L)=\vol\left(\Ubf_{L, \VVV}(\OOO_v)\right).
\]  
Then for $r=r_L=\dim \aaa_L^G$ we have
\[
\frac{\left| J_L^G(\VVV, \One_{\cpt_v})\right|}{\nu_v(\VVV,L)}
\ll_{n, d}  \left|\frac{\zeta_{F, v}^{(r)}(1)}{\zeta_{F, v}(1)}\right|	
\]
if $L\neq G$, and if $L=G$, 
\[
 \frac{\left| J_G^G(\VVV, \One_{\cpt_v})\right|}{\nu_v(\VVV,G)}
=1.
\]
\end{lemma}
\begin{proof}
If $L=G$, then
\[
J_{G}^G(\VVV, \One_{\cpt_v})= \int_{\Ubf_{G, \VVV}(F_v)} \One_{\cpt_v}(x)dx
= \vol(\Ubf_{G, \VVV}(F_v)\cap {\cpt_v})=\nu_v(\VVV,G)
\]
so that we can assume $L\neq G$ from now on. In particular, we may assume $\dim \Ubf_{L, \VVV}\geq1$, since $\dim \Ubf_{L, \VVV}=0$ implies $L=G$, $\Mbf_{M,\VVV}=G$.
By Lemma \ref{lemma_first_est_weight_orb_int} it suffices to estimate the non-negative integrals
\[
\Lambda_{\VVV,\varpi}:=\int_{\Ubf_{L, \VVV}(\OOO_v)} \left|\log\left\|W_{\varpi}^{\VVV}(x)\right\|_{v}\right|^r dx
\]
for  $\varpi\in\Omega_L$ and $\VVV\in\Ufrak^G$.

We first show that there exists  $c>0$ (uniform in $F$ and $v$) such that
\begin{equation}\label{first_est_to_show}
\frac{ \Lambda_{\VVV,\varpi}}{\nu_v(\VVV,L)}
\leq \begin{cases}
      c									&\text{if } v\in \Pi_F,\\
      c\left|\zeta_{F, v}^{(r)}(1)\right|				&\text{if } v\not\in \Pi_F,~r>0,\\
      1									&\text{if } v\not\in \Pi_F,~r=0.
     \end{cases}
\end{equation}
Write 
\[
\Lambda_{\VVV,\varpi}
=\Lambda_{\VVV,\varpi}^+ +\Lambda_{\VVV,\varpi}^-
\]
with
\begin{equation}\label{part_of_weight_orb_int1A}
\Lambda_{\VVV,\varpi}^-
:= \int_{x\in \Ubf_{L, \VVV}(\OOO_v),~\left\|W_{\varpi}^{\VVV}(x)\right\|_v<1} \left|\log \left\|W_{\varpi}^{\VVV}(x)\right\|_v\right|^r dx,
\end{equation}
and  $\Lambda_{\VVV,\varpi}^+$ defined similarly as the integral over those $x\in \Ubf_{L, \VVV}(\OOO_v)$ with $\left\|W_{\varpi}^{\VVV}(x)\right\|_v\geq 1$.

First, we bound $\Lambda_{\VVV,\varpi}^+$ by
\begin{equation}\label{part_of_weight_orb_int1b}
\Lambda_{\VVV,\varpi}^+
\leq \sup_{\substack{x\in \Ubf_{L, \VVV}(\OOO_v) \\ \left\|W_{\varpi}^{\VVV}(x)\right\|_v\geq 1}}
\left|\log\left\|W_{\varpi}^{\VVV}(x)\right\|_v\right|^r
\int_{\Ubf_{L, \VVV}(\OOO_v)}dx
\begin{cases}
	    \leq c \nu_v(\VVV,L) 									&\text{if }v\in\Pi_F,\\
	    =0												&\text{if }v\not\in\Pi_F,~r>0,\\
	    =\nu_v(\VVV,L)										&\text{if }v\not\in\Pi_F,~r=0,
\end{cases}
\end{equation}
for a suitable constant $c>0$ depending only on $n$. (Recall that for $v\not\in \Pi_F$,  $\left\|W_{\varpi}^{\VVV}(x)\right\|_v\leq 1$ for all $x\in \Ubf_{L, \VVV}(\OOO_v)$ by \eqref{inequality_for_weight_polynomial}).

We now estimate $\Lambda_{\VVV, \varpi}^-$.
The proof of \cite[Lemma 7.1]{Ar88a} gives $k, C>0$, both independent of $F$ and $v$ (but depending on the set of polynomials $\{W_{\varpi}^{\VVV}\}$) such that
\[
\vol\big(\big\{x\in \Ubf_{L, \VVV}(\OOO_v)\mid \left\|W_{\varpi}^{\VVV}(x)\right\|_v\leq\eps\big\}\big)\\
\leq C\eps^{1/k}\nu_v(\VVV,L)
\]
for any $\eps\in (0,1]$. This implies for every $l\in\Z_{\geq0}$ that
\[
\vol\big(\big\{x\in \Ubf_{L, \VVV}(\OOO_v)\mid \left\|W_{\varpi}^{\VVV}(x)\right\|_v<q_v^{-l}\big\}\big)\\
\leq Cq_v^{-\lceil (l+1)/k\rceil}\nu_v(\VVV,L),
\]
where $\lceil a\rceil$ denotes the smallest integer $\geq a$. Here we are allowed to replace $q_v^{-(l+1)/k}$ by $q_v^{-\lceil (l+1)/k\rceil}$, because the possible values of the volume on the left hand side are contained in $q_v^{\Z} \nu_v(\VVV,L)$. 
Using this estimate and following \cite[p.260]{Ar88a}, we can bound the integral in \eqref{part_of_weight_orb_int1A} by the sum over $l\geq 0$ of the integrals over the subsets of $\Ubf_{L, \VVV}(\OOO_v)$ given by $\big\|W_{\varpi}^{\VVV}(x)\big\|_v< 2^{-l}$ if $v\in\Pi_F$, and by $\big\|W_{\varpi}^{\VVV}(x)\big\|_v< q_v^{-l}$ if $v\not\in \Pi_F$. 
In this way we obtain
\begin{multline*}
\frac{1}{\nu_v( \VVV,L)}\int_{x\in \Ubf_{L, \VVV}(\OOO_v),~\left\|W_{\varpi}^{\VVV}(x)\right\|_v< 1} \left|\log \left\|W_{\varpi}^{\VVV}(x)\right\|_v\right|^r dx\\
\leq
\begin{cases}
C\sum_{l=0}^{\infty} 2^{-l/k} (\log 2^{l+1})^r =:C_1									&\text{if }v\in\Pi_F,\\
C\sum_{l=0}^{\infty} q_v^{-\lceil(l+1)/k\rceil}(\log q_v^{(l+1)})^r							&\text{if }v\not\in\Pi_F, ~r>0,\\
0															&\text{if }v\not\in\Pi_F, ~r=0,
\end{cases}
\end{multline*}
with  $C_1>0$ only depending on $C$ and $k$.
Now if $v\not\in\Pi_F$ and $r>0$ we can compute
\begin{multline*}
\sum_{l=0}^{\infty} q_v^{-\lceil(l+1)/k\rceil}(\log q_v^{(l+1)})^r
=\sum_{m=0}^{\infty}\sum_{a=0}^{k-1} q_v^{-\lceil(km+a+1)/k\rceil}(\log q_v^{(km+a+1)})^r\\
=\sum_{m=0}^{\infty}q_v^{-m-1} (\log q_v^{m+1})^r \sum_{a=0}^{k-1} \big(\frac{km+a+1}{m+1}\big)^r
\leq C_2 \sum_{m=0}^{\infty}q_v^{-m-1} (\log q_v^{m+1})^r 
=C_2 \big|\zeta_{F,v}^{(r)} (1)\big|
\end{multline*}
for $C_2>0$ a constant depending only on $k$ and $r$.

Summarising, we get
\begin{equation}\label{part_of_weight_orb_int1c}
\frac{\Lambda_{\VVV,\varpi}^-}{\nu_v(\VVV,L)}
\leq 
\begin{cases}
C_1 			 												&\text{if }v\in\Pi_F,\\
C_2\big|\zeta_{F, v}^{(r)}(1)\big|											&\text{if }v\not\in\Pi_F,~r>0,\\
0															&\text{if }v\not\in\Pi_F,~r=0.
\end{cases}
\end{equation}
Taking \eqref{part_of_weight_orb_int1b} and \eqref{part_of_weight_orb_int1c} together yields
\[
\frac{\Lambda_{\VVV, \varpi}}{\nu_v(\VVV,L)}
\leq 
\begin{cases}
C_1'									&\text{if }v\in\Pi_F,\\
C_2\Big|\zeta_{F, v}^{(r)}(1)\Big| 					&\text{if }v\not\in\Pi_F,~r>0,\\
1									&\text{if }v\not\in\Pi_F,~r=0
\end{cases}
\]
for some $C_1'>0$ which is independent of $F$ and $v$ so that \eqref{first_est_to_show} is proved.
Since 
$|\Pi_F|\leq d |\Pi|$, 
and
$|\zeta_{F, v}(1)|\leq 2$
 for all $v$, the assertion of the lemma follows.
\end{proof}

\begin{cor}\label{estimate_glob_orb_int_special_test_fcts}
 Let $n, d\in \Z_{\geq1}$ and let $\signature$ be  a signature of degree $d$. Let $F$ be a number field of signature $\signature$ and $S$ a finite set of places of $F$ containing all archimedean places.
 For $\VVV\in\Ufrak^L$ set 
$\nu_S(\VVV,L)=\prod_{v\in S_{\text{fin}}} \nu_v(\VVV,L)$.
Then for all $L\in \LLL$, $\VVV_1\in\Ufrak^L$, and $\VVV_2\in\Ufrak^G$ we have 
\[
\frac{\big|J_L^G(\VVV_1, f_{\signature}^{\VVV_2}\otimes\One_{\cpt_{S_{\text{fin}}}})\big|}{\nu_S(\VVV_1,L)}
\ll_{n, d} \sum_{\substack{r_v\in\Z_{\geq0},v\in S_{\text{fin}}: \\ \sum r_v\leq r_L}}\; 
\prod_{v\in S_{\text{fin}}} \bigg|\frac{\zeta_{F, v}^{(r_v)}(1)}{\zeta_{F, v}(1)}\bigg|			
\]
if $\III_L^G\VVV_1=\VVV_2$, and 
\[
J_L^G(\VVV_1, f_{\signature}^{\VVV_2}\otimes\One_{\cpt_{S_{\text{fin}}}})=0
\]
if $\III_L^G\VVV_1\neq\VVV_2$.
 Here $r_L=\dim\aaa_L^G$ and $f_{\signature}^{\III_L^G\VVV_2}\in\CCC_{\signature}$ is one of the finitely many test functions defined in \S \ref{subsection_test_fcts}.
\end{cor}
\begin{proof}
 This is an immediate consequence of the splitting formula \eqref{red_to_loc_case} and Lemma \ref{estimate_for_local_weighted_orb_int} after noting that there are only finitely many signatures $\signature$ of degree $d$, and for any $n$ and $\signature$, there are only finitely many fixed test functions $f^{\VVV_2}_{\signature}$, i.e.,  they do not depend on the specific field $F$ but only on its signature.
\end{proof}

This finishes the estimates for the weighted orbital integrals. The next sections will be occupied with the estimation of the global distribution $J_{\text{unip}}^G$.

 \section{Reduction theory for $\GL_n$ over number fields}\label{section_reduction_theory}
To find an upper bound for $|J_{\text{unip}}^G(f)|$ in \S \ref{section_unip_contr}, we will need to replace an integration over the quotient $\GL_n(F)\backslash \GL_n(\A_F)^1$ by an integration over a Siegel set for $\GL_n(\A_F)^1$. To obtain a bound which depends explicitly on $F$, we first need to make the reduction theory over $F$ sufficiently explicit in the sense of Proposition \ref{red_theory_gln_number_field} below. 

Let $n\geq2$  and let $\signature$ be a signature of degree $d$. Let $F$ be of signature $\signature$ and  absolute discriminant $D_F$.
For $T_1\in -\aaa_0^+$ write
\[
A_P^{G, \infty}(T_1)
=\{a\in A_{0}^{G, \infty}\mid \alpha(H_0(a)-T_1)\geq 0 ~ \forall \alpha\in\Delta_P\}
\]
and let $A_0^{G, \infty}(T_1)=A_{P_0}^{G, \infty}(T_1)$.

By reduction theory there exists $T_1\in -\aaa_0^+$ (see e.g.\ \cite[Theorem 4.15]{PlRa94}) such that
\begin{equation}\label{red_theory_equation_numb_field}
\GL_n(\A_F)^1
= \GL_n(F) P_0(\A_F)^1 A_{0}^{\GL_n, \infty}(T_1) \cpt.
\end{equation}
\begin{proposition}\label{red_theory_gln_number_field}
The equality \eqref{red_theory_equation_numb_field} holds for
\begin{equation}\label{definition_of_T_1}
T_1=T_1^F:=\log c_F \sum_{\varpi\in \widehat{\Delta}_0}\varpi^{\vee}= \log c_F\rho^{\vee}\in -\aaa_0^{G,+}\subseteq -\aaa_0^+
\end{equation}
where
\[
c_F:=\left(\frac{\pi}{4}\right)^d D_F^{-1}.
\]
Furthermore,
\begin{equation}\label{estimates_for_T_1}
\left\|T_1^F\right\|
\ll_{n, d} 1+ \log D_F
~~~~~~~~~~~
\text{ and }
~~~~~~~~~~~
1\leq e^{-\alpha(T_1^F)}\ll_{n, d}  D_F^{\alpha(\rho^{\vee})}
\end{equation}
for all positive roots $\alpha\in\Sigma^+(P_0, A_0^{\infty})$, where the euclidean norm $\|\cdot\|$ on $\aaa_0^G$ is given by $\|X\|=\sqrt{X_1^2+\ldots X_n^2}$ if $X=(X_1, \ldots, X_n)\in\aaa_0^G$.
\end{proposition}
Note that the validity of the estimates in \eqref{estimates_for_T_1} is a direct consequence of the definition of $T_1^F$ so that it will suffice to show that \eqref{red_theory_equation_numb_field} holds for $T_1=T_1^F$.
This will be an easy consequence of an upper bound for the adelic Hermite constant of an $F$-lattices from \cite{Ic97} combined with the usual method of proving \eqref{red_theory_equation_numb_field} by induction on $n$ for $\GL_n$ over $\Q$ (in which case the optimal $c_{\Q}=\frac{\sqrt{3}}{2}$ is larger than our $c_{\Q}=\frac{\pi}{4}$). For this second step we will follow \cite[\S 4.2]{PlRa94}.
We need to introduce some further notation.
For a place $v$ of $F$, denote by $\|x\|_v$ the usual vector norm on $F_v^n$, that is, if $x=(x_1, \ldots, x_n)\in F_v^n$, then
\begin{equation}\label{definition_norm}
\|x\|_v=\begin{cases}
         (x_1^2+\ldots +x_n^2)^{\frac{1}{2}}					&\text{if } v \text{ is real,}\\
	  x_1\overline{x_1}+\ldots + x_n\overline{x_n}				&\text{if } v \text{ is complex,}\\
	  \max\{|x_1|_v, \ldots, |x_n|_v\}					&\text{if } v \text{ is non-archimedean}.
        \end{cases}
\end{equation}
If $x=(x_v)_{v}\in \A_F^n$ with $\|x_v\|_v=1$ for all but finitely many $v$ and $\|x_v\|_v\neq0$ for all $v$, then let $\|x\|_{\A_F}$ be the product $\prod_v \|x_v\|_v$. 

An $F$-module $\Lambda\subseteq\A_F^n$ of the form $\Lambda= F^n A_{\Lambda}$ for some $A_{\Lambda}\in\GL_n(\A_F)$ is called an $F$-lattice.
Let $\gamma_n(F)$ denote the adelic Hermite constant over $F$. By definition, $\gamma_n(F)$ is the smallest real number such that for any $F$-lattice $\Lambda$ there exists $\xi\in\Lambda$, $\xi\neq 0$, such that (cf.\ also \cite{Th98} for the definitions)
\[
\|\xi\|_{\A_F}\leq \gamma_n(F)^{\frac{d}{2}} \left|\det A_{\Lambda}\right|_{\A_F}^{\frac{1}{n}}.
\]
By \cite[Theorem 1]{Ic97}, the Hermite constant for $n=2$ satisfies the upper bound
\begin{equation}\label{bound_for_2_diml_Hermite}
 \gamma_2(F)^{\frac{d}{2}}
\leq \left(\frac{4}{\pi}\right)^{\frac{d}{2}}\Delta_F.
\end{equation}
Defining $c_F$ as in the proposition above, we have
\[
c_F^{-1}\geq  \gamma_2(F)^{d}.
\]
We now first prove the proposition for the case $n=2$.
\begin{lemma}\label{prop_red_theory_over_number_field}
 We have 
\[
 \GL_2(\A_{F})^1
= \GL_2(F) P_0(\A_F)^1 A_{0}^{\GL_2,\infty}(T_1^F) \cpt
\]
with $T_1^F$ defined as in \eqref{definition_of_T_1}.
\end{lemma}

\begin{proof}
Let $g\in \GL_2(\A_F)^1$. We need to show that there exists $\gamma \in \GL_2(F)$ with 
$
e^{\alpha(H_0(\gamma g))}\geq c_{F}
$
for the root $\alpha\in\Delta_0=\{\alpha\}$.
By using the right $\cpt$-invariance of $H_0$ and Iwasawa decomposition, we may assume that 
$
g=\left(\begin{smallmatrix}a_1 & x \\ 0& a_2^{-1}\end{smallmatrix}\right)\in P_0(\A_F)
$
with
$
|a_1|_{\A_F}=|a_2|_{\A_F}
$.

If 
$
\gamma=\left(\begin{smallmatrix} 0 & 1 \\ 1 & 0\end{smallmatrix}\right) \left(\begin{smallmatrix} t & y \\ 0 & 1 \end{smallmatrix}\right)
$
with $t\in F^{\times}$ and $y\in F$,
then 
\[
e^{\alpha(H_0(\gamma g))}
=\| (t a_1, t x+ a_2^{-1} y) \|_{\A_F}^{-2}
=\| t(a_1, x) + y(0, a_2^{-1})\|_{\A_F}^{-2}
=\| (t, y) g\|_{\A_F}^{-2}
\]
and
$
e^{\alpha(H_0(g))}
= \|(0, a_2^{-1})\|_{\A_F}^{-2} =|a_2|_{\A_F}^{2}=|a_1|_{\A_F}^2
$.
The vectors $\xi_1=(a_1, x)$ and $\xi_2=(0, a_2^{-1})$ span an $F$-lattice in $\A_F^2$ of full rank. 
By definition of  $\gamma_2(F)$ 
there exists $(s_1, s_2)\in F^2$, $(s_1, s_2)\neq (0,0)$, with
\[
\|s_1\xi_1+s_2\xi_2\|_{\A_F}^2\leq \gamma_{2}(F)^{d}
\]
Thus by \eqref{bound_for_2_diml_Hermite} there exists $(s_1, s_2)\in F^2\minus\{0\}$ with
$
\|s_1\xi_1+ s_2\xi_2\|_{\A_F}^{-2}\geq c_{F}
$.
If $s_1=0$, then 
$
c_{F}\leq \|s_1\xi_1+s_2\xi_2\|_{\A_F}^{-2}= |a_2|_{\A_F}^2=e^{\alpha(H_0(g))}
$
and we are done. 
If $s_1\neq 0$, then 
$
\gamma=\left(\begin{smallmatrix}0&1\\1&0\end{smallmatrix}\right)\left(\begin{smallmatrix}s_1&s_2\\0&1\end{smallmatrix}\right)
$ 
yields 
$
c_{F}\leq \|s_1\xi+ s_2\xi_2\|_{\A_F}^{-2}=e^{\alpha(H_0(\gamma g))}
$.
This finishes the proof of the lemma.
\end{proof}
\begin{rem}\label{remark_on_minimum}
The proof also implies that if for fixed $g\in\GL_n(\A_F)$ the map
$
F^2\minus\{0\}\ni z\mapsto \|z g\|_{\A_F}
$
attains its minimum at $z=z_0$ (such a $z_0\neq 0$ exists since $F^2\subseteq \A_F^2$ is discrete), then $e^{\alpha(H_0(\gamma g))}\geq c_F$ for $\gamma\in \GL_2(F)$ satisfying $(0,1)\gamma =z_0$. 
\end{rem}

\begin{proof}[Proof of Proposition \ref{red_theory_gln_number_field}]
We prove the proposition by induction on $n$ with Lemma \ref{prop_red_theory_over_number_field} covering the initial case $n=2$.
Thus we let $n\geq 3$ and assume that the proposition holds for $n-1$.
Let $g\in \GL_n(\A_F)^1$ and consider the map
$
F^n\minus 0 \ni z\mapsto \| z g\|_{\A_F}
$.
Since $F^n$ is discrete in $\A_F^n$, this map attains its (strictly positive) infimum at some vector $z_0\in F^n$, $z_0\neq 0$. Let $\gamma\in \GL_n(F)$ be such that $e_n \gamma=z_0$ where $e_n=(0, \ldots, 0, 1)\in F^n$, and let $g'=\gamma g$. Let $g'=b' k'$ be the Iwasawa decomposition of $g'$ with
\[
b'=\begin{pmatrix} a_1' 	& 	u_{1,2}' 	& 	\ldots 		& 		u_{1,n}' 		\\
		  			\;		& 	a_2'    	&  		\;	& 	\vdots	 		\\
		   			\; 		& 	\;		& 	\ddots 	& 	u_{n-1, n}' 	\\
		  			\;  	& 	\; 			&		\;	& 	a_n'
  \end{pmatrix}
\]
and $k'\in\cpt$. Because of the right $\cpt$-invariance of the vector norm, we have 
$
\| e_n g'\|_{\A_F}
= \|e_n b'\|_{\A_F}
= |a_n'|_{\A_F}
$.
By induction hypothesis there is $\tilde{\gamma}\in\GL_{n-1}(F)$ satisfying
$
e^{\alpha_i(H_0(\gamma' b'))}=\left| \frac{a_i'}{ a_{i+1}'}\right|_{\A_F}\geq c_F
$
for all $i=1, \ldots, n-2$ and $\gamma'=\diag(\tilde{\gamma}, 1)\in\GL_{n-1}(F)$.

Let $\gamma' b'= b'' k''$ be the Iwasawa decomposition with 
\[
b''=\begin{pmatrix} a_1'' & u_{1,2}'' & \ldots & u_{1,n}'' 	\\
		  \;  & a_2''    &  \;	& \vdots  	\\
		   \; & \;	& \ddots & u_{n-1, n}'' 	\\
		  \;  & \; 	&	& a_n''
  \end{pmatrix}
\]
and $a_n''=a_n'$.
Then 
$
\| e_n \gamma' b'\|_{\A_F}
= |a_n''|_{\A_F}
=|a_n'|_{\A_F}
$
so that
$e_n$ is a minimising vector for the map $F^n\minus\{0\}\ni z\mapsto \|z \gamma' g'\|_{\A_F}$. 
Thus $e_2^2=(0,1)\in F^2$ is a minimising vector for the map 
\[
F^2\minus \{0\} \ni z_2 \mapsto \left\| z_2 	\left(\begin{smallmatrix} 
										      a_{n-1}'' 			& 	u_{n-1, n)}'' 	\\ 
											0				& 	a_n'
							\end{smallmatrix}\right)\right\|_{\A_F}.
\]
Therefore Remark \ref{remark_on_minimum} implies that 
$
\left| \frac{a_{n-1}''}{ a_n'} \right|_{\A_F}\geq c_{F}
$.
Hence for any $\alpha\in\Delta_0$ we get 
\[
e^{\alpha(H_0(\gamma'\gamma g))}
= e^{\alpha(H_0(b''))}\geq c_{F}
\]
which proves the proposition.
\end{proof}

Recall the following well-known and easy fact: If $N\subseteq U_0(\A_F)$ is a compact set, then
$
\bigcup_{a\in A_0^{G, \infty}(T_1^F)} a^{-1} Na
$
 is again compact. 
Further recall from \S \ref{section_fundamental_domains} the definition of the compact set $\NNN$ containing a fundamental domain for $U_0(F)\hookrightarrow U_0(\A_F)$.
Hence
$\bigcup_{a\in A_0^{G, \infty}(T_1^F)} a^{-1}\NNN a$
is again compact, and we can make this more precise as follows.

\begin{lemma}\label{conjugated_cpt_unipotent}
 The set $\bigcup_{a\in A_0^{G, \infty}(T_1^F)} a^{-1}\NNN a$ is contained in the compact set $\NNN^{(1)}:=\NNN^{(1)}_{\infty}U_0(\widehat{\OOO}_F)$, where $\widehat{\OOO}_F=\prod_{v\not\in S_{\infty}}\OOO_v$ denotes the completion of $\OOO_F$ at all finite places, and $\NNN^{(1)}_{\infty}\subseteq U_{0}(F_{\infty})$ is given by
\[
\NNN^{(1)}_{\infty}=\{u\in U_0(F_{\infty})\mid \forall ~1\leq i,j\leq n: ~ \|u_{i,j}\|\leq c_1 D_F^{c_2}\}
\]
for suitable constants $c_1=c_1(n, d), c_2=c_2(n, d)$ depending only on $n$ and  $d$. Here the norm $\|\cdot\|$ on $F_{\infty}\simeq \R^{\signature}$ is defined in  \S \ref{section_inner_products} and only depends on the signature $\signature$.
\end{lemma}
\begin{proof}
 The assertion is a direct consequence of the definition of $\NNN$ in Lemma \ref{fundamental_domains_in_TU} and the properties of $T_1^F$ given in \eqref{estimates_for_T_1}.
\end{proof}

\section{Estimates for the unipotent contributions; proof of Theorem \ref{estimate_for_coeff}}\label{section_unip_contr}
Recall the definition of the compact set $\Xi_{\signature}$ from \S \ref{subsection_test_fcts}, and the definition of the norm $\|\cdot\|_k$ on $C_{\Xi_{\signature}}^{\infty}(G(\R^{\signature})^1)$ from \S \ref{subsection_orb_int}.
\begin{proposition}\label{bounding_unip_contr}
Let $ n,d\in\Z_{\geq 1}$ and let $\signature$ be a signature of degree $d$.
Then there are $m\in\R_{\geq0}$ and $k\in\Z_{\geq0}$ such that 
for any number field $F$ of signature $\signature$ we have 
\begin{equation}\label{inequ_unip_contr}
 \left| J_{\text{unip}}(f_{\infty}\otimes\One_{\cpt_{\text{fin}}})\right|
\ll_{n, d} \vol(F^{\times}\backslash \A_F^1)^n D_{F}^{m} 
\left\|f_{\infty}\right\|_{k}
\end{equation}
for all $f_{\infty}\in C_{\Xi_{\signature}}^{\infty}(G(\R^{\signature})^1)$. 
\end{proposition}

Before we prove this proposition below, we finish the proof of our main result.
\begin{proof}[Proof of Theorem \ref{estimate_for_coeff}]
Let $F$ be a number field of  signature $\signature$.
We prove the theorem by induction on the semisimple rank $\dim\aaa_0^M$ of $M$.
The initial case $M=T_0$ is trivial, since $\Ufrak^{T_0}=\{\One^{T_0}\}$ and by \cite[Corollary 8.5]{Ar85} we have
$
a^{T_0}(\One^{T_0}, S)
=\vol(T_0(F)\backslash T_0(\A_F)^1)
=\vol(F^{\times}\backslash\A_F^1)^n
$. 
Therefore 
$
a^{T_0}(\One^{T_0}, S)=\left(\lambda_{-1}^F\right)^{n}\ll_{n, d, \eps} D_F^{\eps}
$
for any $\eps>0$ by Proposition \ref{brauer_siegel}.

Now let $M\in \LLL$, $M\neq T_0$, and assume that the theorem holds for all $L\in\LLL$ with $L\subsetneq M$. By \eqref{factoring_coeff} it suffices to assume that $M=G$ and that the theorem holds for any proper Levi subgroup $L\subsetneq G$.

Recall the definition of the test functions $f^{\VVV}_{\signature}\in C_{\Xi_{\signature}}^{\infty}(G(\R^{\signature})^1)$, $\VVV\in\Ufrak^G$, from \S \ref{subsection_test_fcts} and their properties from Lemma \ref{prop_spec_test_fcts}. 
In particular,  
$\left\|f^{\VVV}_{\signature}\right\|_k\ll_{n, d,k} 1$,
 since  $\BBB_k\subseteq \UUU(\ggG_{\signature}(\C))_{\leq k}$ and 
$\CCC_{\signature}=\{f^{\VVV}_{\signature}\}_{\VVV\in\Ufrak^G}$
were fixed once and for all.
By Corollary \ref{estimate_glob_orb_int_special_test_fcts},
\[
\Big| J_{L}^G(\VVV_1, f^{\VVV_2}_{\signature}\otimes\One_{\cpt_{S_{\text{fin}}}})\Big|
\ll_{n, d}
\sum_{\substack{ r_v\in\Z_{\geq0},v\in S_{\text{fin}}: \\ \sum r_v\leq r_L}}\;
\prod_{v\in S_{\text{fin}}} \bigg|\frac{\zeta_{F, v}^{(r_v)}(1)}{\zeta_{F, v}(1)}\bigg|
\]
for all $\VVV_1\in\Ufrak^L$, $\VVV_2\in\Ufrak^G$,  where $r_L=\dim\aaa_L^G$. Note that
$\dim\aaa_0^L
=n-1-r_L$.

Using this and our induction hypotheses, we  get
\begin{multline*}
 \bigg|\sum_{\substack{L\in\LLL \\ L\neq G}}\frac{|W^L|}{|W^G|}\sum_{\VVV\in\Ufrak^L} 
a^L(\VVV, S)J_L^G(\VVV, f^{\VVV_2}_{\signature}\otimes\One_{\cpt_{S_{\text{fin}}}})\bigg|\\
\ll_{n,d} \sum_{\substack{L\in\LLL \\ L\neq G}}\left|\Ufrak^G\right|
\Bigg( D_F^{\kappa} \sum_{\substack{s_v\in\Z_{\geq0},v\in S_{\text{fin}}: \\ \sum s_v \leq n-r_L-1}}\;
\prod_{v\in S_{\text{fin}}} \bigg|\frac{\zeta_{F, v}^{(s_v)}(1)}{\zeta_{F, v}(1)}\bigg| \Bigg)
\Bigg(\sum_{\substack{ r_v\in\Z_{\geq0},v\in S_{\text{fin}}: \\ \sum r_v\leq r_L}}\;
\prod_{v\in S_{\text{fin}}} \bigg|\frac{\zeta_{F, v}^{(r_v)}(1)}{\zeta_{F, v}(1)}\bigg|\Bigg)
\end{multline*}
for all $\VVV_2\in\Ufrak^G$, where the constant $\kappa=\kappa(n-1, d)\geq0$ exists by the induction hypothesis.
 Using Lemma \ref{estimate_for_local_zeta} and $1\leq |\zeta_{F,v}(1)|\leq 2$ to estimate the product of two logarithmic derivatives of local zeta functions, the above is bounded by
\begin{equation}\label{eq:est1}
 \ll_{n, d} D_F^{\kappa}
\sum_{\substack{s_v\in \Z_{\geq0}, v\in S_{\text{fin}}: \\ \sum s_v\leq n-1}}\;
\prod_{v\in S_{\text{fin}}} \bigg|\frac{\zeta_{F, v}^{(s_v)}(1)}{\zeta_{F, v}(1)}\bigg|.
\end{equation}
We identify the function $f^{\VVV_2}_{\signature}\otimes\One_{\cpt_{S_{\text{fin}}}}\in C_c^{\infty}(G(F_S)^1)$ with $f^{\VVV_2}_{\signature}\otimes\One_{\cpt_{\text{fin}}}\in C_c^{\infty}(G(\A_F)^1)$. Then by Proposition \ref{bounding_unip_contr} there are $c=c(n, d), k=k(n, d)\geq0$ such that 
\[
|J_{\text{unip}}(f^{\VVV_2}_{\signature}\otimes\One_{\cpt_{\text{fin}}})|
\ll_{n, d} D_F^{c}\vol(F^{\times}\backslash \A_F^1)^n
\left\|f^{\VVV_2}_{\signature}\right\|_k
\ll_{n, d} D_F^{c}\vol(F^{\times}\backslash \A_F^1)^n.
\]
Combining this with \eqref{eq:est1}, the right hand side of \eqref{reduction_to_ind_step} for $L=G$ and $f=f^{\VVV_2}_{\signature}\otimes\One_{\cpt_{\text{fin}}}$ is bounded by
\[
\ll_{n, d} D_F^{\kappa'}
\sum_{\substack{s_v\in \Z_{\geq0}, v\in S_{\text{fin}}: \\ \sum s_v \leq n-1}}\;
\prod_{v\in S_{\text{fin}}} \bigg|\frac{\zeta_{F, v}^{(s_v)}(1)}{\zeta_{F, v}(1)}\bigg|
\]
for some $\kappa'=\kappa'(n, d)>0$.

Now by Lemma \ref{prop_spec_test_fcts}\ref{prop_spec_test_fctsiii}, 
$
J_G^G(\VVV_1, f^{\VVV_2}_{\signature}\otimes\One_{\cpt_{S_{\text{fin}}}})
=\prod_{v\in S_{\text{fin}}}\N(\DDD_v)^{-\frac{1}{4}\dim \VVV_1}
$
for $\VVV_1\in\Ufrak^G$ 
if $\VVV_1=\VVV_2$, and 
$
J_G^G(\VVV_1, f^{\VVV_2}_{\signature}\otimes\One_{\cpt_{S_{\text{fin}}}})=0
$   
if $\VVV_1\neq\VVV_2$.
Hence the absolute value of left hand side of \eqref{reduction_to_ind_step} for the test function $f=f^{\VVV_2}_{\signature}\otimes\One_{\cpt_{S_{\text{fin}}}}$ reduces to 
\[
\bigg(\prod_{v\in S_{\text{fin}}}\N(\DDD_v)^{-\frac{1}{4}\dim \VVV_2}\bigg) \left|a^G(\VVV_2, S)\right|,
\]
and
$
\prod_{v\in S_{\text{fin}}}\N(\DDD_v)^{\frac{1}{4}\dim \VVV_2}
\leq D_F^{\frac{1}{4}\dim \VVV_2}
$. 
This finishes the proof of the theorem.
\end{proof}

\subsection{Reduction theory for the proof of Proposition \ref{bounding_unip_contr}}
To prove the above proposition, we will essentially follow the strategy from \cite{Ar85} for which we need to introduce two more families of distributions depending on a parameter $T\in\aaa_0^+$. These distributions are closely related to $J_{\text{unip}}(f)$ and will be defined in the next section. The purpose of this section is to make some further reduction theoretic properties more explicit in preparation of  the following sections. 

For $T\in\aaa_0$ we define
\[
d(T)=\min_{\alpha\in\Delta_0}\alpha(T).
\]
Note that $T$ is contained in the closed positive Weyl chamber $\overline{\aaa_0^+}$ if and only if $d(T)\geq0$. 
To prove the results in \cite{Ar78,Ar85} Arthur has to assume that the parameter $T\in\aaa_0$ is contained in the positive Weyl chamber and is sufficiently far away from the walls of this chamber in a sense depending on the support of $f$.
This is equivalent  to $d(T)$ being sufficiently large in a sense depending on the support of $f$. Such a point $T$ is called suitably regular in \cite{Ar78,Ar85}.  We need to make the property \emph{sufficiently far away}, or \emph{suitably regular}, more explicit, see Definition \ref{def_suit_reg} below.

For $T\in \aaa_0^+$ let
\[
A_{0}^{G, \infty}(T_1^F, T)
=\{a\in A_{0}^{G, \infty}(T_1^F) \mid \forall \varpi\in \widehat{\Delta}_0 :~ \varpi(H_0(a)-T)\leq 0 \}\subseteq A_0^{G,\infty}(T_1^F),
\]
and define $A_0^{P, \infty}(T_1^F, T)$ with respect to $M_P$ instead of $G$ analogously.
Recall the definition of the truncation function
\[
F^P(\cdot, T)
: ~ G(\A_F)^1\longrightarrow \C
\]
from \cite[p. 941]{Ar78}. 
It is the characteristic function of all 
\[
g=\delta pak\in G(F) P_0(\A_F)^1 A_0^{G, \infty}(T_1^F)\cpt= G(\A_F)^1
\]
 satisfying 
\begin{itemize}
\item $\forall\alpha\in\Delta_0^P$: $\alpha(H_0(a)-T_1^F)>0$, and 
\item  $\forall\varpi\in\widehat{\Delta}_0^P$: $\varpi(H_0(a)-T)\leq 0$.
\end{itemize}
In particular,  $F^P(\cdot, T)$ is left invariant under $M_P(F) A_P^{G, \infty} U_P(\A_F)$ and right invariant under $\cpt$. If $P=G$, then $F(\cdot, T):=F^G(\cdot, T)$ is the characteristic function of $G(F) P_0(\A_F)^1 A_0^{G, \infty}(T_1^F, T)\cpt$.

Let $P_1, P_2\in\FFF_{\text{std}}$ be two standard parabolic subgroups, $P_1\subseteq P_2$, and write $\aaa_1^2=\aaa_{P_1}^{P_2}$, $\aaa_1=\aaa_{P_1}$. 
We define the function $\sigma_1^2=\sigma_{P_1}^{P_2}:\aaa_0\longrightarrow \C$  as the characteristic function of all $H\in\aaa_0$ satisfying
\begin{itemize}
\item $\forall \alpha\in\Delta_1^2:~\alpha(H_1)>0$,

\item $\forall\alpha\in\Delta_1\minus\Delta_1^2:~\alpha(H_1)\leq 0$, and

\item $\forall\varpi\in \widehat{\Delta}_2:~\varpi(H_1)>0$,
\end{itemize}
where $H_1\in\aaa_1$ denotes the projection of $H$ onto $\aaa_1\subseteq \aaa_0$ (cf. \cite[\S 6]{Ar78}). 
It is clear from this definition that for every $T\in\aaa_0^+$ we have 
\begin{equation}\label{non-vanishing_sigma}
\sigma_1^2(H-T)\neq0 ~ \Rightarrow ~ \forall~\alpha\in \Delta_1^2 :~ \alpha(H_1^2)>0
\end{equation}
for $H_1^2\in\aaa_1^2$ denoting the projection of $H$ onto the subspace $\aaa_1^2\subseteq\aaa_0$.

The following lemma is essentially contained in \cite[pp. 1256-1257]{Ar85} (see also the proof of \cite[Lemma 2.1, 2.2]{La_Schwartz_Space}). It is a technical result that will be used in the proofs of Lemma \ref{suff_reg_def2} and Lemma \ref{estimate_sum_integral_via_lattice_pts} and is important for the explicit definition of a suitably regular point $T$.
If $x\in G(\A_F)^1$, write 
\[
H_0(x)
=\sum_{\beta\in\Delta_0} x_{\beta} \beta^{\vee}
\]
with $x_{\beta}\in \R$.

\begin{lemma}\label{lemma_bound_sum_by_bruhat_decomp}
There exists $C>0$ depending only on $n$ such that for all $A>0$ the following holds.
Let $B\subseteq A_0^{G, \infty}$ be the compact set consisting of all $b\in A_0^{G,\infty}$ satisfying 
 \begin{equation}\label{bound_cpt_set}
 \forall \beta\in\Delta_0:~ \left|b_{\beta}\right|\leq C(A-\beta(T_1^F)).
 \end{equation}
 Then for all $\gamma\in G(F)$ and all $a\in A_0^{G, \infty}(T_1^F)$ with 
\[
a^{-1}\gamma a\in 
\{x\in G(\A_F)^1\mid \forall\beta\in\Delta_0:~ |x_{\beta}|\leq A\}
\]
we can find  $a_1\in B$ and $a_2\in A_{P}^{G, \infty}(T_1^F)$ such that $a=a_1a_2$. Here $P(F)$ is the minimal standard parabolic subgroup containing $\gamma$.
\end{lemma}
\begin{proof}
Let $\gamma\in G(F)$.
According to the Bruhat decomposition write 
$
\gamma=u_1 tn_w u_2
$
with $u_1\in U_0(F)$, $t\in T(F)$, $w\in W^G$, and $u_2\in U_0(F)\cap n_w^{-1} \overline{U_0}(F) n_w$ for $n_w\in G(F)\cap\cpt$ the permutation matrix representing $w$. 
Here $\overline{P_0}\in \FFF$ denotes the parabolic subgroup opposite to $P_0$ and $\overline{U_0}$ its unipotent radical.
Let $P=MU$ be the minimal standard parabolic subgroup containing $\gamma$.
Then $M$ is the smallest standard Levi subgroup containing $n_w$. It is also the smallest standard Levi subgroup containing  $U_0(F)\cap n_w^{-1} \overline{U_0}(F) n_w$.
Let $a\in A_0^{G, \infty}(T_1^F)$ be such that $a^{-1}\gamma a\in \{x\in G(\A_F)^1\mid \forall\beta\in\Delta_0:~|x_{\beta}|\leq A\}$, and write $a=a_1a_2$ with $a_1\in A_0^{P, \infty}(T_1^F)$ and $a_2\in A_{P}^{G, \infty}$. 
Then $a_2$ commutes with $t$, $n_w$, and $u_2$, and we can therefore compute
\begin{align*}
H_0(a^{-1}\gamma a)
= H_0( (a^{-1} u_1 a) (a_1^{-1} t n_w u_2 a_1))
& = -H_0(a_1) + H_0(n_w u_2 a_1)\\
& = -H_0(a_1) + wH_0(a_1) + H_0(n_w a_1^{-1}u_2 a_1).
\end{align*}
Now 
\[
wH_0(a)-H_0(a)
=w H_0(a_1)-H_0(a_1)
=\sum_{\beta>0,~w^{-1}\beta<0} c_{\beta} \beta^{\vee} 
\]
and 
$
H_0(n_w a_1^{-1} u_2 a_1)
=\sum_{\beta>0, ~w^{-1}\beta<0} d_{\beta}\beta^{\vee}
$
with $d_{\beta}\leq 0$ (see \cite[Lemma 6.3]{ClLaLa84}).
By the assumption on $a^{-1}\gamma a$, we get
$
\left|c_{\beta}+d_{\beta}\right|\leq n A
$
for all positive roots $\beta$.
This together with $d_{\beta}\leq 0$ implies $c_{\beta}\geq -nA$.

Write 
$
w H_0(a_1)-H_0(a_1)
=\sum_{\alpha\in\Delta_0} \tilde{c}_{\alpha}\alpha^{\vee}
$.
Then 
$
\tilde{c}_{\alpha}
=\sum_{\beta>0,~w^{-1}\beta<0} n_{\beta, \alpha}c_{\beta} 
$ 
where $n_{\beta, \alpha}\geq 0$ are defined by
$
\beta^{\vee}
=\sum_{\alpha\in\Delta_0}n_{\beta, \alpha} \alpha^{\vee}
$.
In particular, 
\[
\varpi_{\alpha}(w H_0(a_1)-H_0(a_1)) 
=\tilde{c}_{\alpha}\geq -n'A
\]
 for all $\alpha\in\Delta_0$ where $n'>0$ is a suitable constant depending only on the root system because of $c_{\beta}\geq -nA$.

Now write $H_0(a_1)=\sum_{\alpha\in\Delta_0} a_{\varpi_{\alpha}}\varpi_{\alpha}^{\vee}$, where $\varpi_{\alpha}\in \widehat{\Delta}_0$ is such that $\varpi_{\alpha}(\beta^{\vee})=\delta_{\alpha \beta}$ (Kronecker $\delta$) for all $\alpha, \beta\in\Delta_0$. 
As  $a_1\in A_0^{P, \infty}(T_1^F)$, we have  $a_{\varpi_{\alpha}}\geq \alpha(T_1^F)$ for all $\alpha\in \Delta_0^P$.
The minimality of $M$ with respect to $n_w$ implies that
$
\varpi_{\alpha}-w^{-1}\varpi_{\alpha}
=\sum_{\beta\in \Delta_0} m_{\beta, \alpha} \beta^{\vee}
$
with $m_{\beta, \alpha}\geq 0$ and $m_{\alpha,\alpha}>0$ for all $\alpha\in\Delta_0^P$ (see \cite[p. 103]{Ar80}), and these coefficients again depend only on the root system.

Therefore, for any $\alpha\in\Delta_0^P$,
\begin{align*}
n'A
& \geq -\tilde{c}_{\alpha} 
=-\varpi_{\alpha}(w H_0(a_1)-H_0(a_1))
=-\varpi_{\alpha}(w H_0(a)-H_0(a))
= (\varpi_{\alpha}-w^{-1}\varpi_{\alpha})(H_0(a))\\
& =\sum_{\beta\in\Delta_0} m_{\beta, \alpha} \beta(H_0(a))
\geq m_{\alpha, \alpha} \alpha(H_0(a_1))+\sum_{\beta\neq\alpha} m_{\beta, \alpha} \beta(T_1^F)
= m_{\alpha, \alpha} a_{\varpi_{\alpha}}+\sum_{\beta\neq\alpha} m_{\beta, \alpha} \beta(T_1^F),
\end{align*}
since $a\in A_0^{G, \infty}(T_1^F)$. Since $m_{\alpha, \alpha}>0$, this gives an upper bound on $a_{\varpi_{\alpha}}$ for $\alpha\in\Delta_0^P$.
This implies that there exists a $C'>0$ depending only on $n$ such that the  compact set $B'$ defined by the condition \eqref{bound_cpt_set} with respect to $C'$ instead of $C$ contains $a_1$.
Since for any $\alpha\in \Delta_0$ we have
\[
\alpha(T_1^F)
<\alpha(H_0(a))
=\alpha(H_0(a_1))+\alpha(H_0(a_2))
=a_{\varpi_{\alpha}} + \alpha(H_0(a_2)),
\]
every $\alpha(H_0(a_2))$ is bounded from below by a constant depending only on $B'$ and $T_1^F$. Enlarging $C'$ to a constant $C$ if necessary, it follows that  we can write $a_2$ as the product of $a_3a_4$ such that $a_4\in A_P^{G, \infty}(T_1^F)$ and $a_1a_3\in B$ for $B$ defined by \eqref{bound_cpt_set} with respect to $C$. This proves the lemma.
\end{proof}

\begin{lemma}\label{suff_reg_def2}
 Let $\Xi\subseteq G(\R^{\signature})$ be  a compact set. Then there exist constants $\rho_1, \rho_2>0$ depending only on $\Xi$, $\signature$, and $n$ such that the following holds. Let $P_1\subseteq P=MU\subseteq P_2=M_2U_2$ be standard parabolic subgroups, $\gamma\in M(F) $, $\gamma\not\in P_1(F)\cap M(F)$, and suppose that $T\in\aaa_0$ and that $x\in G(\A_F)$ are such that 
$
F^{P_1}(x, T)\sigma_1^2(H_0(x)-T)\neq0
$
and  $x$ can be written as a product 
$
x=umak
$
for suitable $u\in\NNN$, $m\in\MMM$, $a\in A_0^{G, \infty}$, and $k\in\cpt$.  
Then for every $f=f_{\infty}\otimes \One_{\cpt_{\text{fin}}}$ with $f_{\infty}\in C_{\Xi}^{\infty}(G(\R^{\signature}))$, we have
\[
 \int_{U(\A_F)} f(x^{-1} \gamma n x) dn =0
\]
if $d(T)\geq \rho_1+ \rho_2\log D_F>0$.
\end{lemma}
This result is used to prove that the truncated constant function $1$ on $G(F)\backslash G(\A_F)^1$ equals  $F(\cdot, T)$, cf. \cite[Lemma 2.1]{Ar85} which is important for establishing an explicit expression of the distribution $j^T_{\text{unip}}$ which will be introduced below, cf. \cite[Lemma 2.2]{Ar85}. However, we will directly define the distribution $j_{\text{unip}}^T$ as this explicit expression so that we will not explicitly use this lemma here.
The result of this lemma is used in the proof of  \cite[Theorem 7.1]{Ar78} to show that a certain sum over elements in $M(F)\cap \UUU_{G}(F)$ can instead be taken over $P_1(F)\cap M(F)\cap \UUU_G(F)$. This is later used implicitly in \S \ref{section_proofs_of_lemmas} where we apply the methods of \cite{Ar78, Ar85}.

\begin{proof}
 The assertion essentially follows from \cite[pp. 943-944]{Ar78}. 
Write $x=umak$ as in the lemma. We can further write $u=u^*u_*$ with $u_*\in U_0^{M_2}(\A_F)\cap\NNN$ and $u^*$ contained in a suitable compact subset of $U_2(\A_F)$. Now the assumption $F^{P_1}(a, T)\sigma_1^2(H_0(a)-T)=F^{P_1}(x, T)\sigma_1^2(H_0(x)-T)\neq0$ implies by the definition of $F^{P_1}$ and $\sigma_1^2$ that $\alpha(H_0(a))>\alpha(T_1^F)$ for all $\alpha\in \Delta_0^2$. Hence, using Lemma \ref{conjugated_cpt_unipotent} (with respect to the Levi subgroup $M_2$ instead of $G$), the element $a^{-1} u_*ma$ is contained in a compact set of the form 
\[
\{ y\in \Mat_{n\times n}(\A_F)\mid \|y\|_{\A_F}\leq c_1D_F^{c_2}\}\cap G(\A_F)^1
\]
with $c_1, c_2\geq0$ constants depending only on $n$ and $\signature$.  (Here the adelic norm $\|\cdot\|_{\A_F}$ on $\Mat_{n\times n}(\A_F)$ is defined analogously to the norm in \S \ref{section_reduction_theory}.)
Thus to show the lemma it suffices to show that if $a\in A_0^{G, \infty}$ and $T\in\aaa_0$ are such that $F^{P_1}(a, T)\sigma_1^2(H_0(a)-T)\neq0$, and $\gamma$ and $u^*$ are as before, then for $d(T)$ sufficiently large (in a sense we want to specify) 
\[
 \int_{U(\A_F)} f'(a^{-1} (u^*)^{-1} \gamma n u^* a) dn=0
\]
for $f'$ the characteristic function of the compact set
\[
 \Xi'=\{ y\in \Mat_{n\times n}(\A_F)\mid \|y\|_{\A_F}\leq C_1D_F^{C_2}\}\cap G(\A_F)^1
\]
for $C_1, C_2\geq0$ suitable constants depending only on $n$, $\signature$, and $\Xi$.
As $U(\A_F)\ni n\mapsto \gamma^{-1} (u^*)^{-1} \gamma n u^*\in U(\A_F)$ is an isomorphism, after a change of variables we need to show that
\[
  \int_{U(\A_F)} f'(a^{-1} \gamma n a) dn=0.
\]
Now if this integral does not vanish, there exists $n\in U(\A_F)$ with $a^{-1}\gamma na\in\Xi'$, and thus also $a^{-1}\gamma a\in \Xi'\cap M(\A_F)\subseteq \Xi'$. Hence by definition of $\Xi'$, there exist constants $c_3, c_4\geq0$ depending only on $n$, $\signature$, and $\Xi$ such that 
$
|(a^{-1}\gamma a)_{\beta}|\leq c_3+c_4\log D_F
$
for all $\beta\in\Delta_0$ where we use the same notation as defined right before Lemma \ref{lemma_bound_sum_by_bruhat_decomp}.
But then because of our assumptions on $\gamma$, $a$ and the parabolic subgroups, we may apply Lemma \ref{lemma_bound_sum_by_bruhat_decomp} to conclude that there exist constants $c_5, c_6>0$ depending only on $n$, $\signature$, and $\Xi$ such that 
\[
 \alpha_0(H_0(a))\leq c_5+c_6\log D_F
\]
for at least one $\alpha_0\in\Delta_0^{P}\minus\Delta_0^{P_1}\neq\emptyset$. 
As we also assumed $F^{P_1}(a,T)\sigma_1^2(H_0(a)-T)\neq 0$, it follows that
\[
 \alpha_0(T)<\alpha_0(H_0(a)) \leq c_5+c_6 \log D_F.
\]
This can of course not happen if $d(T)>c_5+c_6 \log D_F$. Letting $\rho_1=c_5+1$ and $\rho_2=c_6$ the lemma follows.
\end{proof}

\begin{lemma}\label{suff_reg_def}
There exists $\rho_3>0$ depending only on $n$ (but not on $F$) such for every $T\in \aaa_0$ with
\[
d(T)\geq -\rho_3 \log c_F>0
\]
and every standard parabolic subgroup $P\in\FFF_{\text{std}}$ the equality
\[
\sum_{\substack{P_1\in\FFF_{\text{std}} \\ P_1\subseteq P}}\;\;
 \sum_{\delta\in P_1(F)\backslash P(F)} F^{P_1}(\delta x, T)\tau_{P_1}^{P}(H_0(\delta x) -T)=1
\]
is satisfied for all $x\in G(\A)$.
\end{lemma}
\begin{proof}
 If $T$ is suitably regular, i.e., if $d(T)$ is sufficiently large in Arthur's sense, then the stated assertion holds by \cite[Lemma 6.4]{Ar78} (cf. also \cite[Proposition 3.2.1]{ClLaLa84}). However, from the proof of this lemma it follows that the only crucial point is to check that \cite[Lemma 3.2.2]{ClLaLa84} holds for some $\rho_3$ depending only on $n$. The proof of lemma \cite[Lemma 3.2.2]{ClLaLa84} shows that such a $\rho_3$ can be chosen depending only on the root system of $G$, i.e., on $n$. (Note that since we follow the notation of \cite{Ar85} here, the roles of $T_0$ and $T_1=T_1^F$ are interchanged in the proof of that lemma in \cite{ClLaLa84}, and further $T_0=0$ as we consider $G=\GL_n$.) 
\end{proof}

\begin{definition}\label{def_suit_reg}
Let $\Xi\subseteq G(\R^{\signature})$ be a compact set.
Let $\rho_1, \rho_2, \rho_3>0$ be chosen in dependence on $n$, $\signature$, and $\Xi$ as in Lemma \ref{suff_reg_def2} and Lemma \ref{suff_reg_def}, respectively. 
We call a point $T\in \aaa_0$ suitably regular with respect to $\Xi$ if
\begin{equation}\label{def_suff_reg}
d(T)\geq \tau_F=\tau_F(\Xi, n):= \max\{ \rho_1+\rho_2\log D_F, -\rho_3\log c_F\}>0.
\end{equation}
If $\Xi$ is clear from the context, we may simply say that $T$ is suitably regular.
\end{definition}
Note that the two assertions given in Lemma \ref{suff_reg_def2} and Lemma \ref{suff_reg_def} are precisely the properties with respect to which $T$ has to be suitably regular in the sense of \cite{Ar78,Ar85}. In particular, if $T$ is suitably regular in our sense, we may apply all results and methods from \cite{Ar78,Ar85}.

\subsection{Proof of Proposition \ref{bounding_unip_contr}}
The strategy to prove Proposition \ref{bounding_unip_contr} is essentially the same as proving \cite[Proposition 3]{FiLaMu_limitmult}, but we need to make the dependence on $F$ explicit. For that we need to introduce two more families of distributions.
The first family of distributions 
$
j_{\text{unip}}^{T}: C_{c}^{\infty}(G(\A_F)^1)\longrightarrow\C
$
can by \cite{Ar85} be defined for suitably regular $T\in\aaa_0$ as the absolutely convergent integral
\[
j_{\text{unip}}^{T}(f)
=\int_{G(F)\backslash G(\A_F)^1}F(x, T)\sum_{\gamma\in\UUU_G(F)}f(x^{-1}\gamma x)dx.
\] 
Note that we can define $j_{\text{unip}}^T$ in this way only because Lemma \ref{suff_reg_def2} holds, cf. \cite[Lemma 2.2]{Ar85}.
The  second family $J_{\text{unip}}^{T}(f)$ is defined in \cite{Ar85} in terms of an absolutely convergent integral over a certain integral kernel if $T$ is suitably regular. It is shown in \cite{Ar81} that $J_{\text{unip}}^{T}(f)$ is a polynomial in $T$ of degree at most $n$ for which the coefficients are distributions in the test functions $f$, and as such can be defined at any point $T\in \aaa_0$. 
 The unipotent distribution $J_{\text{unip}}(f)$ is by definition  the value of this polynomial at a special point  $T=T_0$ which is given by \cite[Lemma 1.1]{Ar81}. For $G=\GL_n$, we have $T_0=0$ so that
$J_{\text{unip}}(f)$
equals the constant term of the polynomial $J_{\text{unip}}^T(f)$.
The two families are related by the fact that $j_{\text{unip}}^T(f)$ approximates $J_{\text{unip}}^{T}(f)$ asymptotically in $T$ as shown in \cite[Theorem 3.1]{Ar85} (see also Lemma \ref{lemma_est_int_diff} below).
We first show an upper bound similar to \eqref{inequ_unip_contr} for the polynomial  $\big|J_{\text{unip}}^{T}(f)\big|$, $f=f_{\infty}\otimes\One_{\cpt_{\text{fin}}}$, by finding estimates for
$
\big|j_{\text{unip}}^T(f)\big|
$
and for the difference 
$
\big|J_{\text{unip}}^{T}(f)-j_{\text{unip}}^T(f)\big|
$ 
in Lemmas \ref{lemma_est_trunc_int} and \ref{lemma_est_int_diff} below.
An extrapolation argument for polynomials will give us a bound on the constant term of $J_{\text{unip}}^{T}(f)$, i.e.\  by definition a bound on $\big|J_{\text{unip}}(f)\big|$.
It might be possible to deduce an estimate \eqref{inequ_unip_contr} directly by the methods in \cite{Ho08} without using the auxiliary distributions $J_{\text{unip}}^{T}(f)$ and $j_{\text{unip}}^T(f)$.

The following two lemmas are essentially given by \cite[Lemma 4.1]{Ar85} and \cite[Theorem 3.1]{Ar85}, but again we need to make the dependence on $F$ more explicit. 
 \begin{lemma}\label{lemma_est_trunc_int}
Let $n$, $d$, and $\signature$ be as in Proposition \ref{bounding_unip_contr}.
Then there exists $c=c(n, d)\geq0$ such that for any number field $F$ of signature $\signature$, and all $f_{\infty}\in C_{\Xi_{\signature}}^{\infty}(G(\R^{\signature})^1)$ we have
\[
 \left| j_{\text{unip}}^T(f_{\infty}\otimes\One_{\cpt_{\text{fin}}}) \right| 
\ll_{n, d} D_F^c \vol(F^{\times}\backslash \A_F^1)^n \left\|f_{\infty}\right\|_0 (1+\left\|T\right\|)^{n-1}
\]
for all $T\in \aaa_0$ with  $d(T)\geq \tau_F$ where $\tau_F$ is defined in \eqref{def_suff_reg}. Here $\|\cdot\|:\aaa_0^G\longrightarrow\C$ denotes the norm given by $\|X\|=\sqrt{X_1^2+\ldots+ X_n^2}$ for $X=(X_1, \ldots, X_n)$.
\end{lemma}

\begin{lemma}\label{lemma_est_int_diff}
Let $n$, $d$, and $\signature$ be as in Proposition \ref{bounding_unip_contr}.
Then there exist $c=c(n, d)\geq 0$ and $k=k(n, d)\in\Z_{\geq0}$ such that for any number field $F$ of signature $\signature$, and all $f_{\infty}\in C_{\Xi_{\signature}}^{\infty}(G(\R^{\signature})^1)$ we have
\begin{equation}\label{est_int_diff1}
\left| J_{\text{unip}}^{T}(f_{\infty}\otimes\One_{\cpt_{\text{fin}}})- j_{\text{unip}}^T(f_{\infty}\otimes\One_{\cpt_{\text{fin}}})\right|
\ll_{n, d}
 D_F^{c} \vol(F^{\times}\backslash \A_F^1)^n
\left\|f_{\infty}\right\|_k e^{-d(T)}(1+\|T\|)^{n-1}
\end{equation}
for all $T\in \aaa_0$ with $d(T)\geq \tau_F$.
\end{lemma}

Before proving these two auxiliary results in \S \ref{section_prf_of_first_lemma} and \S \ref{section_prf_of_second_lemma} below, we finish the proof of the proposition from the beginning of this section.
\begin{proof}[Proof of Proposition \ref{bounding_unip_contr}]
By Lemma \ref{lemma_est_trunc_int} and Lemma \ref{lemma_est_int_diff}, there are constants $c_0, k_0\geq0$  such that
\[
\left|J_{\text{unip}}^{ T}(f_{\infty}\otimes \One_{\cpt_{\text{fin}}})\right|
\ll_{n, d}  D_F^{c_0} \vol(F^{\times}\backslash\A_F^1)^n \left\|f_{\infty}\right\|_{k_0}
 (1+\left\|T\right\|)^{n-1}
\]
for all $T\in\aaa_0$ with $d(T)\geq\tau_F$, and all fields $F$ of degree $d$. 
 Now  $J_{\text{unip}}^{T}(f_{\infty}\otimes\One_{\cpt_{\text{fin}}})$ is  a polynomial in $T$ of degree at most $n$. Therefore an extrapolation argument as in \cite[Lemma 5.2]{Ar82a} (cf.\ also \cite[p.122]{Ar05}) shows that the absolute value of the constant term of the polynomial $J_{\text{unip}}^T(f_{\infty}\otimes\One_{\cpt_{\text{fin}}})$ is bounded by
\[
\ll_{n, d} \tau_F^n D_F^{c_0} \vol(F^{\times}\backslash\A_F^1)^n  \left\|f_{\infty}\right\|_{k_0}.
\]
The constant term of the polynomial equals by definition $J_{\text{unip}}(f_{\infty}\otimes\One_{\cpt_{\text{fin}}})$ so that together with the definition of $\tau_F$ in \eqref{def_suff_reg} and the properties of $T_1^F$ given in \eqref{estimates_for_T_1} the desired result follows.
\end{proof}

\section{Proof of Lemmas \ref{lemma_est_trunc_int} and \ref{lemma_est_int_diff}}\label{section_proofs_of_lemmas}
\subsection{Proof of Lemma \ref{lemma_est_trunc_int}}\label{section_prf_of_first_lemma}
\begin{proof}[Proof of Lemma \ref{lemma_est_trunc_int}]
The assertion is essentially given by \cite[Lemma 4.1]{Ar85} (except for the dependence on $F$),  but we also  use arguments from \cite[Lemma 2.2]{La_Schwartz_Space} and \cite[\S 5]{FiLaMu_limitmult}. However, in contrast to the arguments in \cite{Ar85,La_Schwartz_Space,FiLaMu_limitmult}, we have to keep track of the dependence  of the various constants on $F$ the whole time.
We keep the notation introduced earlier in this section and write $f=f_{\infty}\otimes\One_{\cpt_{\text{fin}}}$.

To prove the lemma it will suffice to find an upper bound for
\begin{equation}\label{est_trunc_int1}
\int_{G(F)\backslash G(\A_F)^1} F(x, T) \sum_{\gamma\in\UUU_G(F)}\left| f(x^{-1}\gamma x)\right|dx.
\end{equation}
For that we can of course replace the sum over $\gamma\in\UUU_G(F)$ by a sum over $\gamma\in G(F)$. 
Further we replace the integral over $G(F)\backslash G(\A_F)^1$ by an integral over a Siegel domain, i.e. instead of \eqref{est_trunc_int1}, we consider
\[
\int_{A_G P_0(F)\backslash \SSS_{T_1^F}(\A_F)} F(x,T) \sum_{\gamma\in G(F)} \left| f(x^{-1}\gamma x)\right| dx,
\]
where $\SSS_{T_1^F}(\A_F)=\{g\in G(\A_F)^1\mid \forall \alpha\in \Delta_0:~ \alpha(H_0(g)-T_1^F)\geq 0 \}\subseteq G(\A_F)^1$.
By definition, $F(\cdot, T)$ is right $\cpt$-invariant, and $f$ is $\cpt$-conjugation invariant. Hence we may replace the integral over  $A_G P_0(F)\backslash \SSS_{T_1^F}(\A_F)$ by an integral over 
\[
A_G P_0(F)\backslash \{p\in P_0(\A_F)\mid \forall\alpha\in\Delta_0:~ \alpha(H_0(p)-T_1^F)\geq0\}.
\]
This in turn can be replaced by a  multiple integral over $A_0^{G, \infty}(T_1^F)$, and over compact fundamental domains for $T_0(F)\backslash T_0(\A_F)^1$ and $U_0(F)\backslash U_0(\A_F)$. For an upper bound it suffices to use compact domains containing such fundamental domains so that we can use the compact sets $\MMM\subseteq T_0(\A_F)^1$ and $\NNN\subseteq U_0(\A_F)$ defined in \S \ref{section_fundamental_domains}. Hence it will suffice to find an upper bound for
\begin{align*}
&\int_{\MMM}\int_{\NNN}\int_{A_{0}^{G,\infty}(T_1^F)}F(a,T)\delta_0(a)^{-1}
\sum_{\gamma\in G(F)} \left| f(m^{-1}a^{-1}u^{-1} \gamma uam) \right|
da~ du~ dm\\
&=\int_{\MMM}\int_{\NNN}\int_{A_{0}^{G,\infty}(T_1^F, T)}\delta_0(a)^{-1}
\sum_{\gamma\in G(F)} \left| f(m^{-1}a^{-1}u^{-1} \gamma uam) \right|
da~du~dm.
\end{align*}
We may of course replace $f$ by the product of $\left\|f_{\infty}\right\|_0$ with the characteristic function $\chi_{\signature}:G(\A_F)^1\longrightarrow \C$ of $\Xi_{\signature}\times \cpt_{\text{fin}}\subseteq G(\A_F)^1$. Hence we are left to estimate
\begin{equation}\label{est_trunc_1a}
\left\|f_{\infty}\right\|_0 \int_{\MMM}\int_{\NNN}\int_{A_{0}^{G,\infty}(T_1^F, T)}\delta_0(a)^{-1}
\sum_{\gamma\in G(F)} \chi_{\signature}(m^{-1}a^{-1}u^{-1} \gamma uam) 
da~du~dm.
\end{equation}
Now $m^{-1}a^{-1}u^{-1} \gamma uam\in \Xi_{\signature}\times\cpt_{\text{fin}}$ for some $m\in \MMM$ and $u\in \NNN$  only if $a^{-1}\gamma a\in \tilde{\Xi}\subseteq G(\A_F)^1$, where 
\[
\tilde{\Xi}
=\{nm\xi m^{-1} n^{-1}\mid n\in\NNN^{(1)},~m\in\MMM, ~\xi\in\Xi_{\signature}\times \cpt_{\text{fin}}\}
\]
with $\NNN^{(1)}$ defined in Lemma \ref{conjugated_cpt_unipotent}. Hence \eqref{est_trunc_1a} is bounded by the product of
\[
\left\|f_{\infty}\right\|_0 \vol(\NNN)\vol(\MMM) 
\]
with 
\[
\int_{A_0^{G, \infty}(T_1^F, T)} \delta_0(a)^{-1}
\sum_{\gamma\in G(F)} \chi_{\tilde{\Xi}}(a^{-1}\gamma a) da,
\]
where $\chi_{\tilde{\Xi}}$ denotes the characteristic function of the compact set $\tilde{\Xi}\subseteq G(\A_F)^1$.
For this last sum-integral an upper bound is given in Lemma \ref{estimate_sum_integral_via_lattice_pts} below. Combined with the volume estimates for $\NNN$ and $\MMM$, \eqref{est_trunc_1a} is therefore bounded by
\[
\ll_{n, d} \left\|f_{\infty}\right\|_0\vol(F^{\times}\backslash \A_F^1)^n D_F^a(1+\|T\|)^{n-1}
\]
for a suitable constant $a=a(n,d)\geq0$ depending only on $n$ and $d$.
\end{proof}

To complete the above proof, we still need to show the following estimate.
\begin{lemma}\label{estimate_sum_integral_via_lattice_pts}
 With the notation as in the proof of Lemma \ref{lemma_est_trunc_int}, we have
\begin{equation}\label{int_over_lattice_pts}
\int_{A_0^{G, \infty}(T_1^F, T)} \delta_0(a)^{-1}
\sum_{\gamma\in G(F)} \chi_{\tilde{\Xi}}(a^{-1}\gamma a) da
\ll_{n, d} D_F^c (1+\|T\|)^{n-1}
\end{equation}
for a suitable constant $c=c(n, d)\geq0$ depending only on $n$ and $d$.
\end{lemma}
\begin{proof}
The properties of $T_1^F$ given in \eqref{estimates_for_T_1} imply that
\begin{equation}\label{vol_est_trunc}
\vol(A_{0}^{G, \infty}(T_1^F, T))
\ll_{n, d} (1+\log D_F +\|T\|)^{n-1}.
\end{equation}
Thus it suffices to show that the integrand on the left hand side of \eqref{int_over_lattice_pts}, namely,
\begin{equation}\label{est_trunc_int3}
\delta_0(a)^{-1}\sum_{\gamma\in G(F)}\chi_{\tilde{\Xi}}(a^{-1}\gamma a),
\end{equation}
 can be bounded uniformly in $a\in A_{0}^{G,\infty}(T_1^F)$ with explicit dependence on $F$.

As in the proof of  \cite[Lemma 2.2]{La_Schwartz_Space} we decompose $G(F)$ according to the Bruhat decomposition: For $P\in\FFF_{\text{std}}$ let $G^P(F)$ be the set of elements $\gamma\in G(F)$ for which $P(F)$ is the smallest standard parabolic subgroup in which $\gamma$ is contained. Then 
$G(F)$ equals   
$\bigcup_{P\in\FFF_{\text{std}}}G^P(F)$ (and this union is disjoint).
In particular, any element $\gamma\in G^P(F)$ can be written as $\gamma=\mu \nu$ with $\mu\in M_P(F)$ and $\nu\in U_P(F)$, and we can moreover apply Lemma \ref{lemma_bound_sum_by_bruhat_decomp} to $\gamma\in G^P(F)$ and $P$. 
Hence  \eqref{est_trunc_int3} is bounded by the sum over $P\in\FFF_{\text{std}}$ of
\[
\delta_{P}(a_1)^{-1}\sum_{\mu\in M_P(F)}
\sum_{\nu\in U_P(F)}\chi_{\tilde{\Xi}}^B(a_1^{-1}\mu\nu a_1)
\]
for a suitable $a_1\in A_P^{G, \infty}(T_1^F)$ (depending on $P$)
and
$
\chi_{\tilde{\Xi}}^B(p)=\sup_{b\in B}\delta_0(b)^{-1}\chi_{\tilde{\Xi}}(b^{-1} p b)
$,
$p\in P(\A_F)$,
with $B$ as in Lemma \ref{lemma_bound_sum_by_bruhat_decomp}. (Note that $B$ only depends on $\Xi_{\signature}$ and $n$.) The properties of the sets $B$ given in Lemma \ref{lemma_bound_sum_by_bruhat_decomp} combined with \eqref{estimates_for_T_1}, 
 show that $\sup_{b\in B}\delta_0(b)^{-1}$ is bounded by $c_1 D_F^{c_2}$ with $c_1, c_2\geq0$ suitable constants depending only on $n$ and $d$. 
Replacing $\tilde{\Xi}$ by the compact set 
$
\Xi'=\bigcup_{b\in B} b\tilde{\Xi} b^{-1}\subseteq G(\A_F)
$,
it will therefore suffice to show that for any $P\in\FFF_{\text{std}}$ the sum
\begin{equation}\label{est_trunc_int4}
\delta_{P}(a_1)^{-1}\sum_{\mu\in M_P(F)}
\sum_{\nu\in U_P(F)}\chi_{\Xi'}(a_1^{-1}\mu\nu a_1)
\end{equation}
is bounded  independently of $a_1\in A_P^{G, \infty}(T_1^F)$ for $\chi_{\Xi'}:G(\A_F)\longrightarrow\C$ the characteristic function of the compact set $\Xi'$.

Recall that we fixed a set $\AAA_F$ of representatives for the ideal classes of $F$ with norm bounded by the Minkowski constant.
Note that by definition of $\Xi'$ we have $\chi_{\Xi'}(a^{-1}\gamma a)=0$ for $a\in A_0^{G, \infty}$ and $\gamma\in G(F)$ unless every entry of $\gamma$ (considered as an $n\times n$-matrix) is contained in one of the inverse ideals $\aaa^{-1}$, $\aaa\in\AAA_F$. Let $\Gamma(F)\subseteq \Mat_{n\times n}(F)$ denote the set of $n\times n$-matrices for which every entry is contained in one of these inverse ideals. Note that $\Gamma(F)$ can be identified with the union of the $h_F^{n^2}$ sets $\bbb_1^{-1}\oplus\ldots\oplus\bbb_{n^2}^{-1}$ with $\bbb_i\in\AAA_F$, and that $h_F^{n^2}\ll_d D_F^{c_0}$ for a suitable constant $c_0=c_0(n,d)\geq0$ by Proposition \ref{bound_on_class_number}.

To bound \eqref{est_trunc_int4}, it will therefore suffice to count for each $P\in\FFF_{\text{std}}$ and $a_1\in A_{P}^{G, \infty}(T_1^F)$ the number of points $\mu\in M_P(F)$ and $\nu\in U_P(F)$ with 
\begin{equation}\label{conditions_on_mu_nu0}
\mu\nu\in \Gamma(F),\;\;\text{ and }\;\;
a_1^{-1}\mu \nu a_1\in \Xi'.
\end{equation}
Now if $\mu\in M_P(F)$, $\nu\in U_P(F)$, $a_1\in A_P^{G, \infty}(T_1^F)$, then
\[
a_1^{-1}\mu\nu a_1
=\mu a_1^{-1}\nu a_1
=\mu+ a_1^{-1}(\nu-1)a_1
\]
so that we may replace the  above conditions \eqref{conditions_on_mu_nu0} by 
\begin{equation}\label{conditions_on_mu_nu}
\mu\in\Gamma(F)\cap \Xi'\cap M_P(F),
~~~~~~~~~~~~~~
\text{ and }
~~~~~~~~~~~~~
\nu-1\in \Gamma(F)\cap a_1\Xi' a_1^{-1}\cap U_P(F).
\end{equation}
Note that $\Xi' \cap G(F_{\infty})\subseteq G(F_{\infty})_{\leq c_3 D_F^{c_4}}$ for suitable constants $c_3, c_4\geq 0$ depending only on $n$ and $d$, where 
\[
G(F_{\infty})_{\leq r}:=\{x=(x_{ij})_{i,j}\in G(F_{\infty})\mid \|x_{ij}\|\leq r~\forall i,j\}
\]
with $\|\cdot\|$ denoting the norm on $F_{\infty}$ from \S \ref{section_inner_products}.

We first bound the number of $\mu$ satisfying \eqref{conditions_on_mu_nu}. 
For that it will suffice to bound for each $\aaa\in\AAA_F$ the number of $x\in F\subseteq F_{\infty}$ with
$
x\in \aaa^{-1}
$ 
and
$
\|x\|\leq c_3 D_F^{c_4}
$.
But this was already done in Lemma \ref{bound_for_sum_over_ideal_points} so that the number of $\mu$ satisfying \eqref{conditions_on_mu_nu} is bounded from above by
$
\ll_{n,d} D_F^{c_5}
$
for a suitable $c_5=c_5(n, d)\geq0$, where we used the upper bound for the class number $h_F$ from Proposition \ref{bound_on_class_number} again.

We now bound the number of $\nu$ satisfying \eqref{conditions_on_mu_nu}.
Let $\alpha$ be a positive root in $\Sigma^+(P, A_P^{\infty})$. Then the norm of the $\alpha$-coordinate $u_{\alpha}$ of 
$
u\in \left(a_1 G(F_{\infty})_{\leq c_3 D_F^{c_4}} a_1^{-1}\right)\cap U_P(F_{\infty})
$
 is bounded by
$
\|u_{\alpha}\|\leq e^{\alpha(H_0(a_1))/d}c_3 D_F^{c_4}
$.
(Note that since we identify elements $t\in\R^{\times}$ with $(t^{1/d},\ldots,t^{1/d},1\ldots)\in\A_F^{\times}$, we need to take the $d$-th root of $e^{\alpha(H_0(a_1))}$.)
We can estimate the number of contributing $\alpha$-coordinates for each $\alpha$ separately. Using Lemma \ref{bound_for_sum_over_ideal_points} and Proposition \ref{bound_on_class_number} again, we see that the number of $\nu$ satisfying \eqref{conditions_on_mu_nu} is bounded from above by
\[
\ll_{n, d} D_F^{c_6}\prod_{\alpha\in\Sigma^+(P, A_P^{\infty})} \left(1+e^{\alpha(H_0(a_1))} D_F^{c_7}\right)
\]
for suitable $c_6=c_6(n, d), c_7=c_7(n, d)\geq0$.
Since $a_1\in A_P^{G, \infty}(T_1^F)$, we have 
$
1\ll_{n, d} D_F^{\alpha(\rho_0^{\vee})} e^{\alpha(H_0(a_1))}
$
so that the number of contributing $\nu$ is bounded by
\[
\ll_{n, d} D_F^{c_6}\prod_{\alpha\in\Sigma^+(P, A_P^{\infty})} \left(e^{\alpha(H_0(a_1))}(D_F^{\alpha(\rho_0^{\vee})}+ D_F^{c_7})\right)
\ll_{n, d}\delta_P(a_1) D_F^{c_8}
\]
for some constant $c_8=c_8(n, d)\geq0$ depending only on $n$ and $d$.

Combining the estimate on the number of contributing $\mu$ and $\nu$ with the bound on the class number, the sum \eqref{est_trunc_int4} is bounded from above by 
$
\ll_{n, d} D_F^{c_9}
$
for a constant $c_9=c_9(n, d)\geq 0$. 
Together with the volume estimate \eqref{vol_est_trunc} this implies the assertion of the lemma.
\end{proof}

\subsection{Proof of Lemma \ref{lemma_est_int_diff}}\label{section_prf_of_second_lemma}

\begin{proof}[Proof of Lemma \ref{lemma_est_int_diff}]
We keep the notation from \S \ref{section_prf_of_first_lemma}. In particular, we write $f=f_{\infty}\otimes\One_{\cpt_{\text{fin}}}$.
Further, if $P_1\subseteq P_2$ are standard parabolic subgroups, we write $P_i=M_iU_i$ for their Levi decomposition,  $\ppp_i=\mmm_i+\uuu_i$ for the corresponding decomposition of the Lie algebras, and put $\uuu_{P_1}^{P_2}= \uuu_1^2= \uuu_1\cap \mmm_2$. 

Since Lemma \ref{suff_reg_def} holds, we may use \cite[(3.1)]{Ar85} to see that the left hand side of \eqref{est_int_diff1} can be bounded by a sum over pairs of standard parabolic subgroups $(P_1, P_2)$, $P_1\subsetneq P_2$, of 
\begin{equation}\label{est_int_diff2}
 \int_{P_1(F)\backslash G(\A_F)^1} F^{P_1}(x, T)\sigma_{1}^{2}(H_{P_1}(x)-T)
 \sum_{\gamma\in\UUU_{M_{1}}(F)}
\sum_{\zeta\in\uuu_1^2(F)'}  \left|\Phi_{\gamma}(x, \zeta)\right|dx
\end{equation}
where for $Y\in \uuu_1^2(\A_F)$, $x\in G(\A_F)^1$, and $m\in M_1(\A_F)$ the function
\[
\Phi_{m}(x, Y)
=\int_{\uuu_1(\A_F)} f(x^{-1}m\exp(X) x) \psi(\langle X, Y\rangle)dX
\]
is the partial Fourier transform of $f$ along $\uuu_1$.
Further, $\uuu_1^2(F)'$ denotes the set of all elements in $\uuu_1^2(F)$ which do not belong to any of the spaces 
$\uuu_1^P(F)$, $P\in\FFF_{\text{std}}$ with $P_1\subseteq P\subsetneq P_2$,
 and 
$
\sigma_{1}^{2}:\aaa_0\longrightarrow\C
$
is the function defined in \S \ref{bounding_unip_contr}.
The functions $f(x^{-1}m \exp(\cdot)x)$ and $\Phi_m(x, \cdot)$ are Schwartz-Bruhat functions on $\uuu_1(\A_F)$ varying smoothly with $x$ and $m$.
We can decompose the domain of integration as
\[
P_1(F)\backslash G(\A_F)^1
=A_{P_1}^{G, \infty} \times U_1(F)\backslash U_1(\A_F)\times M_1(F)\backslash M_1(\A_F)^1 \times\cpt.
\]
Because of the $\cpt$-conjugation invariance of $f$, \eqref{est_int_diff2} equals
\begin{multline*}
\int_{A_{P_1}^{G, \infty}} \int_{M_1(F)\backslash M_1(\A_F)^1}\int_{U_1(F)\backslash U_1(\A_F)}
 F^{P_1}(m, T)\sigma_{1}^{2}(H_{P_1}(a_1)-T)\cdot\\
\sum_{\gamma\in\UUU_{M_{1}}(F)}
\sum_{\zeta\in\uuu_1^2(F)'}  \left|\Phi_{\gamma}(a_1mu, \zeta)\right|du~dm~da_1.
\end{multline*}
We further decompose $M_1(\A_F)^1$ as
\[
 M_1(\A_F)^1
=M_1(F) \times T_0(\A_F)^1 \times U_0^{M_1}(\A_F) \times A_0^{M_1,\infty}(T_1^F) \times \cpt^{M_1}
\]
so that by using the definition of the truncation function $F^{P_1}(am,T)$, the above is bounded by
\[
\int_{A_{P_1}^{G, \infty}(\sigma_1^2(\cdot-T))}\int_{A_{0}^{M_1,\infty}(T_1^F, T)} \int_{\MMM}\int_{\NNN} 
 \sum_{\gamma\in\UUU_{M_{1}}(F)}
\sum_{\zeta\in\uuu_1^2(F)'}  \left|\Phi_{\gamma}(a_1a_2 t n, \zeta)\right|dn~dt~da_2~da_1,
\]
where $A_{P_1}^{G, \infty}(\sigma_1^2(\cdot-T))$ denotes the set of all $a_1\in A_{P_1}^{G, \infty}$ with $\sigma_1^2(H_{P_1}(a_1)-T)\neq0$.
By \eqref{non-vanishing_sigma} every $a_1\in A_{P_1}^{G, \infty}(\sigma_1^2(\cdot-T))$ in particular satisfies $\alpha(H_{P_1}(a_1))>0$ for all $\alpha\in\Delta_1^2$.
Then
\begin{align*}
\Phi_{\gamma}(a_1a_2tn, \zeta)
& =\int_{\uuu_1(\A_F)} f((tn)^{-1}(a_2^{-1}\gamma a_2) (a_1a_2)^{-1}\exp(X)(a_1a_2)(tn)) \psi(\langle X, \zeta\rangle)dX\\
&=\delta_{P_1}(a_1a_2)\int_{\uuu_1(\A_F)} f((tn)^{-1}(a_2^{-1}\gamma a_2) \exp(X) (tn)) \psi(\langle X, \Ad(a_1a_2) \zeta\rangle)dX\\
& =\delta_{P_1}(a_1a_2) \Phi_{a_2^{-1}\gamma a_2}(tn, \Ad(a_1a_2)\zeta) . 
\end{align*}
Note that for $\Phi_{a_2^{-1}\gamma a_2}( t n, \Ad(a_1a_2)\zeta)$ to vanish not identically in $a_1, a_2, t, n$, we must have $\gamma\in\Gamma(F)\cap \tilde{\Xi}$, where $\tilde{\Xi}$ and $\Gamma$ are defined as  in the proof of Lemma \ref{lemma_est_trunc_int}.

By applying suitable differential operators to $\Phi_{a_2^{-1}\gamma a_2}(tn, \cdot)$, we can bound $\left|\Phi_{\gamma}(a_1a_2 tn, \zeta)\right|$ for every $k\in\Z_{\geq0}$ by
\[
C_k\delta_{P_1}(a_1a_2) \|\Ad(a_1a_2)\zeta\|^{-2k} 
\int_{\uuu_1(\A_F)} \sum_{Y\in\BBB_{2k}}\left|Y*f((tn)^{-1}(a_2^{-1}\gamma a_2) \exp(X) (tn))\right| dX,
\]
where $C_k>0$ is a suitable constant depending on the fixed basis $\BBB_{2k}$ of $\UUU(\ggG_{\signature}(\C))_{\leq 2k}$.

The support of the function $Y*f$ is again contained in the compact set $\Xi_{\signature}\times \cpt_{\text{fin}}$, and hence $(Y*f)((tn)^{-1}(a_2^{-1}\gamma a_2) \exp(X) (tn))$ vanishes identically  unless 
\[
a_2^{-1}\gamma a_2 \exp(X)\in \{(tn) \xi (tn)^{-1}\mid n\in\NNN, ~ t\in \MMM,~ \xi\in\tilde{\Xi}\} . 
\]
This last condition implies that 
\[
\exp( X)\in \{\xi_1^{-1} (tn)^{-1} \xi_2(tn)\mid n\in\NNN, ~ t\in \MMM,~ \xi_1, \xi_2\in\tilde{\Xi}\}
\]
and that $\zeta\in \uuu_1^{2\prime}(F)\cap \Gamma_2(F)$ for $\Gamma_2(F)$ defined similarly as $\Gamma(F)$, but with $\aaa^{-1}$, $\aaa\in\AAA_F$, replaced by $\aaa^{-2}$, $\aaa\in \AAA_F$.
The volume of this last compact set is bounded from above by $c_1D_F^{c_2}$ for suitable $c_1, c_2>0$ depending only on $n$ and $d$ so that
\[
\left|\Phi_{\gamma}(a_1a_2 tn, \zeta)\right|
\leq  C_kc_1 D_F^{c_2}\delta_{P_1}(a_1a_2) \|\Ad(a_1a_2)\zeta\|^{-2k}  \|f_{\infty}\|_{2k} \chi_{\tilde{\Xi}^{M_1}}(a_2^{-1}\gamma a_2).
\]
where $\chi_{\tilde{\Xi}^{M_1}}:M_1(\A_F)^1\longrightarrow\C$ is the characteristic function of $\tilde{\Xi}^{M_1}$ which is the set defined analogous to $\tilde{\Xi}$, but with respect to $M_1$ instead of $G$.

Hence \eqref{est_int_diff2} is bounded from above by the product of 
$
  C_k' D_F^{c_2'}  \|f_{\infty}\|_{2k}
$
with
\begin{multline*}
 \int_{A_{P_1}^{G, \infty}(\sigma_1^2(\cdot-T))}\int_{A_{0}^{M_1,\infty}(T_1^F, T)}
\delta_0(a_1a_2)^{-1}\delta_{P_1}(a_1a_2) 
\bigg(\sum_{\zeta\in \uuu_1^2(F)'\cap\Gamma_2(F)} \|\Ad(a_1a_2)\zeta\|^{-2k} \bigg)\cdot\\
 \bigg(\sum_{\gamma\in \Gamma(F)\cap \tilde{\Xi}^{M_1}} \chi_{\tilde{\Xi}^{M_1}}(a_2^{-1}\gamma a_2)\bigg) da_2~da_1
\end{multline*}
where $C_k'>0$ and $c_2'>0$ are constants depending only on $n, d$ and $k$.
Now $\delta_0(a_1a_2)^{-1}\delta_{P_1}(a_1a_2) =\delta_0^1(a_2)^{-1}$  so that by the proof of  Lemma \ref{estimate_sum_integral_via_lattice_pts} we get for all $a_1,a_2$ that
\[
\delta_0(a_1a_2)^{-1}\delta_{P_1}(a_1a_2) \sum_{\gamma\in \Gamma(F)\cap \tilde{\Xi}^{M_1}} \chi_{\tilde{\Xi}^{M_1}}(a_2^{-1}\gamma a_2)
\ll_{n, d} D_F^{c} 
\]
for some $c>0$  depending only on $n$ and $d$. Hence we are left to estimate
\begin{equation}\label{est_int_diff5}
 \int_{A_{P_1}^{G, \infty}(\sigma_1^2(\cdot-T))}\int_{A_{0}^{M_1,\infty}(T_1^F, T)}
\sum_{\zeta\in \uuu_1^2(F)'\cap \Gamma_2(F) } \|\Ad(a_1a_2)\zeta\|^{-2k} da_2~da_1.
\end{equation}
By \cite[pp. 1247-1248]{Ar85}, the integrand can be bounded by a sum over subsets $R\subseteq \Delta_0^2$ of
\[
\bigg(\sum_{\zeta\in \uuu_1^2(F)'\cap \Gamma_2(F)} \|\zeta\|^{-2k/|R|}\bigg) \prod_{\alpha\in R} e^{-\frac{2k}{|R|} \alpha(H_0(a_1a_2))/d}
\]
with $R$ having the following property: For any $\alpha\in\Delta_0^2\minus\Delta_0^1$ there exists $\beta\in R$ with positive $\alpha$-coordinate.

Using the bound for the  class number $h_F$ from  Proposition \ref{bound_on_class_number} again, Lemma \ref{bound_for_sum_over_ideal_points} yields
\[
\sum_{\zeta\in \uuu_1^2(F)'\cap \Gamma_2(F)} \|\zeta\|^{-2k/|R|}
\ll_{d,n} D_F^{2kd/|R|+c_3}
\]
for some $c_3=c_3(n, d)\geq0$ provided that $2k/|R|\geq \dim_{\R}\uuu_1^2(\R^{\signature})+2$.
Now the sum over the subsets $R$ as above of the integral over $a_1\in A_{P_1}^{G, \infty}(\sigma_1^2(\cdot-T))$ of 
$
\prod_{\alpha\in R} e^{-\frac{2k}{|R|} \alpha(H_0(a_1a_2))/d}
$
is bounded by the discussion in \cite[pp. 1248-1249]{Ar85} for all  $a_2\in A_0^{M_1, \infty}(T_1^F, T)$  by
\[
\bigg(\prod_{\delta\in\Delta_0^1}e^{-k_{\delta} \delta(T_1^F)}\bigg)
\prod_{\beta\in \Delta_1^2}\left(e^{-\beta(T)} \int_0^{\infty} p_{\beta}(t_{\beta}) e^{-t_{\beta}} dt_{\beta}\right)
\]
where $p_{\beta}$ are rational polynomials which only depend on the root system of $G$, i.e. on $n$, and $k_{\delta}$ are certain non-negative integers which can be bounded polynomially in $n$ and $d$. By the properties of $T_1^F$ given in \eqref{estimates_for_T_1}, this expression is therefore bounded for some $c_4=c_4(n, d)\geq0$ by 
\[
\ll_{n,d} D_F^{c_4}\prod_{\beta\in \Delta_1^2}\left(e^{-\beta(T)}\int_{0}^{\infty} p_{\beta}(t_{\beta}) e^{-t_{\beta}} dt_{\beta}\right).
\]
Hence \eqref{est_int_diff5} is bounded  from above for suitable $c_5=c_5(n, d), c_6=c_6(n, d)\geq0$ by
\begin{align*}
&\ll_{n, d}D_F^{c_5} \vol(A_{0}^{M_1,\infty}(T_1^F, T))\prod_{\beta\in \Delta_1^2}\left(e^{-\beta(T)}\int_{0}^{\infty} p_{\beta}(t_{\beta}) e^{-t_{\beta}} dt_{\beta}\right)\\
&\ll_{n, d} D_F^{c_6} (1+\|T\|)^{\dim \aaa_0^M} e^{- d(T)}
\end{align*}
(recall that $d(T)=\min_{\alpha\in\Delta_0}\alpha(T)$). Since $\dim\aaa_0^M\leq n-1$, the assertion follows.
\end{proof}

\section{Examples: Coefficients for ${\rm GL}_2$ and ${\rm GL}_3$}\label{section_example_gl2}
In this section we give exact formulas for the coefficients for $G=\GL_2$ and $G=\GL_3$, and  verify the first part \eqref{conj1} of Conjecture \ref{conj} in both cases. The second part of this conjecture is equivalent to the first one if the lower bound of the Brauer-Siegel Theorem holds for $F$ (for example, if $F$ is Galois over $\Q$). However, we shall see in the examples that the term $a^{\Mbf_{L, \VVV}}(\One^{\Mbf_{L,\VVV}}, S)$ occurring in the statement of the conjecture, also naturally occurs in the exact formula for the coefficients in all the computed cases.

\subsection{Coefficients for ${\rm GL}_2$}
For $G=\GL_2$ the unipotent distribution can be written as (see \cite[\S 16]{JaLa70}, for example)
\begin{align*}
J_{\text{unip}}(f)& =
\vol(\GL_2(F)\backslash \GL_2(\A_F)^1) f(1)\\
&+\vol(T_0(F)\backslash T_0(\A_F)^1)\frac{\lambda_0^{S}}{\lambda_{-1}^S}
\int_{\cpt_{S}}\int_{F_{S}} f(k^{-1} \left(\begin{smallmatrix} 1&a\\0&1\end{smallmatrix}\right) k) da~ dk\\
&+\vol(T_0(F)\backslash T_0(\A_F)^1) \int_{\cpt_{S}}\int_{F_{S}} f(k^{-1}  \left(\begin{smallmatrix} 1&a\\0&1\end{smallmatrix}\right) k) \log|a|_{S} da~ dk
\end{align*}
with 
$ \zeta^{S}_F(s)=\lambda_{-1}^{S}(s-1)^{-1}+\lambda_0^{S} + \lambda_1^{S}(s-1)+\ldots$
  the Laurent expansion around $s=1$ of the partial Dedekind zeta function given by
$\zeta^{S}_{F}(s)=\prod_{v\not\in S}(1-q_v^{-s})^{-1}$
if $\Re s>1$. 
In particular, we have $\LLL^{\GL_2}=\{T_0, \GL_2\}$, $\Ufrak^{T_0}=\{\One^{T_0}\}$, and $\Ufrak^{\GL_2}$ consists of the regular class $\VVV_{\text{reg}}$ and the trivial class $\One^{\GL_2}$.
The regular class in $\GL_2$ is the only non-trivial case. It satisfies $\Mbf_{\GL_2,\VVV_{\text{reg}}}= T_0$ and
\[
a^{\GL_2}(\VVV_{\text{reg}},S)
= \vol(T_0(F)\backslash T_0(\A_F)^1) \frac{\lambda^{S}_{0}}{\lambda_{-1}^{S}}
=a^{T_0}(\One^{T_0}, S)\frac{\lambda^{S}_{0}}{\lambda_{-1}^{S}}.
\]
Now
\begin{equation}\label{local_kronecker}
\frac{\lambda_{0}^{S}}{\lambda_{-1}^{S}}
=\frac{\lambda_0^F}{\lambda_{-1}^F}-\sum_{v\in S_{\text{fin}}}\frac{\zeta_{F, v}^{\prime}(1)}{\zeta_{F, v}(1)}
=\frac{\lambda_0^F}{\lambda_{-1}^F}+\sum_{v\in S_{\text{fin}}}\left|\frac{\zeta_{F, v}^{\prime}(1)}{\zeta_{F, v}(1)}\right|
\end{equation}
so that
\[
\left|a^G(\VVV_{\text{reg}}, S)\right|
=\lambda_{-1}^F\lambda_0^F+ (\lambda_{-1}^F)^2\sum_{v\in S_{\text{fin}}}\left|\frac{\zeta_{F, v}^{\prime}(1)}{\zeta_{F, v}(1)}\right|.
\]
By Proposition \ref{brauer_siegel} we can bound the coefficients $\lambda_{-1}^F$ and $\lambda_0^F$ by $\ll_d D_F^{\eps}$ for every $\eps>0$ so that in this case the first part \eqref{conj1} of Conjecture \ref{conj} holds.

\subsection{Coefficients for ${\rm GL}_3$}
Up to conjugation, there are three Levi subgroups in $\LLL^{\GL_3}$: $T_0$, $M_1$, and $\GL_3$, where 
$M_1=\GL_2\times \GL_1\hookrightarrow\GL_3$.
There are three different orbits in $\Ufrak^{\GL_3}$: the trivial class $\One^{\GL_3}$, the subregular  class $\VVV_{\text{s-r}}$, and the regular class $\VVV_{\text{reg}}$.
 We summarise the different unipotent conjugacy classes  in the Levi subgroups and some further information about them  in the following table.

\begin{center}
\begin{tabular}{lcccc}
$L$				&$\VVV\in\Ufrak^L$			& $\III_L^{\GL_3} \VVV\in\Ufrak^{\GL_3}$	&$\Mbf_{L,\VVV}$\\	
\noalign{\smallskip}\hline\noalign{\smallskip}
$T_0$				&$\One^{T_0}$				&$\VVV_{\text{reg}}$				&$T_0$	\\ 
\noalign{\smallskip}
$M_1$				&$\One^{M_1}$				&$\VVV_{\text{s-r}}$				&$M_1$	\\
\noalign{\smallskip}
$M_1$				&$\VVV^{M_1}_{\text{reg}}$		&$\VVV_{\text{reg}}$				&$T_0$\\
\noalign{\smallskip}
$\GL_3$				&$\One^{\GL_3}$				&$\One^{\GL_3}$					&$\GL_3$  \\
\noalign{\smallskip}	
$\GL_3$				&$\VVV_{\text{s-r}}$			&$\VVV_{\text{s-r}}$				&$M_1$\\
\noalign{\smallskip}	
$\GL_3$				&$\VVV_{\text{reg}}$			&$\VVV_{\text{reg}}$				&$T_0$\\
\noalign{\smallskip}\hline
\end{tabular}
\end{center}

The first, second and fourth case are trivial so that we are left with the remaining cases $\VVV_{\text{reg}}\subseteq \GL_3$, $\VVV_{\text{reg}}^{M_1}\subseteq M_1$, and $\VVV_{\text{s-r}}\subseteq \GL_3$.
For these we get from \cite[Lemma 4]{Fl82} and \cite[Lemma 9]{diss}  that
\begin{align*}
a^{\GL_3}(\VVV_{\text{reg}}, S)	
&=\vol(T_0(F)\backslash T_0(\A_F)^1)\bigg(\bigg(\frac{\lambda_0^{S}}{\lambda_{-1}^{S}}\bigg)^2+\frac{\lambda_1^{S}}{\lambda_{-1}^{S}}\bigg)
=a^{T_0}(\One^{T_0}, S)\bigg(\bigg(\frac{\lambda_0^{S}}{\lambda_{-1}^{S}}\bigg)^2+\frac{\lambda_1^{S}}{\lambda_{-1}^{S}}\bigg),\\
a^{M_1}(\VVV_{\text{reg}}^{M_1},S)
&=\vol(T_0(F)\backslash T_0(\A_F)^1)\frac{\lambda_{0}^{S}}{\lambda_{-1}^{S}}
=a^{T_0}(\One^{T_0}, S)\frac{\lambda_{0}^{S}}{\lambda_{-1}^{S}},\;\;\text{ and }\\
a^{\GL_3}(\VVV_{\text{s-r}}, S)	
&=\vol(M_1(F)\backslash M_1(\A_F)^1)\frac{\zeta^{S\prime}_F(2)}{\zeta_F^{S}(2)}
=a^{M_1}(\One^{M_1}, S)\frac{\zeta^{S\prime}_F(2)}{\zeta_F^{S}(2)}.
\end{align*}
The second coefficient is already covered by the considerations for $\GL_2$. 
For the coefficient associated with the subregular class in $\GL_3$, we get
\[
\left|\frac{a^{\GL_3}(\VVV_{\text{s-r}}, S)}{\vol(M_1(F)\backslash M_1(\A_F)^1)}\right|
=\left|\frac{\zeta^{S\prime}_F(2)}{\zeta_F^{S}(2)}\right|
\ll_{d} 1
\]
so that for this coefficient both parts of Conjecture \ref{conj} hold without condition on the field.
For the coefficient associated with the regular conjugacy class in $\GL_3$ the first part of Conjecture \ref{conj} follows from a similar computation as \eqref{local_kronecker}  by using the upper bounds for the coefficients $\lambda_i^F$, $i\in\{-1,0, 1\}$, from Proposition \ref{brauer_siegel} again.

\section{Coefficients for arbitrary elements; proof of Corollary \ref{cor_bounds_gen_coeff}}\label{section_arb_coeff}
In this section we prove Corollary \ref{cor_bounds_gen_coeff} giving an  upper bound for the general coefficients $a^M(\gamma, S)$ with $\gamma\in M(F)$ not necessarily unipotent. These coefficients are defined in \cite[(8.1)]{Ar86} in terms of coefficients for unipotent elements.
To recall the definition of $a^M(\gamma, S)$ in the case $G=\GL_n$, let $\gamma\in M(F)$ and write $\gamma=\sigma \nu=\nu\sigma$ for the Jordan decomposition of $\gamma$ with $\sigma=\gamma_s\in M(F)$ semisimple and $ \nu\in \UUU_{M_{\sigma}}(F)$ unipotent. 
Here  $M_{\sigma}\subseteq M$ is the centraliser of $\sigma$ in $M$, which is connected, since $M$ is isomorphic to a product of general linear groups.
The definition of the general coefficient in  \cite[(8.1)]{Ar86} simplifies for $\GL_n$ to
\[
a^M(\gamma, S)
=
      \begin{cases}
	a^{M_{\sigma}}(\nu, S)				&\text{if } \sigma\text{ is elliptic in }M(F),\text{ i.e., }A_M=A_{M_{\sigma}},\\
	0						&\text{otherwise}.
	\end{cases}
\]
We may therefore assume that $M=\GL_{n_1}\times\ldots\times \GL_{n_{r+1}}$ and that $\sigma$ is elliptic in $M$. 
Accordingly, 
$\sigma=\diag(\sigma_1, \ldots, \sigma_{r+1})$
for $\sigma_i\in\GL_{n_i}(F)$ elliptic in $\GL_{n_i}(F)$,
and  $\nu =\diag(u_1, \ldots, u_{r+1})$ with $u_i\in \UUU_{\GL_{n_i, \sigma_i}}(F)$
so that
$
\gamma=\diag(\gamma_1, \ldots, \gamma_{r+1})
$  
with
$\gamma_i=\sigma_i u_i=u_i\sigma_i$. 
Similar as in the unipotent case we have
\begin{equation}\label{factorising_coeff_general}
a^{G}(\gamma, S)
=a^{M_{\sigma}}(\nu, S)
=\prod_{i=1}^{r+1} a^{\GL_{n_i, \sigma_i}}(u_i, S)
=\prod_{i=1}^{r+1} a^{\GL_{n_i}}(\gamma_i, S).
\end{equation}
Now there are integers $m_i|n_i$ and regular elliptic elements $\tau_i\in \GL_{m_i}(F)$  such that 
$\sigma_{i}$
is conjugate in $\GL_{m_i}(F)$ to
$\diag(\tau_{i}, \ldots, \tau_i)$,
and we may assume $\sigma_i$ to be of this form.
Let $M_1\subseteq M$ be the smallest $F$-Levi subgroup in which $\sigma$ is contained so that $\sigma$ is regular elliptic in $M_1(F)$.
For any $i\in \{1, \ldots, r+1\}$ restriction of scalars gives an isomorphism 
$
\psi_i:\GL_{k_i}(E_i)\longrightarrow \GL_{n_i, \sigma_i}(F)
$
where $E_i$ is a suitable extension of $F$ of degree $[E_i:F]=m_i\leq n$, $[E_i:\Q]\leq n[F:\Q]$, with absolute discriminant $D_{E_i}$ (over $\Q$)
Let $S_{E_i}$ be the set of places of $E_i$ lying above $S$.
Then 
$\psi_i^{-1}(\UUU_{\GL_{n_i, \sigma_i}}(F))=\UUU_{\GL_{k_i}}(E_i)$, 
$\psi_i^{-1}(M_{\sigma}(F)) = \GL_{k_1}(E_1)\times\ldots\times\GL_{k_{r+1}}(E_{r+1})$,
and
\begin{equation}\label{reduction_to_unip_case}
a^{\GL_{m_i}}(\gamma_i, S )
= a^{\GL_{k_i}}(\psi_i^{-1}(u_i), S_{E_i})
\end{equation}
where the right hand side is now computed with respect to the ground field $E_i$ instead of $F$.
We assume from now on that the eigenvalues of $\gamma_s$ are an algebraic integer over $F$, i.e., their characteristic polynomials have integral coefficients. Note that if $\gamma\in M(F)$ is a general element, then $a^M(\alpha \gamma, S)=a^M(\gamma,S)$ for every $\alpha\in F^{\times}$ so that this assumption is no real restriction.
Then, since the eigenvalues  of $\gamma$ are algebraic integers, we get
\[
 D_{E_1}^{k_1}\cdot\ldots\cdot D_{E_{r+1}}^{k_{r+1}}
\leq |\discr^{M_1}\sigma|_{\infty},
\]
where $\discr^{M_1}\sigma$ denotes the discriminant of $\sigma$ as an element of $M_1(F)$, and $|\cdot|_{\infty}$ the product of the norms at all archimedean places.
Moreover, if $v$ is a non-archimedean place of $E_i$ above a place $v_F$ of $F$, then
\[
\left|\frac{\zeta_{E_i, v}^{(r)}(1)}{\zeta_{E_i, v}(1)}\right|
\leq
2\left|\frac{\zeta_{F,v_F}^{(r)}(1)}{\zeta_{F,v_F}(1)}\right|
\]
for any $r\in\Z_{\geq 0}$. 
Combining  \eqref{factorising_coeff_general}, \eqref{reduction_to_unip_case}, and Theorem \ref{estimate_for_coeff}, the asserted upper bound for $a^M(\gamma, S)$ in Corollary \ref{cor_bounds_gen_coeff} follows by noting that $\sum_{i=1}^{r+1} (k_i-1) =\dim\aaa_{M_{1,\gamma_s}}^{M_{\gamma_s}}$.

Having proven Corollary \ref{cor_bounds_gen_coeff}, we finally state the analogue of Conjecture \ref{conj} for arbitrary coefficients.
\begin{conjecture}\label{conj_arb_coeff}
 For any $\kappa>0$ we have
\begin{align*}
 \left|a^M(\gamma, S)\right|
& \ll_{n, d, \kappa} |\discr^{M_1}(\gamma_s)|_{\infty}^{\kappa}
\sum_{\substack{s_v\in\Z_{\geq0},v\in S_{\text{fin}}: \\ \sum s_v\leq \eta}}\; 
\prod_{v\in S_{\text{fin}}}\bigg|\frac{\zeta_{F, v}^{(s_v)}(1)}{\zeta_{F,v}(1)}\bigg|, \;\;\text{ and }\\
 \left|\frac{a^M(\gamma, S)}{a^{L^{\gamma}}(\gamma_s, S)}\right|
& \ll_{n, d, \kappa} |\discr^{M_1}(\gamma_s)|_{\infty}^{\kappa}
\sum_{\substack{s_v\in\Z_{\geq0},v\in S_{\text{fin}}: \\ \sum s_v\leq\eta}}\;
\prod_{v\in S_{\text{fin}}}\bigg|\frac{\zeta_{F, v}^{(s_v)}(1)}{\zeta_{F,v}(1)}\bigg|
\end{align*}
where $L^{\gamma}\in\LLL^M$ is the Levi subgroup in $M$ such that $\gamma_s\in L^{\gamma}(F)$ and such that the conjugacy class of $\gamma$ in $M(F)$ is induced from the $L^{\gamma}(F)$-conjugacy class of $\gamma_s$ in $L^{\gamma}(F)$.
\end{conjecture}

\bibliographystyle{abbrv}
\bibliography{coeff_est_bib}
\end{document}